\documentclass[11pt]{amsart}

\usepackage{
  amsmath,
  amsfonts,
  amssymb,
  amsthm,
  amscd,
  gensymb,      
  graphicx,
  comment,      
  etoolbox,     
  mathtools,    
  enumitem,
  mathdots,     
  stmaryrd,     
  txfonts,      
  booktabs,     
  stackrel,     
  xifthen,
}
\usepackage[all]{xy}
\usepackage[usenames,dvipsnames]{xcolor}

\makeatletter
\@namedef{subjclassname@2020}{%
  \textup{2020} Mathematics Subject Classification}
\makeatother

\usepackage{bbm}                     

\DeclareFontFamily{OT1}{pzc}{}
\DeclareFontShape{OT1}{pzc}{m}{it}{<-> s * [1.10] pzcmi7t}{}
\DeclareMathAlphabet{\mathpzc}{OT1}{pzc}{m}{it}

\usepackage[colorlinks=true, linkcolor=blue, citecolor=blue, urlcolor=blue, breaklinks=true]{hyperref}

\usepackage{tikz}
\usetikzlibrary{decorations.markings}
\usetikzlibrary{cd}
\usetikzlibrary{patterns}

\tikzset{anchorbase/.style={baseline={([yshift=-0.5ex]current bounding box.center)}}}
\tikzset{wipe/.style={white,line width=4pt}}

\newcommand{\braidto}{to[out=up,in=down]}

\tikzset{->-/.style={decoration={

  markings,
  mark=at position #1 with {\arrow{>}}},postaction={decorate}}}
\tikzset{-<-/.style={decoration={
  markings,
  mark=at position #1 with {\arrow{<}}},postaction={decorate}}}

\newcommand\dotlabel[1]{$\scriptstyle{#1}$}
\newcommand\shiftlabel[1]{$\color{purple} \scriptstyle{#1}$}
\newcommand\shiftline[3]{\draw[purple] (#1) to (#2) node[anchor=west] {\shiftlabel{#3}}}

\newcommand\reddot[1]{\filldraw[fill=white, draw=red] (#1) circle (1.5pt)}
\newcommand\bluedot[1]{\filldraw[fill=white, draw=blue] (#1) circle (1.5pt)}
\newcommand\blackdot[1]{\filldraw[fill=white, draw=black] (#1) circle (1.5pt)}

\newcommand\multreddot[3][]{%
  \ifthenelse{\isempty{#1}}{%
    \filldraw[fill=white, draw=red] (#2) circle (1.5pt)
  }{%
    \filldraw[fill=white, draw=red] (#2) circle (1.5pt) node[anchor=#1] {\color{red} \dotlabel{#3}}
  }
}
\newcommand\multbluedot[3][]{%
  \ifthenelse{\isempty{#1}}{%
    \filldraw[fill=white, draw=blue] (#2) circle (1.5pt)
  }{%
    \filldraw[fill=white, draw=blue] (#2) circle (1.5pt) node[anchor=#1] {\color{blue} \dotlabel{#3}}
  }
}
\newcommand\multblackdot[3][]{%
  \ifthenelse{\isempty{#1}}{%
    \filldraw[fill=white] (#2) circle (1.5pt)
  }{%
    \filldraw[fill=white] (#2) circle (1.5pt) node[anchor=#1] {\dotlabel{#3}}
  }
}
\newcommand\redtoken[3][]{%
  \ifthenelse{\isempty{#1}}{%
    \filldraw[red] (#2) circle (1.5pt)
  }{%
    \filldraw[red] (#2) circle (1.5pt) node[anchor=#1] {\dotlabel{#3}}
  }
}
\newcommand\bluetoken[3][]{%
  \ifthenelse{\isempty{#1}}{%
    \filldraw[blue] (#2) circle (1.5pt)
  }{%
    \filldraw[blue] (#2) circle (1.5pt) node[anchor=#1] {\dotlabel{#3}}
  }
}

\newcommand\graytoken[3][]{%
 \ifthenelse{\isempty{#1}}{%
    \filldraw[gray] (#2) circle (1.5pt)
  }{%
    \filldraw[gray] (#2) circle (1.5pt) node[anchor=#1] {$\color{black}\scriptstyle{#3}$}
  }
}

\newcommand\blacktoken[3][]{%
  \ifthenelse{\isempty{#1}}{%
    \filldraw (#2) circle (1.5pt)
  }{%
    \filldraw (#2) circle (1.5pt) node[anchor=#1] {\dotlabel{#3}}
  }
}

\newcommand\opensq[2][black]{%
  \filldraw[fill=white,draw=#1] (#2)++(-0.06,-0.06) rectangle ++(0.12,0.12)
}

\newcommand\dashdumb[2]{
  \draw[very thick,dotted] (#1) to (#2);
  \blackdot{#1};
  \blackdot{#2}
}
\newcommand\dotdumb[2]{
  \draw (#1) to (#2);
  \blackdot{#1};
  \blackdot{#2}
}
\newcommand\sqdumb[2]{
  \draw (#1) to (#2);
  \opensq{#1};
  \opensq{#2}
}
\newcommand\dottrip[3]{
  \draw (#1) to (#2) to (#3);
  \blackdot{#1};
  \blackdot{#2};
  \blackdot{#3}
}
\newcommand\sqtrip[3]{
  \draw (#1) to (#2) to (#3);
  \opensq{#1};
  \opensq{#2};
  \opensq{#3}
}
\newcommand\teleport[4]{
  \draw (#3) to (#4);
  \filldraw[#1] (#3) circle (1.5pt);
  \filldraw[#2] (#4) circle (1.5pt)
}

\newcommand\intleft[2]{
  \filldraw[->,draw=#1,fill=white] (#2)++(-0.18,0) arc (-180:180:0.18)
}
\newcommand\intright[2]{
  \filldraw[->,draw=#1,fill=white] (#2)++(0.18,0) arc (0:-360:0.18)
}
\newcommand\intleftsm[2]{
  \filldraw[white] (#2) circle (0.1);
  \draw[#1,->] (#2)++(0,-0.1) to[out=0,in=-90] ++(0.1,0.1) to[out=90,in=0] ++(-0.1,0.1) to[out=180,in=90] ++(-0.09,-0.14);
  \draw[#1] (#2)++(-0.09,-0.02) to[out=-84,in=0] ++(0.09,-0.08)
}
\newcommand\intrightsm[2]{
  \filldraw[white] (#2) circle (0.1);
  \draw[#1,->] (#2)++(0,-0.1) to[out=180,in=-90] ++(-0.1,0.1) to[out=90,in=180] ++(0.1,0.1) to[out=0,in=90] ++(0.09,-0.14);
  \draw[#1] (#2)++(0.09,-0.02) to[out=-96,in=0] ++(-0.09,-0.08)
}

\newcommand\bubright[4][black]{
  \draw[->,#1] (#2)++(0,0.2) arc(90:-270:0.2);
  \filldraw[#1] (#2)++(0.2,0) circle (1.5pt) node[anchor=west] {\dotlabel{#3}};
  \filldraw[fill=white, draw=#1] (#2)++(-0.2,0) circle (1.5pt) node[anchor=east] {{\color{#1} \dotlabel{#4}}}
}

\newcommand\jonbubright[4][black]{
  \draw[->,#1] (#2)++(0.2,0) arc(0:-360:0.2);
  \filldraw[#1] (#2)++(-0.18,.1) circle (1.5pt) node[anchor=east] {\dotlabel{#3}};
  \filldraw[fill=white, draw=#1] (#2)++(-0.18,-0.1) circle (1.5pt) node[anchor=east] {{\color{#1} \dotlabel{#4}}}
}
\newcommand\JONbubright[4][black]{
  \draw[->,#1] (#2)++(0,0.3) arc(90:-270:0.3);
  \filldraw[#1] (#2)++(-0.25,.17) circle (1.5pt) node[anchor=east] {\dotlabel{#3}};
  \filldraw[fill=white, draw=#1] (#2)++(-0.25,-0.17) circle (1.5pt) node[anchor=east] {{\color{#1} \dotlabel{#4}}}
}

\newcommand\bubleft[4][black]{
  \draw[->,#1] (#2)++(0,0.2) arc(90:450:0.2);
  \filldraw[#1] (#2)++(-0.2,0) circle (1.5pt) node[anchor=east] {\dotlabel{#3}};
  \filldraw[fill=white, draw=#1] (#2)++(0.2,0) circle (1.5pt) node[anchor=west] {{\color{#1} \dotlabel{#4}}}
}
\newcommand\jonbubleft[4][black]{
  \draw[->,#1] (#2)++(-0.2,0) arc(180:540:0.2);
  \filldraw[#1] (#2)++(0.18,-0.1) circle (1.5pt) node[anchor=west] {\dotlabel{#3}};
  \filldraw[fill=white, draw=#1] (#2)++(0.18,.1) circle (1.5pt) node[anchor=west] {{\color{#1} \dotlabel{#4}}}
}
\newcommand\JONbubleft[4][black]{
  \draw[->,#1] (#2)++(0,0.3) arc(90:450:0.3);
  \filldraw[#1] (#2)++(0.25,-0.17) circle (1.5pt) node[anchor=west] {\dotlabel{#3}};
  \filldraw[fill=white, draw=#1] (#2)++(0.25,.17) circle (1.5pt) node[anchor=west] {{\color{#1} \dotlabel{#4}}}
}

\newcommand\bubgenright[3][black]{%
  \draw[->,#1] (#2)++(0,0.2) arc(90:-270:0.2);
  \node at (#2) {\dotlabel{#3}}
}
\newcommand\bubgenleft[3][black]{%
  \draw[<-,#1] (#2)++(0,0.2) arc(90:-270:0.2);
  \node at (#2) {\dotlabel{#3}}
}

\newcommand\cbubble[3][black]{
  \begin{tikzpicture}[baseline={(0,-0.15)}]
    \jonbubright[#1]{0,0}{#2}{#3};
  \end{tikzpicture}
}

\newcommand\ccbubble[3][black]{
  \begin{tikzpicture}[baseline={(0,-0.15)}]
    \jonbubleft[#1]{0,0}{#2}{#3};
  \end{tikzpicture}
}

\tikzset{darkg/.style={green!70!black}}

\tikzset{->-/.style={decoration={
  markings,
  mark=at position #1 with {\arrow{>}}},postaction={decorate}}}
\tikzset{-<-/.style={decoration={
  markings,
  mark=at position #1 with {\arrow{<}}},postaction={decorate}}}

\usepackage[capitalize]{cleveref}   

\crefname{defin}{Definition}{Definitions}
\crefname{eg}{Example}{Examples}
\crefname{lem}{Lemma}{Lemmas}
\crefname{theo}{Theorem}{Theorems}
\crefname{rem}{Remark}{Remarks}
\crefname{equation}{}{}
\crefname{enumi}{}{}

\newcommand\Z{\mathbb{Z}}
\newcommand\Q{\mathbb{Q}}

\newcommand\N{\mathbb{N}}
\newcommand\OO{\mathbb{O}}
\newcommand\kk{\Bbbk}
\newcommand\KK{\mathbb{K}}
\newcommand{\inv}{\!\operatorname{-inv}}
\newcommand{\anti}{\!\operatorname{-anti}}

\newcommand\WC[1][A]{\mathpzc{Wr}(#1)}                
\newcommand\AWC[1][A]{\mathpzc{Wr}^\textup{aff}(#1)}  
\newcommand\WA[2][A]{\mathrm{Wr}_{#2}(#1)}                     
\newcommand\AWA[2][A]{\mathrm{Wr}^\textup{aff}_{#2}(#1)}       
\newcommand\AWAR[2][A]{\mathrm{Wr}^\textup{aff}_{#2}({#1}_R)}  
\newcommand\CWA[3][A]{\mathrm{Wr}^{#3}_{#2}(#1)}               
\newcommand\CWAR[3][A]{\mathrm{Wr}^{#3}_{#2}(#1_R)}          
\newcommand\CWAopR[3][A]{\mathrm{Wr}^{#3}_{#2}(#1_R^\op)}
\newcommand\GCQ[3][R]{\mathpzc{H}_{#1}(#2|#3)} 
\newcommand\Zpi{{\Z_\pi}}
\newcommand\Qpi{{\Q_\pi}}
\newcommand\Zq{{\Z_\pi[q,q^{-1}]}}
\newcommand\Qq{{\Q_\pi[q,q^{-1}]}}
\newcommand\cA{\mathcal{A}}
\newcommand\cB{\mathcal{B}}
\newcommand\cC{\mathcal{C}}
\newcommand\cD{\mathcal{D}}
\newcommand\cEnd{\mathcal{E}nd}

\newcommand\cIR{\mathcal{I}_R}
\newcommand\cP{\mathcal{P}}

\newcommand\cV{\mathcal{V}}
\newcommand\cW{\mathcal{W}}
\newcommand{\B}{\mathbf{B}}
\newcommand\ba{\mathbf{a}}
\newcommand\bb{\mathbf{b}}

\newcommand\be{\mathbf{e}}

\newcommand\one{\mathbbm{1}}
\newcommand\blambda{{{\text{\boldmath$\lambda$}}}}
\newcommand\bmu{{{\text{\boldmath$\mu$}}}}

\newcommand\rHeis{q\!\operatorname{-Heis}}
\newcommand\rtHeis{\operatorname{Heis}}

\newcommand\op{\mathrm{op}}                 

\newcommand\rev{\mathrm{rev}}               

\newcommand\even{{\bar 0}}
\newcommand\odd{{\bar 1}}

\newcommand\gsmod{\operatorname{gsmod-}\!}
\newcommand\smod{\operatorname{smod-}\!}

\newcommand\pgsmod{\operatorname{pgsmod-}\!}
\newcommand\psmod{\operatorname{psmod-}\!}
\newcommand\ug[1]{\underline{#1}}                   
\newcommand\tgamma{\underline{\gamma}}

\newcommand\fh{\mathfrak{h}}
\newcommand\fS{\mathfrak{S}}

\newcommand{\dA}{d}                         
\newcommand{\AR}{A_R}                     

\newcommand\actplus{{\scriptstyle\,\oplus\,}}
\newcommand\actminus{{\scriptstyle\,\ominus\,}}

\newcommand\red[1]{{\color{red} #1 }}
\newcommand\blue[1]{{\color{blue} #1}}

\newcommand\upblue{\blue{\uparrow}}
\newcommand\upred{\red{\uparrow}}
\newcommand\downblue{\blue{\downarrow}}
\newcommand\downred{\red{\downarrow}}
\newcommand\cDelt[2]{%
  \Delta_{{\color{blue} #1} | {\color{red} #2}}
}

\newcommand\Heis[3][black]{%
  {\color{#1} {\mathpzc{Heis}}_{#3}(#2)}
}
\newcommand\HeisR[3][black]{%
  {\color{#1} {\mathpzc{Heis}}_{#3}({#2}_R)}
}

\newcommand\Heisenv[3][black]{%
  {\color{#1} {\mathpzc{Heis}}_{#3}(#2)_{q,\pi}}
}
\newcommand\extdim[2]{
    \binom{#2}{#1}_{q,\pi}
}
\newcommand\qint[1]{
  [#1]
}

\newcommand\SVec{\mathpzc{SVec}}

\newcommand\GSVec{\mathpzc{GSVec}}

\newcommand\SCat{\mathpzc{SCat}}
\newcommand\GSEnd{\mathpzc{GSEnd}}
\newcommand\SEnd{\mathpzc{SEnd}}

\newcommand\barotimes{\, \overline{\otimes}\, }
\newcommand\xodot[1]{\,\odot_{#1}\,}
\newcommand\barodot{\, \overline{\odot}_{2d}\, }
\newcommand\barodotnonumber{\, \overline{\odot}\, }

\newcommand\cocenter[1]{
    \dot{#1}
}

\newcommand\Laurent[1]{
    (\!(#1)\!)
}

\DeclareMathOperator{\Add}{Add}

\DeclareMathOperator{\End}{End}

\DeclareMathOperator{\Ev}{Ev}
\DeclareMathOperator{\flip}{flip}
\DeclareMathOperator{\gr}{gr}       
\DeclareMathOperator{\grdim}{gsdim} 
\DeclareMathOperator{\Hom}{Hom}

\DeclareMathOperator{\id}{id}       
\DeclareMathOperator{\Id}{Id}       
\DeclareMathOperator{\ind}{ind}
\DeclareMathOperator{\Kar}{Kar}     
\DeclareMathOperator{\res}{res}

\DeclareMathOperator{\Sym}{Sym}
\DeclareMathOperator{\tr}{tr}

\newtheorem{theo}{Theorem}[section]
\newtheorem{prop}[theo]{Proposition}
\newtheorem{lem}[theo]{Lemma}
\newtheorem*{lem*}{Lemma}
\newtheorem{cor}[theo]{Corollary}

\theoremstyle{definition}
\newtheorem{defin}[theo]{Definition}
\newtheorem{rem}[theo]{Remark}
\newtheorem{eg}[theo]{Example}

\numberwithin{equation}{section}
\allowdisplaybreaks

\setenumerate[1]{label=(\alph*)}  
\setenumerate[2]{label=(\roman*)}

\newtoggle{comments}
\newtoggle{details}
\newtoggle{detailsnote}

\iftoggle{comments}{%
  \newcommand{\acomments}[1]{
    \ \\
    {\color{red}
      \textbf{AS:} #1
    }
    \ \\
    }
  \newcommand{\bcomments}[1]{
    \ \\
    {\color{blue!50!black}
      \textbf{BW:} #1
    }
    \ \\
    }
  
}{%
  \newcommand{\acomments}[1]{}
  \newcommand{\bcomments}[1]{}
}

\iftoggle{details}{%
  \newcommand{\details}[1]{
      \ \\
      {\color{OliveGreen}
        \textbf{Details:} #1
      }
      \ \\
  }
}{%
  \newcommand{\details}[1]{}
}

\begin{document}

\title{Foundations of Frobenius Heisenberg categories}

\author{Jonathan Brundan}
\address[J.B.]{
  Department of Mathematics \\
  University of Oregon \\
  Eugene, OR, USA
}
\email{brundan@uoregon.edu}

\author{Alistair Savage}
\address[A.S.]{
  Department of Mathematics and Statistics \\
  University of Ottawa \\
  Ottawa, ON, Canada
}
\urladdr{\href{http://alistairsavage.ca}{alistairsavage.ca}, \textrm{\textit{ORCiD}:} \href{https://orcid.org/0000-0002-2859-0239}{orcid.org/0000-0002-2859-0239}}
\email{alistair.savage@uottawa.ca}

\author{Ben Webster}
\address[B.W.]{
  Department of Pure Mathematics, University of Waterloo \&
  Perimeter Institute for Theoretical Physics\\
  Waterloo, ON, Canada
}
\email{ben.webster@uwaterloo.ca}

\thanks{J.B.\ supported in part by NSF grant DMS-1700905.}
\thanks{A.S.\ supported by Discovery Grant RGPIN-2017-03854 from the Natural Sciences and Engineering Research Council of Canada.}
\thanks{B.W.\ supported by Discovery Grant RGPIN-2018-03974 from the Natural Sciences and Engineering Research Council of Canada.  Research at Perimeter Institute is supported in part by the Government of Canada through the Department of Innovation, Science and Economic Development Canada and by the Province of Ontario through the Ministry of Colleges and Universities.}

\begin{abstract}
    We describe bases for the morphism spaces of the Frobenius Heisenberg categories associated to a symmetric graded Frobenius algebra, proving several open conjectures.  Our proof uses a categorical comultiplication and generalized cyclotomic quotients of the category.  We use our basis theorem to prove that the Grothendieck ring of the Karoubi envelope of the Frobenius Heisenberg category recovers the lattice Heisenberg algebra associated to the Frobenius algebra.
\end{abstract}

\subjclass[2020]{Primary 18M05; Secondary 17B10, 17B65}
\keywords{Heisenberg category, wreath product, Frobenius algebra}

\ifboolexpr{togl{comments} or togl{details}}{%
  {\color{magenta}DETAILS OR COMMENTS ON}
}{%
}

\maketitle
\thispagestyle{empty}


\section{Introduction}

Throughout this article we work over a field $\kk$ of characteristic
zero (see also \cref{park}).  The \emph{Frobenius Heisenberg category}
$\Heis{A}{k}$ is a strictly pivotal graded $\kk$-linear monoidal
category depending on a graded Frobenius algebra $A$ and a central
charge $k \in \Z$.  Its name arises from the fact that it categorifies
the quantum analog $\rHeis_k(A)$ of a lattice Heisenberg algebra depending on $A$.  The category first appeared for $A=\kk$ and $k=-1$ in work of Khovanov \cite{Kho14}, motivated by the study of induction and restriction functors between modules for symmetric groups.  The definition was then extended to the case of arbitrary $A$ and $k$ in a series of works by several authors \cite{CL12,RS17,MS18,Bru18,Sav18}. In fact, in its most general form, $A$ can be a graded Frobenius {\em super}algebra, and $\Heis{A}{k}$ is a monoidal \emph{super}category in the sense of \cite{BE17}.

The Frobenius Heisenberg categories are defined in terms of generators and relations.  In order to fully understand them, one wants an explicit linear basis for each morphism space.  The standard approach to proving such basis theorems involves two steps.  First, one derives relations that provide a straightening algorithm, allowing one to reduce arbitrary morphisms to linear combinations of morphisms in a standard form.  Then, to prove linear independence, one exploits actions on certain module categories constructed from some natural cyclotomic quotients.  This was the original approach used to give a basis theorem for $\Heis{\kk}{-1}$ \cite{Kho14}, and then also to give a basis theorem for $\Heis{\kk}{k}$ \cite{MS18} for all $k \neq 0$.

Unfortunately, the above method for proving linear independence
\emph{fails} for general $A$.  The natural action of $\Heis{A}{k}$ is
on modules for the \emph{cyclotomic wreath product algebras} studied
in \cite{Sav20} (see also \cite{KM19}).  However, one can show by
degree considerations that certain morphisms, expected to be nonzero
in $\Heis{A}{k}$, must act as zero in any such module category; see
\cite[Rem.~8.11]{RS17}.  Therefore, until now, basis theorems in this
general setting have remained conjectural, even in the important case
where $A$ is a zigzag algebra as in \cref{motivatingeg}, when the Frobenius Heisenberg category is related to the geometry of Hilbert schemes \cite{CL12} and categorical vertex operators \cite{CL11}.

One of the main results of the current paper is a basis theorem for the $\Heis{A}{k}$ in general (\cref{basis}).  Our proof involves the construction of a categorical comultiplication
\[
  \Delta_{l|m} \colon \Heis{A}{k} \to \Add\left(\Heis{A}{l}\barodotnonumber\Heis{A}{m}\right),
  \quad k=l+m,
\]
where $\Heis{A}{l}\barodotnonumber\Heis{A}{m}$ is a certain
localization of a symmetric product of the monoidal supercategories
$\Heis{A}{l}$ and $\Heis{A}{m}$, and $\Add$ denotes additive envelope.
This functor allows one to form tensor products of Heisenberg module
categories, provided that the localized morphism acts invertibly.
Equipped with this method of forming larger module categories, we are
able to construct asymptotically faithful module categories using
\emph{generalized cyclotomic quotients}. This technique, which was
first used in \cite{Web16}, was employed in  \cite{BSW-K0} to give a
new proof of the basis theorem for  $\Heis{\kk}{k}$.  The current
paper further illustrates the wide applicability of this approach.   It will  be used in \cite{BSW-qFrobHeis} to prove a basis theorem for quantum Frobenius Heisenberg categories built from the quantum affine wreath product algebras of \cite{RS19}. 

Let $e \in A$ be a homogeneous idempotent such that $A = AeA$, so that
$A$ and $eAe$ are graded Morita equivalent.  Note that $eAe$ is also a graded Frobenius superalgebra with trace map that is the restriction of the given trace on $A$.  As a first application of the basis theorem, we prove that
the graded Karoubi envelopes of $\Heis{A}{k}$ and $\Heis{eAe}{k}$ are equivalent as graded monoidal supercategories (\cref{moritathm}).  For example, if $A$ is semisimple and purely even with trivial grading, so that $A$ is Morita equivalent to the direct sum of $N$ copies of the field $\kk$ (where $N$ is the number of pairwise inequivalent irreducible $A$-modules), this result implies that the graded Karoubi envelope of $\Heis{A}{k}$ is monoidally equivalent to the graded Karoubi envelope of the symmetric product of $N$ copies of $\Heis{\kk}{k}$; the latter category is studied in \cite{Gan18}.

The other main result of the current paper concerns the Grothendieck ring of the Frobenius Heisenberg supercategory.  We prove that the Grothendieck ring of the graded Karoubi envelope of $\Heis{A}{k}$ is isomorphic to the lattice Heisenberg algebra $\rHeis_k(A)$ of central charge $k$
associated to $A$ (\cref{K0isom}) providing the following hypothesis holds (see also \cref{Dagger}):
\begin{equation} \tag{$\dagger$} \label{virginia}
    \text{The graded Frobenius superalgebra $A$ is positively graded
      with $A_0$ being purely even and semisimple.}
\end{equation}
In the case that the grading on $A$ is nontrivial, this was proved already in \cite[Th.~1.5]{Sav18}, although the proof there left many details to the reader and it was assumed there that additional relations were imposed in the category, allowing one to reduce degree zero bubbles to scalars.  The argument given here is based instead on ideas from \cite{BSW-K0}, which computed the Grothendieck ring in the special case that $A = \kk$.  These methods allow us to give a more concise argument that also does not require the additional bubble relations needed in \cite[Th.~1.5]{Sav18}.  This provides another example of the power of the categorified comultiplication technique.  We also explain a sense in which the comultiplication $\Delta_{l|m}$ categorifies the comultiplication on $\rHeis_k(A)$ (\cref{camping}).

We do make one simplifying assumption compared to \cite{Sav18}: we assume throughout that $A$ is a \emph{symmetric} graded Frobenius superalgebra. This greatly simplifies the exposition and covers most of the cases of current interest, including zigzag algebras (where the Frobenius Heisenberg category is related to the geometry of the Hilbert scheme \cite{CL12} and vertex operators \cite{CL11}) and group algebras of finite groups (where the Frobenius Heisenberg category provides a tool for studying representations of wreath product groups and their higher level analogues).  The one exception is the case where $A$ is a rank one Clifford algebra, which is excluded by our assumption.  Nevertheless, we expect that the methods of the current paper could be adapted to the Clifford case in one of two ways.  One method is to relax the assumption that $A$ is symmetric, and keep track of the Nakayama automorphism as in \cite{Sav18}.   Alternatively, one can allow the trace map of $A$ to be odd, in which case the (polynomial) dot generators become odd.  This is considered in the case $k=-1$ in \cite{RS17}.  We expect the Clifford case to be treated in detail in \cite{CK}.

\iftoggle{detailsnote}{
\subsection*{Hidden details} For the interested reader, the tex file of the arXiv version of this paper includes hidden details of some straightforward computations and arguments that are omitted in the pdf file.  These details can be displayed by switching the \texttt{details} toggle to true in the tex file and recompiling.
}{}

\subsection*{Corrections to published version}  This version of the paper contains corrections of some errors present in the published version:
\begin{itemize}
    \item Indices on $(f \circ g)$ at the end of \cref{pienv} were fixed.
    \item Equation \cref{wednesdaypm} was corrected.
    \item In the first line of the proof of \cref{Chicago}, $\OO_r$ was changed to $\OO^{(r)}$.
    \item Three minus signs in the proof of \cref{Chicago} were deleted.
\end{itemize}

\section{Monoidal supercategories\label{sec:gradecat}}

Let $\kk$ be a fixed ground field of characteristic zero.  All vector spaces, algebras, categories and
functors will be assumed to be linear over $\kk$ unless otherwise specified.  Unadorned tensor products denote tensor products over $\kk$.  Almost everything in the article will be enriched over the category $\SVec$ of \emph{vector superspaces}, that is, $\Z/2$-graded vector spaces $V = V_{\even}\oplus V_{\odd}$ with parity-preserving morphisms.  Writing $\bar v \in \Z/2$ for the parity of a homogeneous vector $v \in V$, the category $\SVec$ is a symmetric monoidal category with symmetric braiding $V \otimes W \rightarrow W\otimes V$ defined by $v\otimes w \mapsto (-1)^{\bar v \bar w} w \otimes v$; this formula as written only makes sense if both $v$ and $w$ are homogeneous, with its intended meaning for general vectors following by linearity.

Sometimes we will be working with an additional $\Z$-grading, that is,
we will be working in the category $\GSVec$ of \emph{graded vector
  superspaces} $V = \bigoplus_{n \in \Z} V_n = \bigoplus_{n \in \Z}
V_{n,\even}\oplus V_{n,\odd}$ with grading-preserving morphisms.  The
term \emph{positively graded} will mean graded with all negative
graded pieces equal to zero, and the grading will be called
\emph{trivial} if it is concentrated in degree zero, i.e.\ $V = V_0 =
V_{0,\even}\oplus V_{0,\odd}$.  Assuming all $V_{n,r}$ are finite dimensional, we set
\begin{equation}
    \grdim V := \sum_{n \in \Z, r \in \Z/2} q^n\pi^r \dim V_{n,r} \in \Zq,
\end{equation}
where $\Zpi := \Z[\pi] / (\pi^2-1)$.  Also let $\Qpi := \Q[\pi] /
(\pi^2-1)$. Since all of the subtleties involving signs come from the underlying vector superspace, it is usually straightforward to incorporate this additional $\Z$-grading into the definitions, so we will not say much more about it below.

For superalgebras $A=A_\even\oplus A_\odd$ and $B = B_\even\oplus
B_\odd$, multiplication in the superalgebra $A \otimes B$ is defined
by \begin{equation}
(a' \otimes b) (a \otimes b') = (-1)^{\bar a \bar b} a'a \otimes
bb'
\end{equation}
for homogeneous $a,a' \in A$, $b,b' \in B$.  The \emph{opposite}
superalgebra $A^\op$ is a copy  $\{a^\op : a \in A\}$
of the vector superspace $A$ with
multiplication defined from
\begin{equation}\label{mups}
a^\op  b^\op := (-1)^{\bar a \bar b} (ba)^{\op}.
\end{equation}
The \emph{center} $Z(A)$ is the subalgebra generated by all homogeneous $a \in A$ such that
\begin{equation}\label{center}
    ab = (-1)^{\bar a \bar b}ba
\end{equation}
for all homogeneous $b \in A$.  The \emph{cocenter} $C(A)$ (which is merely a vector superspace not a
superalgebra!) is the quotient of $A$ by the subspace spanned by $ab-(-1)^{\bar a \bar b}ba$ for all homogeneous $a,b \in A$.  In case $A$ and $B$ are graded superalgebras, so are $A \otimes B$, $A^\op$ and $Z(A)$, while $C(A)$ is a graded superspace.

In fact, throughout this document, we will be working with \emph{strict monoidal supercategories} in the sense of \cite{BE17}.  In such a category, the \emph{super interchange law} is
\begin{equation}\label{interchange}
    (f' \otimes g) \circ (f \otimes g')
    = (-1)^{\bar f \bar g} (f' \circ f) \otimes (g \circ g').
\end{equation}
We denote the unit object by $\one$ and the identity morphism of an object $X$ by $1_X$.  We will use the usual calculus of string diagrams, representing the horizontal composition $f \otimes g$ (resp.\ vertical composition $f \circ g$) of morphisms $f$ and $g$ diagrammatically by drawing $f$ to the left of $g$ (resp.\ drawing $f$ above $g$).  Care is needed with horizontal levels in such diagrams due to
the signs arising from the super interchange law:
\begin{equation}\label{intlaw}
    \begin{tikzpicture}[anchorbase]
        \draw (-0.5,-0.5) -- (-0.5,0.5);
        \draw (0.5,-0.5) -- (0.5,0.5);
        \filldraw[fill=white,draw=black] (-0.5,0.15) circle (5pt);
        \filldraw[fill=white,draw=black] (0.5,-0.15) circle (5pt);
        \node at (-0.5,0.15) {$\scriptstyle{f}$};
        \node at (0.5,-0.15) {$\scriptstyle{g}$};
    \end{tikzpicture}
    \quad=\quad
    \begin{tikzpicture}[anchorbase]
        \draw (-0.5,-0.5) -- (-0.5,0.5);
        \draw (0.5,-0.5) -- (0.5,0.5);
        \filldraw[fill=white,draw=black] (-0.5,0) circle (5pt);
        \filldraw[fill=white,draw=black] (0.5,0) circle (5pt);
        \node at (-0.5,0) {$\scriptstyle{f}$};
        \node at (0.5,0) {$\scriptstyle{g}$};
    \end{tikzpicture}
    \quad=\quad
    (-1)^{\bar f\bar g}\
    \begin{tikzpicture}[anchorbase]
        \draw (-0.5,-0.5) -- (-0.5,0.5);
        \draw (0.5,-0.5) -- (0.5,0.5);
        \filldraw[fill=white,draw=black] (-0.5,-0.15) circle (5pt);
        \filldraw[fill=white,draw=black] (0.5,0.15) circle (5pt);
        \node at (-0.5,-0.15) {$\scriptstyle{f}$};
        \node at (0.5,0.15) {$\scriptstyle{g}$};
    \end{tikzpicture}
    \ .
\end{equation}
We refer the reader to \cite{Sav-exp} for a brief overview of these concepts, to \cite[Ch.~1, 2]{TV17} for a more in-depth treatment, and to \cite{BE17} for a detailed discussion of signs in the superalgebra setting.  We review below a few of the more important ideas that are crucial for our exposition and somewhat less well known.

\medskip

A \emph{supercategory} means a category enriched in $\SVec$. Thus, its morphism spaces are actually superspaces and composition is parity-preserving.  A \emph{superfunctor} between supercategories induces a parity-preserving linear map between morphism superspaces.  For superfunctors $F,G \colon \cA\rightarrow\cB$, a \emph{supernatural transformation} $\alpha \colon F \Rightarrow G$ of \emph{parity $r\in\Z/2$} is the data of morphisms $\alpha_X\in \Hom_{\cB}(FX, GX)_r$ for each $X\in\cA$ such that $Gf \circ \alpha_X = (-1)^{r \bar f}\alpha_Y\circ Ff$ for each homogeneous $f \in \Hom_{\cA}(X, Y)$.  Note when $r$ is odd that $\alpha$ is \emph{not} a natural transformation in the usual sense due to the sign. A \emph{supernatural transformation} $\alpha \colon F \Rightarrow G$ is $\alpha = \alpha_\even + \alpha_\odd$ with each $\alpha_r$ being a supernatural transformation of parity $r$.  A superfunctor $F \colon \cA \to \cB$ is an \emph{equivalence} of supercategories if there exists a superfunctor $G \colon \cB \to \cA$ and even supernatural
isomorphisms $GF \Rightarrow \Id_\cA$ and $FG \Rightarrow \Id_\cB$.

Enriching in $\GSVec$ in place of $\SVec$, one obtains similar notions with an additional $\Z$-grading, which we call \emph{graded supercategory}, \emph{graded superfunctor} and \emph{graded supernatural transformation}.

\begin{defin}\label{pienv}
    For a supercategory $\cA$, its \emph{$\Pi$-envelope} $\cA_\pi$ is the supercategory with objects given by formal symbols $\{ \Pi^r X : X \in \cA,\ r \in \Z/2\}$ and morphisms defined by
    \begin{equation}
        \Hom_{\cA_\pi}( \Pi^r X, \Pi^s Y )
        := \Pi^{s-r} \Hom_{\cA}(X,Y),
    \end{equation}
    where, on the right-hand side, $\Pi$ denotes the parity shift operator determined by $(\Pi V)_r := V_{r-\odd}$ for a vector superspace $V$. The composition law in $\cA_\pi$ is induced in the obvious way from the one in $\cA$: writing $f_r^s$ for the morphism in $\Hom_{\cA_\pi}(\Pi^r X, \Pi^s Y)$
    of parity $\bar{f}+r-s$ defined by $f \in \Hom_{\cA}(X,Y)$, we have that $f_s^u \circ g_r^s = (f \circ g)_r^u$.
\end{defin}

The $\Pi$-envelope $\cA_\pi$ from \cref{pienv} is a \emph{$\Pi$-supercategory} in the sense of \cite[Def.~1.7]{BE17} with parity shift functor $\Pi \colon \cA_\pi\rightarrow \cA_\pi$ sending object $\Pi^r X$ to $\Pi^{r+\odd}X$ and morphism $f_r^s$ to $f_{r+\odd}^{s+\odd}$.  Viewing $\cA$ as a full subcategory of its $\Pi$-envelope $\cA_\pi$ via the canonical embedding
\begin{equation}\label{flowers}
    J \colon \cA \rightarrow \cA_\pi, \qquad X \mapsto \Pi^\even X, \quad f \mapsto f_\even^\even,
\end{equation}
the $\Pi$-envelope satisfies a universal property: any superfunctor $F \colon \cA \rightarrow \cB$ to a $\Pi$-supercategory $\cB$ extends in a canonical way to a graded functor $\tilde F \colon \cA_\pi \rightarrow \cB$ such that $\tilde F \circ \Pi = \Pi \circ \tilde F$.  In turn, any supernatural transformation $\theta \colon F \Rightarrow G$ between superfunctors $F, G \colon \cA\rightarrow \cB$ extends in a unique way to a supernatural transformation $\tilde\theta \colon \tilde F \Rightarrow \tilde G$; see \cite[Lem.~4.2]{BE17}.

The \emph{underlying category} of a $\Pi$-supercategory is the category with the same objects, but only the even morphisms. It has the structure of a \emph{$\Pi$-category} in the sense of \cite[Def.~1.6]{BE17}.

\begin{defin}\label{qpienv}
    For a graded supercategory $\cA$, its \emph{$(Q,\Pi)$-envelope} is the graded supercategory $\cA_{q,\pi}$ with objects $\{Q^m \Pi^r X : X \in\cA,\ m \in \Z,\ r \in \Z/2\}$ and morphisms
    \begin{equation}
        \Hom_{\cA_{q,\pi}}( Q^m \Pi^r X, Q^n \Pi^s Y )
        := Q^{n-m} \Pi^{s-r} \Hom_{\cA}(X,Y),
    \end{equation}
    where $Q$ on the right is the grading shift functor $(Q V)_n :=
    V_{n-1}$.  We use the notation $f_{m,r}^{n,s}$ for the morphism in
    $\Hom_{\cA_{q,\pi}}(Q^m \Pi^r X, Q^n \Pi^s Y)$ of degree
    $\deg(f)+n-m$ and parity $\bar{f}+s-r$ defined by $f \in
    \Hom_{\cA}(X,Y)$.
\end{defin}

The $(Q, \Pi)$-envelope $\cA_{q,\pi}$ from \cref{qpienv} is a \emph{graded $\Pi$-supercategory} in the sense of \cite[Def.~6.4]{BE17}, and satisfies analogous universal properties to $\Pi$-envelopes.

The \emph{underlying category} of a $(Q,\Pi)$-supercategory has only the even morphisms of degree zero, and is a \emph{$(Q,\Pi)$-category} in the sense of \cite[Def.~6.12]{BE17}.  Intuitively, the passage from a graded supercategory $\cA$ to the underlying category of its $(Q,\Pi)$-envelope $\cA_{q,\pi}$ should be thought of as modifying $\cA$ so that its morphisms may be reinterpreted as even morphisms of degree zero between formal parity shifts of objects.

\medskip

The \emph{Karoubi envelope} $\Kar(\cA)$ of a supercategory (resp.\ a graded supercategory) $\cA$ is the completion of its additive envelope $\Add(\cA)$ at all homogeneous idempotents.  Thus, objects of $\Kar(\cA)$ are pairs $(X,e)$ consisting of a finite direct sum $X$ of objects of $\cA$ together with a homogeneous idempotent $e \in \End_{\Add(\cA)}(X)$.  Morphisms $(X,e) \rightarrow (Y,f)$ are elements of $f\Hom_{\Add(\cA)}(X,Y)e$.  Also let $K_0(\Kar(\cA))$ denote the split Grothendieck group of the underlying category of $\Kar(\cA)$, that is, the abelian group generated by degree zero isomorphism classes of objects subject to the relations $[X]+[Y]=[X \oplus Y]$.  If $\cA$ is a $\Pi$-supercategory (resp.\ a graded $(Q,\Pi)$-supercategory), the grading shift functors extend by the usual universal property of Karoubi envelopes to make $\Kar(\cA)$ into a $\Pi$-supercategory (resp.\ a graded $(Q,\Pi)$-supercategory) too.  Then $K_0(\Kar(\cA))$ is a $\Zpi$-module (resp.\ a $\Zq$-module) with
\begin{align} \label{elephant}
    \pi [X] &= [\Pi X],
    &q^n [X] &= [Q^n X].
\end{align}
Starting from a graded supercategory $\cA$, the Karoubi envelope of the $(Q,\Pi)$-envelope of $\cA$ is the \emph{graded Karoubi envelope} $\Kar(\cA_{q,\pi})$ of $\cA$.  This has objects that are summands of finite direct sums of formal degree and parity shifts of objects of $\cA$.

\medskip

The archetypical example of a $\Pi$-supercategory arises as follows.  Let $A$ be a locally unital superalgebra, i.e.\ an associative (but not necessarily unital) superalgebra with a system $\{1_X : X\in \mathbb{A}\}$ of mutually orthogonal idempotents such that $A = \bigoplus_{X,Y \in \mathbb{A}} 1_Y A 1_X$.  Let $\smod A$ be the category of (locally unital) right $A$-supermodules and all (not necessarily homogeneous) $A$-module homomorphisms.  To make the supercategory $\smod A$ into a $\Pi$-supercategory, we define $\Pi$ so that it sends supermodule $V$ to the parity switch $\Pi V$
viewed as a right $A$-module with the same underlying action as before.  On a homogeneous morphism $f \colon V \rightarrow W$, $\Pi f \colon \Pi V \rightarrow \Pi W$ is the function $(-1)^{\bar f} f$.
Note that the usual category of $A$-supermodules and parity-preserving morphisms is the underlying category of our category $\smod A$.  Let $\psmod A$ be the full subcategory of $\smod A$ consisting of the finitely-generated projective supermodules.

The data of a locally unital superalgebra $A$ is just the same as the data of a supercategory $\cA$.  Indeed, given $\cA$, one takes the set $\mathbb{A}$ indexing the distinguished idempotents of $A$ to be the object set of $\cA$, then defines $A$ so that $1_Y A 1_X = \Hom_{\cA}(X,Y)$ with multiplication induced by composition in $\cA$.  There is then a superfunctor $\cA \rightarrow \psmod A$ sending object $X$ to the right ideal $1_X A$ and morphism $a \in 1_Y A 1_X$ to the homomorphism $1_X A \rightarrow 1_Y A$ defined by left multiplication by $a$.  Applying the universal properties first of $\Pi$-envelope then of Karoubi envelope, this induces a superfunctor
\begin{equation}\label{yoneda}
    \Kar\big(\cA_\pi\big) \rightarrow \psmod A.
\end{equation}
The Yoneda lemma implies that this functor is an equivalence of supercategories.  Using it, we may identify the $\Zpi$-module $K_0\left(\Kar\big(\cA_\pi\big)\right)$ with the usual Grothendieck group $K_0(\psmod A)$.

The same constructions can be performed in the graded setting.  The data of a graded supercategory $\cA$ is equivalent to the data of a locally unital graded superalgebra $A$.  Let $\gsmod A$ be the category of graded right $A$-supermodules, with morphisms being sums $f = \sum_{n \in \Z} f_n$ of homogeneous homomorphisms of various degrees; a homogeneous $A$-supermodule homomorphism $f \colon V \rightarrow W$ of degree $n$ is an $A$-supermodule homomorphism as usual sending each graded piece $V_m$ into $W_{m+n}$.  In fact, $\gsmod A$ is a $(Q,\Pi)$-supercategory with $Q$ and $\Pi$ acting
as grading and parity shift functors.  Let $\pgsmod A$ be the full subcategory consisting of the finitely generated projective graded supermodules.  The Yoneda lemma again gives an equivalence of graded supercategories
\begin{equation}\label{yoneda2}
    \Kar\big(\cA_{q,\pi}\big) \rightarrow \pgsmod A.
\end{equation}
Hence, $K_0\left(\Kar\big(\cA_{q,\pi}\big)\right)\cong K_0(\pgsmod A)$ as $\Zq$-modules.

\medskip

Now we step up to the monoidal situation.  We make the (for once, not $\kk$-linear!) category $\SCat$ of supercategories and superfunctors into a symmetric monoidal category following the general construction of \cite[$\S$1.4]{Kelly}.  In particular, for supercategories $\cA$ and $\cB$, their $\kk$-linear product, denoted $\cA\boxtimes\cB$, has as objects pairs $(X,Y)$ for $X \in \cA$ and $Y \in \cB$, and
\begin{equation}
    \Hom_{\cA\boxtimes\cB}((X,Y), (X',Y'))
    = \Hom_{\cA}(X,X') \otimes \Hom_{\cB}(Y,Y')
\end{equation}
with composition defined via \cref{interchange}.  A \emph{strict monoidal supercategory} is a supercategory $\cC$ with an associative, unital tensor functor $-\otimes- \colon \cC \boxtimes \cC \rightarrow \cC$.  For example, given any supercategory $\cA$, the category $\SEnd(\cA)$ of
superfunctors $F \colon \cA\rightarrow\cA$ and supernatural transformations is a strict monoidal supercategory.  See \cite[Def.~1.4]{BE17} for the appropriate notions of (not necessarily strict) \emph{monoidal superfunctors} between strict monoidal supercategories, and of \emph{monoidal natural transformations} between monoidal superfunctors (these are required to be even).

There is also a notion of \emph{strict monoidal $\Pi$-supercategory}; see \cite[Def.~1.12]{BE17}.
Such a category is a $\Pi$-supercategory in the earlier sense with $\Pi := \pi \otimes-$ for a distinguished object $\pi$ admitting an odd isomorphism $\zeta \colon \pi \xrightarrow{\sim} \one$.  For example, if $\cA$ is a $\Pi$-supercategory, then $\SEnd(\cA)$ is actually a strict monoidal $\Pi$-supercategory with the distinguished object $\pi$ being the parity shift functor $\Pi \colon \cA\rightarrow\cA$.  The underlying category of a strict monoidal $\Pi$-supercategory is a \emph{strict monoidal $\Pi$-category}. Again we refer to \cite[Def.~1.14]{BE17} for a detailed discussion.

The \emph{$\Pi$-envelope} $\cC_\pi$ of a strict monoidal supercategory $\cC$ is the $\Pi$-supercategory from \cref{pienv} viewed as a strict monoidal $\Pi$-supercategory with $\pi := \Pi\one$,
tensor product of objects defined by
\begin{equation}\label{curtain}
    \left( \Pi^r X \right) \otimes \left( \Pi^s Y \right)
    := \Pi^{r+s} (X \otimes Y),
\end{equation}
and tensor product (horizontal composition) of morphisms defined by
\begin{equation}\label{pole}
    f_r^s \otimes g_u^v := (-1)^{r(\bar{g}+u+v) +\bar{f} v}(f \otimes g)_{r+u}^{s+v}
\end{equation}
for homogeneous morphisms $f$ and $g$ in $\cC$.  See
\cite[Def.~1.16]{BE17} for more details and discussion of its
universal property.
When working with string diagrams, the morphism $f_{r}^{s}$ in $\cC_{\pi}$ may be represented by adding horizontal lines labeled by $r$ and $s$ at the bottom and top of the diagram for $f \colon X \to Y$:
    \begin{equation}\label{covid1}
        \begin{tikzpicture}[anchorbase]
          \draw (0,-0.4) -- (0,0.4);
          \filldraw[fill=white,draw=black] (0,0) circle(4pt);
          \node at (0,0) {\dotlabel{f}};
          \shiftline{-0.3,-0.4}{0.3,-0.4}{r};
          \shiftline{-0.3,0.4}{0.3,0.4}{s};
        \end{tikzpicture}
        \colon \Pi^r X \to \Pi^s Y.
    \end{equation}
Then the rules for horizontal and vertical composition in $\mathcal{C}_\pi$ become
  \begin{equation} \label{slush}
  \mathord{
  \begin{tikzpicture}[baseline =-1.3]
  	\draw (0.08,-.4) to (0.08,-.13);
  	\draw (0.08,.4) to (0.08,.13);
        \draw (0.08,0) circle (4pt);
     \node at (0.08,0) {$\scriptstyle{f}$};
  \draw[-,purple](.4,-.4) to (-.24,-.4);
  \draw[-,purple](.4,.4) to (-.24,.4);
  \node at (.5,.4) {$\color{purple}\scriptstyle s$};
  \node at (.5,-.4) {$\color{purple}\scriptstyle r$};
  \end{tikzpicture}
  }
  \,\otimes
  \mathord{
  \begin{tikzpicture}[baseline = -1.3]
  	\draw (0.08,-.4) to (0.08,-.13);
  	\draw (0.08,.4) to (0.08,.13);
        \draw (0.08,0) circle (4pt);
     \node at (0.08,0) {$\scriptstyle{g}$};
  \draw[-,purple](.4,-.4) to (-.24,-.4);
  \draw[-,purple](.4,.4) to (-.24,.4);
  \node at (.5,.4) {$\color{purple}\scriptstyle v$};
  \node at (.5,-.4) {$\color{purple}\scriptstyle u$};
  \end{tikzpicture}
  }
  =
  (-1)^{r(\bar{g}+u+v)+\bar{f}v}
  \mathord{
  \begin{tikzpicture}[baseline = -2]
  	\draw[-] (0.08,-.4) to (0.08,-.13);
  	\draw[-] (0.08,.4) to (0.08,.13);
        \draw (0.08,0) circle (4pt);
     \node at (0.08,0) {$\scriptstyle{g}$};
  	\draw[-] (-.8,-.4) to (-.8,-.13);
  	\draw[-] (-.8,.4) to (-.8,.13);
        \draw (-.8,0) circle (4pt);
     \node at (-.8,0) {$\scriptstyle{f}$};
  \draw[-,purple](.45,-.4) to (-1.1,-.4);
  \draw[-,purple](.45,.4) to (-1.1,.4);
  \node at (.74,.4) {$\color{purple}\scriptstyle s+v$};
  \node at (.74,-.4) {$\color{purple}\scriptstyle r+u$};
  \end{tikzpicture}
  },
  \qquad\!
  \mathord{
  \begin{tikzpicture}[baseline = -1.3]
  	\draw[-] (0.08,-.4) to (0.08,-.13);
  	\draw[-] (0.08,.4) to (0.08,.13);
        \draw (0.08,0) circle (4pt);
     \node at (0.08,0) {$\scriptstyle{f}$};
  \draw[-,purple](.4,-.4) to (-.24,-.4);
  \draw[-,purple](.4,.4) to (-.24,.4);
  \node at (.5,.4) {$\color{purple}\scriptstyle t$};
  \node at (.5,-.4) {$\color{purple}\scriptstyle s$};
  \end{tikzpicture}
  }
  \:\circ
  \mathord{
  \begin{tikzpicture}[baseline =-1.3]
  	\draw[-] (0.08,-.4) to (0.08,-.13);
  	\draw[-] (0.08,.4) to (0.08,.13);
        \draw (0.08,0) circle (4pt);
     \node at (0.08,0) {$\scriptstyle{g}$};
  \draw[-,purple](.4,-.4) to (-.24,-.4);
  \draw[-,purple](.4,.4) to (-.24,.4);
  \node at (.5,.4) {$\color{purple}\scriptstyle s$};
  \node at (.5,-.4) {$\color{purple}\scriptstyle r$};
  \end{tikzpicture}
  }
  =\!
  \mathord{
  \begin{tikzpicture}[baseline = 5]
  	\draw[-] (0.08,-.4) to (0.08,-.13);
  	\draw[-] (0.08,.37) to (0.08,.13);
  	\draw (0.08,.63) to (0.08,.9);
        \draw (0.08,0) circle (4pt);
        \draw (0.08,.5) circle (4pt);
     \node at (0.08,0) {$\scriptstyle{g}$};
     \node at (0.08,.51) {$\scriptstyle{f}$};
  \draw[-,purple](.4,-.4) to (-.24,-.4);
  \draw[-,purple](.4,.9) to (-.24,.9);
  \node at (.5,.9) {$\color{purple}\scriptstyle t$};
  \node at (.5,-.4) {$\color{purple}\scriptstyle r$};
  \end{tikzpicture}
  }.
\end{equation}

The \emph{Karoubi envelope} $\Kar(\cC)$ of a strict monoidal $\Pi$-supercategory is a strict monoidal $\Pi$-supercategory.  Its Grothendieck group $K_0(\Kar(\cC))$ is actually a $\Zpi$-algebra with multiplication induced by tensor product.

A \emph{module supercategory} over a strict monoidal supercategory $\cC$ is a supercategory $\cA$ together with a monoidal superfunctor $\cC \rightarrow \SEnd(\cA)$.  If $\cA$ is actually a $\Pi$-supercategory, then $\SEnd(\cA)$ is a strict monoidal $\Pi$-supercategory, so applying the universal property of $\Pi$-envelope gives us an induced monoidal superfunctor $\cC_\pi \rightarrow \SEnd(\cA)$.  In turn, if $\cA$ is also Karoubian, then so is $\SEnd(\cA)$, hence applying the universal property of Karoubi envelope gives us a monoidal superfunctor $\Kar(\cC_\pi) \rightarrow \SEnd(\cA)$. In particular, this makes $K_0(\cA)$ into a module over the $\Zpi$-algebra $K_0(\Kar(\cC_\pi))$.

In the presence of an additional $\Z$-grading, the definitions just recalled become \emph{strict graded monoidal supercategories} and \emph{strict graded monoidal $(Q,\Pi)$-supercategories}; see \cite[$\S$6]{BE17}.  The basic example of a strict graded monoidal supercategory is the category $\GSEnd(\cA)$ of graded superfunctors and graded supernatural transformations for a graded supercategory $\cA$.  For a strict graded monoidal supercategory $\cC$, its $(Q,\Pi)$-envelope
$\cC_{q,\pi}$ is a strict graded monoidal $(Q,\Pi)$-supercategory, and
we represent the morphism $f_{m,r}^{n,s}:Q^m \Pi^r X \rightarrow Q^n
\Pi^s Y$ in $\cC_{q,\pi}$ diagrammatically by
 \begin{equation}\label{covid2}
        \begin{tikzpicture}[anchorbase]
          \draw (0,-0.4) -- (0,0.4);
          \filldraw[fill=white,draw=black] (0,0) circle(4pt);
          \node at (0,0) {\dotlabel{f}};
          \shiftline{-0.3,-0.4}{0.3,-0.4}{m,r};
          \shiftline{-0.3,0.4}{0.3,0.4}{n,s};
        \end{tikzpicture}
        \colon Q^m\Pi^r X \to Q^n\Pi^s Y.
    \end{equation}
Horizontal and vertical composition in $\cC_{q,\pi}$ are defined in
terms of the underlying
parities as in \cref{slush}.
Then $K_0(\Kar(\cC_{q,\pi}))$ is a $\Zq$-algebra.  If $\cA$ is a
graded module supercategory over $\cC$ then $K_0(\Kar(\cA_{q,\pi}))$
is a module over $K_0(\Kar(\cC_{q,\pi}))$.

\section{The wreath product category\label{sec:wreath}}

Let $A$ be a graded superalgebra.  The symmetric group $\fS_n$ acts on
$A^{\otimes n}$ (tensor product of graded superalgebras) by permuting
the factors, with the appropriate sign when odd elements are flipped.  To match our diagrammatic conventions to be introduced below, we will number the factors of $A^{\otimes n}$ from right to left.  In particular,
\[
    s_i(a_n \otimes \dotsb \otimes a_1)
    = (-1)^{\bar a_i \bar a_{i+1}} a_n \otimes \dotsb \otimes a_{i+2} \otimes a_i \otimes a_{i+1} \otimes a_{i-1} \otimes \dotsb \otimes a_1,\quad
    a_1,\dotsc,a_n \in A,
\]
where $s_i$ denotes the simple transposition of $i$ and $i+1$.  The \emph{wreath product algebra} $\WA{n} = A^{\otimes n} \rtimes \fS_n$ is the algebra equal to $A^{\otimes n} \otimes \kk \fS_n$ as a vector space, with multiplication determined by
\[
    (\ba \otimes \sigma) (\bb \otimes \tau) = \ba \sigma(\bb) \otimes \sigma \tau,\quad
    \ba, \bb \in A^{\otimes n},\ \sigma,\tau \in \fS_n.
\]
This a graded superalgebra, where we consider elements of $\fS_n$ to be of degree zero.

Let $\WC$ be the free strict symmetric graded monoidal supercategory generated by an object $\uparrow$ whose endomorphism algebra is $A$.  So $\WC$ is the strict graded monoidal supercategory generated by one object $\uparrow$ and morphisms
\begin{equation} \label{Wrupgen}
    \begin{tikzpicture}[anchorbase]
        \draw [->](0,0) -- (0.6,0.6);
        \draw [->](0.6,0) -- (0,0.6);
    \end{tikzpicture}
    \colon \uparrow \otimes \uparrow \to \uparrow \otimes \uparrow,\quad
    \begin{tikzpicture}[anchorbase]
        \draw[->] (0,0) -- (0,0.6);
        \blacktoken[west]{0,0.3}{a};
    \end{tikzpicture}
    \colon \uparrow \to \uparrow
    ,\ a \in A,
\end{equation}
where the crossing is even of degree zero and the morphism
$
  \begin{tikzpicture}[anchorbase]
    \draw[->] (0,0) -- (0,0.4);
    \blacktoken[west]{0,0.2}{a};
  \end{tikzpicture}
$,
which we refer to as a \emph{token}, is of the same degree and parity as $a$.  The relations are as follows:
\begin{align} \label{wreathrel}
    \begin{tikzpicture}[anchorbase]
        \draw[->] (0,0) -- (0,0.7);
        \blacktoken[west]{0,0.35}{1};
    \end{tikzpicture}
    &=
    \begin{tikzpicture}[anchorbase]
        \draw[->] (0,0) -- (0,0.7);
    \end{tikzpicture}
    ,&
    \lambda\:
    \begin{tikzpicture}[anchorbase]
        \draw[->] (0,0) -- (0,0.7);
        \blacktoken[west]{0,0.35}{a};
    \end{tikzpicture}
    + \mu\:
    \begin{tikzpicture}[anchorbase]
        \draw[->] (0,0) -- (0,0.7);
        \blacktoken[west]{0,0.35}{b};
    \end{tikzpicture}
    &=
    \begin{tikzpicture}[anchorbase]
        \draw[->] (0,0) -- (0,0.7);
        \blacktoken[west]{0,0.35}{\lambda a + \mu b};
    \end{tikzpicture}
    ,&
    \begin{tikzpicture}[anchorbase]
        \draw[->] (0,0) -- (0,0.7);
        \blacktoken[east]{0,0.2}{b};
        \blacktoken[east]{0,0.45}{a};
    \end{tikzpicture}
    &=
    \begin{tikzpicture}[anchorbase]
        \draw[->] (0,0) -- (0,0.7);
        \blacktoken[west]{0,0.35}{ab};
    \end{tikzpicture}
    ,&
    \begin{tikzpicture}[anchorbase]
        \draw[->] (-0.3,-0.3) -- (0.3,0.3);
        \draw[->] (0.3,-0.3) -- (-0.3,0.3);
        \blacktoken[east]{-0.15,-0.15}{a};
    \end{tikzpicture}
    &=
    \begin{tikzpicture}[anchorbase]
        \draw[->] (-0.3,-0.3) -- (0.3,0.3);
        \draw[->] (0.3,-0.3) -- (-0.3,0.3);
        \blacktoken[west]{0.1,0.1}{a};
    \end{tikzpicture}
    ,&
    \begin{tikzpicture}[anchorbase]
        \draw[->] (0,0) -- (1,1);
        \draw[->] (1,0) -- (0,1);
        \draw[->] (0.5,0) .. controls (0,0.5) .. (0.5,1);
    \end{tikzpicture}
    &=
    \begin{tikzpicture}[anchorbase]
        \draw[->] (0,0) -- (1,1);
        \draw[->] (1,0) -- (0,1);
        \draw[->] (0.5,0) .. controls (1,0.5) .. (0.5,1);
    \end{tikzpicture}
    ,&
    \begin{tikzpicture}[anchorbase]
        \draw[->] (0,0) \braidto (0.5,0.5) \braidto (0,1);
        \draw[->] (0.5,0) \braidto (0,0.5) \braidto (0.5,1);
    \end{tikzpicture}
    &=
    \begin{tikzpicture}[anchorbase]
        \draw[->] (0,0) --(0,1);
        \draw[->] (0.5,0) -- (0.5,1);
    \end{tikzpicture}
    \ ,
\end{align}
for $a,b \in A$ and $\lambda,\mu \in \kk$.  It follows from the defining relations that the map
\begin{equation*}
    A \to \End_{\WC}(\uparrow),\quad a \mapsto
    \begin{tikzpicture}[anchorbase]
        \draw[->] (0,0) -- (0,0.6);
        \blacktoken[west]{0,0.3}{a};
    \end{tikzpicture}
    \ ,
\end{equation*}
is a graded superalgebra homomorphism and, also using \cref{intlaw}, we have automatically that
\begin{align}\label{inkyoto}
    \begin{tikzpicture}[anchorbase]
        \draw[->] (0,0) -- (0,0.7);
        \blacktoken[east]{0,0.45}{a};
        \draw[->] (0.4,0) -- (0.4,0.7);
        \blacktoken[west]{.4,0.25}{b};
    \end{tikzpicture}
    &=
    \begin{tikzpicture}[anchorbase]
        \draw[->] (0,0) -- (0,0.7);
        \blacktoken[east]{0,0.35}{a};
        \draw[->] (0.4,0) -- (0.4,0.7);
        \blacktoken[west]{.4,0.35}{b};
    \end{tikzpicture}
    = (-1)^{\bar a \bar b}
    \begin{tikzpicture}[anchorbase]
        \draw[->] (0,0) -- (0,0.7);
        \blacktoken[east]{0,0.25}{a};
        \draw[->] (0.4,0) -- (0.4,0.7);
        \blacktoken[west]{.4,0.45}{b};
    \end{tikzpicture}
    \ ,&
    \begin{tikzpicture}[anchorbase]
        \draw[->] (-0.3,-0.3) -- (0.3,0.3);
        \draw[->] (0.3,-0.3) -- (-0.3,0.3);
        \blacktoken[west]{0.15,-0.15}{a};
    \end{tikzpicture}
    &=
    \begin{tikzpicture}[anchorbase]
        \draw[->] (-0.3,-0.3) -- (0.3,0.3);
        \draw[->] (0.3,-0.3) -- (-0.3,0.3);
        \blacktoken[east]{-0.1,0.1}{a};
    \end{tikzpicture}
    \ .
\end{align}
The objects of $\WC$ are $\left\{\uparrow^{\otimes n} : n \in \N\right\}$.  There are no nonzero  morphisms $\uparrow^{\otimes m} \to \uparrow^{\otimes n}$ for $m \ne n$.  Furthermore, we have an isomorphism of graded superalgebras
\begin{equation} \label{Wrimath1}
    \imath_n \colon \WA{n} \xrightarrow{\cong} \End_{\WC}(\uparrow^{\otimes n})
\end{equation}
sending the simple transposition $s_i$ to the crossing of the $i$-th and $(i+1)$-st strings and $1^{\otimes (n-i)} \otimes a \otimes 1^{\otimes (i-1)}$ to a token labeled $a$ on the $i$-th string,
numbering strings $1,2,\dotsc,n$ from right to left.  We will use this isomorphism to \emph{identify}
elements of $\WA{n}$ (hence also elements of $\fS_n$) with morphisms
in $\End_{\WC}(\uparrow^{\otimes n})$ without further comment.

Suppose $\cC$ and $\cD$ are strict graded monoidal supercategories
defined by generators and relations, e.g.\ $\cC=\cD=\WC$.
Their free product is the strict graded monoidal supercategory defined by taking the disjoint union of
the given generators and relations of $\cC$ and $\cD$.  Then the \emph{symmetric product} $\cC \odot \cD$ is the strict graded monoidal category obtained from the free product by adjoining additional
even degree zero isomorphisms $\sigma_{X,Y} \colon X \otimes Y \xrightarrow{\cong} Y \otimes X$ for each pair of objects $X \in \cC$ and $Y \in \cD$, subject to the relations
\begin{align}\label{flight1}
    \sigma_{X_1 \otimes X_2, Y} &= (\sigma_{X_1,Y} \otimes 1_{X_2}) \circ (1_{X_1} \otimes \sigma_{X_2,Y}),
    &\sigma_{X,Y_1 \otimes Y_2} &= (1_{Y_1} \otimes \sigma_{X,Y_2}) \circ (\sigma_{X,Y_1} \otimes 1_{Y_2}),
    \\
    \sigma_{X_2,Y} \circ (f \otimes 1_Y) &= (1_Y \otimes f) \circ \sigma_{X_1,Y},
    &\sigma_{X,Y_2} \circ (1_X \otimes g) &= (g \otimes 1_X) \circ \sigma_{X,Y_1},\label{flight2}
\end{align}
for all $X, X_1, X_2 \in \cC$, $Y, Y_1, Y_2 \in \cD$, $f \in \Hom_\cC(X_1,X_2)$, and $g \in \Hom_\cD(Y_1,Y_2)$.

We use two colors to denote the factors in the symmetric product $\blue{\WC} \odot \red{\WC}$. Morphisms are then represented by linear combinations of string diagrams colored blue and red.  In order to write down an efficient monoidal presentation for this category, as well as the colored tokens and one-color crossings that are the generating morphisms of $\blue{\WC}$ and $\red{\WC}$, one just needs two additional generating morphisms represented by the two-color crossings
\begin{equation}\label{referencethislatertoo}
    \begin{tikzpicture}[anchorbase]
        \draw [->,blue](0,0) -- (0.6,0.6);
        \draw [->,red](0.6,0) -- (0,0.6);
    \end{tikzpicture}
    := \sigma_{\upblue,\upred}
    \colon \upblue \otimes \upred \to \upred \otimes \upblue,
    \qquad
    \begin{tikzpicture}[anchorbase]
        \draw [->,red](0,0) -- (0.6,0.6);
        \draw [->,blue](0.6,0) -- (0,0.6);
    \end{tikzpicture}
    :=
    \left(\sigma_{\upblue,\upred}\right)^{-1}
    \colon \upred \otimes \upblue \to \upblue \otimes \upred,
\end{equation}
both of which are even of degree zero.  All of the other required morphisms $\sigma_{X,Y}$ and their inverses can be obtained from these using \cref{flight1}.  Then, as well as the defining relations of
$\blue{\WC}$ and $\red{\WC}$, one needs additional relations asserting that the two-color crossings displayed above are mutual inverses, and that the colored tokens and one-color crossings commute with the two-color crossings as in \cref{flight2}.

There is a strict graded monoidal functor, which we call \emph{categorical comultiplication},
\begin{equation}\label{stroppel}
    \Delta \colon \WC \to \Add(\blue{\WC} \odot \red{\WC})
\end{equation}
sending the generating object $\uparrow$ to $\upblue \oplus \upred$, and defined on the generating morphisms by
\[
    \Delta
    \left(\:
        \begin{tikzpicture}[anchorbase]
            \draw[->] (0,-0.3) to (0,0.3);
            \blacktoken[west]{0,0}{a};
        \end{tikzpicture}
    \right)
    =\
    \begin{tikzpicture}[anchorbase]
        \draw[->,blue] (0,-0.3) to (0,0.3);
        \bluetoken[west]{0,0}{a};
    \end{tikzpicture}
    +\
    \begin{tikzpicture}[anchorbase]
        \draw[->,red] (0,-0.3) to (0,0.3);
        \redtoken[west]{0,0}{a};
    \end{tikzpicture}
    \ ,\qquad
    \Delta
    \left(\:
        \begin{tikzpicture}[anchorbase]
            \draw[->] (-0.4,-0.4) to (0.4,0.4);
            \draw[->] (0.4,-0.4) to (-0.4,0.4);
        \end{tikzpicture}
    \: \right)
    =
    \begin{tikzpicture}[anchorbase]
            \draw[->,blue] (-0.4,-0.4) to (0.4,0.4);
            \draw[->,blue] (0.4,-0.4) to (-0.4,0.4);
    \end{tikzpicture}
    \ +\
    \begin{tikzpicture}[anchorbase]
        \draw[->,red] (-0.4,-0.4) to (0.4,0.4);
        \draw[->,red] (0.4,-0.4) to (-0.4,0.4);
    \end{tikzpicture}
    \ +\
    \begin{tikzpicture}[anchorbase]
        \draw[->,blue] (-0.4,-0.4) to (0.4,0.4);
        \draw[->,red] (0.4,-0.4) to (-0.4,0.4);
    \end{tikzpicture}
    \ +\
    \begin{tikzpicture}[anchorbase]
        \draw[->,red] (-0.4,-0.4) to (0.4,0.4);
        \draw[->,blue] (0.4,-0.4) to (-0.4,0.4);
    \end{tikzpicture}
    \ .
\]
This follows from the universal property defining $\WC$: the category $\Add(\blue{\WC}\odot\red{\WC})$ is a strict symmetric monoidal category and the map
$
    a \mapsto
    \begin{tikzpicture}[anchorbase]
        \draw[->,blue] (0,-0.3) to (0,0.3);
        \bluetoken[west]{0,0}{a};
    \end{tikzpicture}
    +\
    \begin{tikzpicture}[anchorbase]
        \draw[->,red] (0,-0.3) to (0,0.3);
        \redtoken[west]{0,0}{a};
    \end{tikzpicture}
$
defines a homomorphism from $A$ to the endomorphism algebra of the
object $\upblue\oplus \upred$.

Let $\cP$ be the set of all partitions, writing $\lambda^T$ for the
transpose of partition $\lambda$.
For $0 \le r \le n$, let $\cP_{r,n}$ denote the set of size $\binom{n}{r}$ consisting of tuples $\lambda = (\lambda_1,\dotsc,\lambda_r) \in \Z^r$ such that $n-r \ge \lambda_1 \ge \dotsb \ge \lambda_r \ge 0$.  Thus, viewing partitions as Young diagrams, $\cP_{r,n}$ is the set of partitions fitting inside an $r \times (n-r)$ rectangle.  We let $\min_{r,n}$ (resp.\ $\max_{r,n}$) denote the element $\lambda \in \cP_{r,n}$ with $\lambda_1 = \dotsb = \lambda_r = 0$ (resp.\ $\lambda_1 = \dotsb = \lambda_r = n-r$).  For $\lambda \in \cP_{r,n}$, we define
\begin{equation}\label{referencelater}
    \uparrow^{\otimes\lambda} :=
    \upblue^{\otimes (n-r-\lambda_1)} \otimes \upred \otimes \upblue^{\otimes (\lambda_1-\lambda_2)} \otimes \upred \otimes \dotsb \otimes \upred \otimes \upblue^{\otimes \lambda_r}
    \in {\color{blue} \WC} \odot {\color{red} \WC}.
\end{equation}
In particular, $\uparrow^{\otimes \min_{r,n}} = \upblue^{\otimes (n-r)} \otimes \upred^{\otimes r}$ and $\uparrow^{\otimes \max_{r,n}} = \upred^{\otimes r} \otimes \upblue^{\otimes (n-r)}$.  We denote the identity morphism of $\uparrow^{\otimes \lambda}$ by $1_\lambda$.  There is a unique isomorphism
\begin{equation} \label{organize}
    \sigma_\lambda \colon \uparrow^{\otimes \lambda}
    \ \xrightarrow{\cong}\ \uparrow^{\otimes \min_{r,n}}
\end{equation}
whose diagram only involves crossings of the form
$
  \begin{tikzpicture}[anchorbase]
    \draw[->,red] (-0.2,-0.25) to (0.2,0.25);
    \draw[->,blue] (0.2,-0.25) to (-0.2,0.25);
  \end{tikzpicture};
$
in particular $\sigma_{\min_{r,n}} = 1_{\min_{r,n}}$.
We will view the tensor product $\WA{n-r} \otimes \WA{r}$ as a subalgebra of $\WA{n}$ via the  identifications
\begin{gather*}
    1 \otimes s_i \leftrightarrow s_i, \quad
    s_j \otimes 1 \leftrightarrow s_{r+j}, \quad
    1 \otimes \ba \leftrightarrow 1^{\otimes (n-r)} \otimes \ba, \quad
    \bb \otimes 1 \leftrightarrow \bb \otimes 1^{\otimes r},
\end{gather*}
for $\ba \in A^{\otimes r}$, $\bb \in A^{\otimes (n-r)}$, $1 \le i \le r-1$, $1 \le j \le n-r-1$.  Generalizing \cref{Wrimath1}, we have an isomorphism of graded superalgebras
\begin{equation} \label{Wrimath2}
    \imath_{r,n} \colon \WA{n-r} \otimes \WA{r} \to \End_{\blue{\WC} \odot \red{\WC}} \left( \upblue^{\otimes (n-r)} \otimes \upred^{\otimes r} \right).
\end{equation}
Combining this with the elements $\{ \sigma_\lambda^{-1} \circ \sigma_\mu : \lambda, \mu \in \cP_{r,n} \}$, which give the matrix units, we see that
\begin{equation} \label{box}
    \End_{\Add(\blue{\WC} \odot \red{\WC})} \left( (\upblue \oplus \upred)^{\otimes n} \right)
    \cong \bigoplus_{r=0}^n \mathrm{Mat}_{\binom{n}{r}} \left( \WA{n-r} \otimes \WA{r} \right).
\end{equation}

\begin{rem} \label{coass1}
    The categorical comultiplication is coassociative in the sense that the compositions $(\Delta \odot \Id) \circ \Delta$ and $(\Id \odot \Delta) \circ \Delta$ agree (on identifying $(\blue{\WC} \odot \red{\WC}) \odot {\color{green} \WC}$ with $\blue{\WC} \odot (\red{\WC} \odot {\color{green} \WC})$).
\end{rem}

In the remainder of the section, we are going to describe what the
above constructions look like at
the level of the Grothendieck ring.
We assume for this that the algebra $A$ is {\em finite dimensional}
and moreover, for simplicity, that
{\em all irreducible $A$-supermodules are of type $\mathtt{M}$}, i.e.\
they do not admit any odd automorphisms. This is automatic, for
example, if $A$ satisfies the hypothesis \eqref{virginia} from the
introduction.
Fix a choice of homogeneous idempotents $e_1,\dotsc,e_N \in A$ such that $e_1 A, \dotsc, e_N A$ is a complete list of representatives of the isomorphism classes of indecomposable projective graded right $A$-supermodules, up to grading and parity shift.  Then $K_0(\pgsmod A)$ is a free $\Zq$-module with basis given by the classes $\left\{[e_i A] : 1\leq i\leq N\right\}$.

We refer to $N$-tuples
$\blambda = (\lambda^1,\dotsc,\lambda^N) \in \cP^N$ of partitions as
\emph{multipartitions},
defining the transpose
\begin{equation}\label{transpose}
\blambda^T:= \left( (\lambda^1)^T, \dotsc, (\lambda^N)^T \right)
\end{equation}
for $\blambda = (\lambda^1,\dotsc,\lambda^N)$.
We denote the \emph{size} of $\lambda \in \cP$ by $|\lambda|$, and set
$|\blambda| := |\lambda^1| + \dotsb + |\lambda^N|$.  For
$\lambda \in \cP$ of size $n$, let $e_\lambda$ denote the corresponding Young idempotent of $\kk \fS_n$, so that $\left\{e_\lambda \,\kk \fS_n : \lambda \in \cP,\, |\lambda|= n\right\}$ is a complete set of representatives of the isomorphism classes of irreducible right $\kk\fS_n$-modules.  In particular, the complete symmetrizer and complete antisymmetrizer in $\kk\fS_n$ are given by
\[
    e_{(n)} = \frac{1}{n!} \sum_{\sigma \in \fS_n} \sigma
    \quad \text{and} \quad
    e_{(1^n)} = \frac{1}{n!} \sum_{\sigma \in \fS_n} (-1)^{\ell(\sigma)} \sigma,
\]
where $\ell(\sigma)$ is the length of the permutation $\sigma$.  Given also $i \in \{1,\dotsc,N\}$, let
\begin{equation} \label{sugar}
    e_{\lambda,i}
    := e_\lambda (e_i \otimes \cdots \otimes e_i \otimes e_i)
    = (e_i \otimes\cdots\otimes e_i \otimes e_i) e_\lambda
    \in \WA{n}.
\end{equation}
Then for $\blambda \in \cP^N$ we let
\begin{equation} \label{berries}
    e_\blambda
    := e_{\lambda^N,N} \otimes \cdots \otimes e_{\lambda^2,2} \otimes e_{\lambda^1,1}
    \in \End_{\WC}(\uparrow^{\otimes n}).
\end{equation}

\begin{lem}[{\cite[Prop.~4.4]{RS17}}] \label{trips}
    The right ideals $\left\{e_\blambda \WA{n} : \blambda \in \cP^N,\,
      |\blambda|=n\right\}$ give a complete set of representatives of
    the isomorphism classes of indecomposable projective graded right
    $\WA{n}$-supermodules, up to grading and parity shift.  Hence
    their isomorphism classes give a basis for $K_0(\pgsmod \WA{n})$ as a free
    $\Zq$-module.
The same assertions hold without the grading, so that the isomorphism classes
$\left\{[e_\blambda \WA{n}] : \blambda \in \cP^N,\,
      |\blambda|=n\right\}$ give a basis for
$K_0(\psmod \WA{n})$ as a free
    $\Zpi$-module.
\end{lem}

Recall from \cref{sec:gradecat} that $\Kar(\WC_{q,\pi})$ denotes the completion of the additive envelope of the $(Q,\Pi)$-envelope $\WC_{q,\pi}$ at all homogeneous idempotents.  Its Grothendieck ring is a $\Zq$-algebra with the $\Zq$-action defined as in \cref{elephant}.  For $\blambda \in \cP^N$ with $|\blambda|=n$, let
\begin{equation} \label{hide}
    S_\blambda := \left( \uparrow^{\otimes n}, e_\blambda \right) \in \Kar\big(\WC_{q,\pi}\big).
\end{equation}
In particular, for $n \in \N$, $1 \le i \le N$, the objects
\begin{align} \label{seek}
    H_{n,i} &:= \left( \uparrow^{\otimes n}, e_{(n),i} \right),
    & E_{n,i} &:= \left( \uparrow^{\otimes n}, e_{(1^n),i} \right)
\end{align}
are equal to $S_{\blambda}$ for the multipartition $\blambda \in \cP^N$ with  $\lambda^i = (n)$ or $(1^n)$, respectively, and $\lambda^j = \varnothing$ for all $j \neq i$.

Let $\Sym_{\Zq}$ and $\Sym_\Qq$ denote the rings of symmetric functions over $\Zq$ and $\Qq$, respectively.  Then let $\Sym_\Zq^{\otimes N} \subseteq \Sym_\Qq^{\otimes N}$ be the rings
\begin{align} \label{dogs}\textstyle
    \Sym_\Zq^{\otimes N}
    &:= \Sym_\Zq \otimes_{\Zq} \cdots \otimes_{\Zq} \Sym_\Zq,
    \\\label{cats}
    \Sym_\Qq^{\otimes N} &:= \Sym_\Qq \otimes_\Qq\cdots\otimes_\Qq \Sym_\Qq,
\end{align}
where in both cases there are $N$ tensor factors.  For $f \in \Sym_\Zq$ and $1 \le i \le N$, we write $f_i$ to denote the pure tensor in $\Sym_\Zq^{\otimes N}$ equal to $f$ in the $i$-th factor and to $1$ in the other factors.  If $p_n$, $e_n$, and $h_n$ denote the usual power sums, elementary symmetric functions, and complete symmetric functions, respectively, then $\Sym_\Zq^{\otimes N}$ is freely generated by $\{e_{n,i} : n >0,\ 1 \le i \le N\}$ or $\{h_{n,i} : n > 0,\ 1 \le i \le N\}$.  In addition, $\{p_{n,i} : n >0,\ 1 \le i \le N\}$ freely generates $\Sym_\Qq^{\otimes N}$.  Writing $s_\lambda$ for the Schur function associated to $\lambda \in \cP$, let
\[
    s_\blambda
    := s_{\lambda^N} \otimes \dotsb \otimes s_{\lambda^1}
    = s_{\lambda^N,N} \dotsm s_{\lambda^1,1}
\]
for $\blambda \in \cP^N$.  Then $\Sym_\Zq^{\otimes N}$ is free as a $\Zq$-module with the basis $\{s_\blambda : \blambda \in \cP^N\}$.

\begin{lem}\label{definingg}
    There is a $\Zq$-algebra isomorphism
    \begin{equation*}
        \gamma \colon \Sym_\Zq^{\otimes N}
        \xrightarrow{\cong}
        K_0\left(\Kar\left(\WC_{q,\pi}\right)\right),\qquad
        s_{\lambda,i} \mapsto [S_{\lambda,i}],\quad
        h_{n,i} \mapsto [H_{n,i}],\quad
        e_{n,i} \mapsto [E_{n,i}].
    \end{equation*}
\end{lem}

\begin{proof}
    In view of \cref{yoneda2}, we can identify $K_0(\Kar(\WC_{q,\pi}))$ with $\bigoplus_{n \geq 0} K_0(\pgsmod \WA{n})$ so that $[S_\blambda]$ corresponds to $[e_\blambda \WA{n}]$ for $\blambda \in \cP^N$ with $|\blambda|=n$.  It follows that the classes $\left\{[S_\blambda] : \blambda \in \cP^N\right\}$ give a basis for $K_0(\Kar(\WC_{q,\pi}))$ as a free $\Zq$-module. Hence, $\gamma$ is a $\Zq$-module isomorphism. For the proof that it is actually an algebra isomorphism, see \cite[Lem.~5.3]{RS17}.
\end{proof}

We view $\Sym_\Zq^{\otimes N}$ as a Hopf algebra over $\Zq$ with comultiplication $\delta \colon f \mapsto \sum_{(f)} f_{(1)} \otimes f_{(2)}$ determined by
\begin{equation} \label{Symcomult}
    \delta(h_{n,i}) = \sum_{r=0}^n h_{n-r,i} \otimes h_{r,i},\quad
    \delta(e_{n,i}) = \sum_{r=0}^n e_{n-r,i} \otimes e_{r,i},\quad
    \delta(p_{n,i}) = p_{n,i} \otimes 1 + 1 \otimes p_{n,i},
\end{equation}
where $h_{0,i} = e_{0,i} = 1$ by convention.
The final result in this section shows that
the functor (\ref{stroppel}) categorifies this comultiplication.
To formulate the result precisely, observe that
the inclusions of $\blue{\WC}$ and $\red{\WC}$ into $\blue{\WC} \odot
\red{\WC}$ induce inclusions of
$\Kar(\blue{\WC_{q,\pi}})$ and $\Kar(\red{\WC_{q,\pi}})$ into $\Kar((\blue{\WC} \odot \red{\WC})_{q,\pi})$.  In turn, these induce a $\Zq$-algebra homomorphism
\begin{equation}
    \epsilon \colon  K_0 \left( \Kar\left(\blue{\WC_{q,\pi}}\right) \right)
    \otimes_\Zq K_0 \left( \Kar\left(\red{\WC_{q,\pi}}\right) \right)
    \rightarrow K_0 \left( \Kar \left (\left(\blue{\WC} \odot \red{\WC}\right)_{q,\pi}\right) \right).
\end{equation}
Moreover, the map
\begin{equation} \label{topaz}
    \epsilon \circ (\gamma\otimes\gamma) \colon \Sym_\Zq^{\otimes N}\otimes_{\Zq} \Sym_\Zq^{\otimes N}
    \xrightarrow{\cong} K_0 \left(\Kar\left(\blue{\WC}\odot \red{\WC}\right)_{q,\pi}\right)
\end{equation}
is an isomorphism of $\Zq$-algebras. This follows from \cref{box,definingg}.
The functor $\Delta$ extends to a monoidal functor $\tilde\Delta \colon \Kar(\WC_{q,\pi}) \to \Kar((\blue{\WC} \odot \red{\WC})_{q,\pi})$.  The following lemma implies that the diagram
\begin{equation}\label{flyingoverVancouverIsland}
    \begin{tikzcd}
        \Sym^{\otimes N}_\Zq \arrow[r, "\delta"] \arrow[d, "\gamma" right, "\cong" left]
        &
        \Sym^{\otimes N}_\Zq \otimes_\Zq \Sym^{\otimes N}_\Zq \arrow[d, "\epsilon \circ (\gamma \otimes \gamma)" right, "\cong" left]
        \\
        K_0\left(\Kar\left(\WC_{q,\pi}\right)\right) \arrow[r, "{[\tilde\Delta]}"]
        &
        K_0\left(\Kar\left(\left(\blue{\WC} \odot \red{\WC}\right)_{q,\pi}\right)\right)
    \end{tikzcd}
\end{equation}
commutes, i.e.\ $\Delta$ categorifies $\delta$.

\begin{lem} \label{Symcat}
    For $n > 0$ and $1 \le i \le N$, there are degree zero isomorphisms
    \begin{align*}
        \tilde\Delta(H_{n,i}) &\cong \bigoplus_{r=0}^n \blue{H_{n-r,i}} \otimes \red{H_{r,i}},
        &
        \tilde\Delta(E_{n,i}) &\cong \bigoplus_{r=0}^n \blue{E_{n-r,i}} \otimes \red{E_{r,i}}
    \end{align*}
in $\Kar((\blue{\WC} \odot \red{\WC})_{q,\pi})$.
\end{lem}

\begin{proof}
    The proof of this result is almost identical to that of \cite[Th.~3.2]{BSW-K0}; one merely replaces the $e_{\lambda}$ appearing there by $e_{\lambda,i}$.
\end{proof}

\section{The affine wreath product category\label{sec:affwreath}}

We assume for the remainder of the article that $A$ is a \emph{symmetric graded Frobenius superalgebra} with trace $\tr \colon A \to \kk$ that is even of degree $-2\dA$, $\dA \in \Z$.  This means that
\begin{equation}
    \tr(ab) = (-1)^{\bar a \bar b} \tr(ba),\quad
    a,b \in A.
\end{equation}
The assumption that $A$ is Frobenius implies that it is
finite dimensional.
Recall also the definitions of the center $Z(A)$ (see \cref{center}) and cocenter $C(A)$.  We denote the canonical image of $a \in A$ in $C(A)$ by $\cocenter{a}$.

\begin{lem} \label{slipstream}
There is a well-defined nondegenerate pairing $Z(A) \times C(A) \to \kk,  (a,\cocenter{b}) \mapsto \tr(ab)$, $a,b \in A$.
\end{lem}

\begin{proof}
    For $c \in A$, we have
    \begin{align*}
        c \in Z(A)
        &\iff ca = (-1)^{\bar c \bar a} ac \ \forall\ a \in A \\
        &\iff \tr(cab) = (-1)^{\bar c \bar a} \tr(acb) \ \forall\ a,b \in A \\
        &\iff \tr \left( c \big( ab - (-1)^{\bar a \bar b} ba \big) \right) = 0 \ \forall\ a,b \in A.
    \end{align*}
  The result then follows from the nondegeneracy of the trace.
\end{proof}

Fix a choice of homogeneous basis $\B_A$ for the finite dimensional algebra $A$. Let $\{b^\vee : b \in \B_A\}$ be the dual basis determined by
\begin{equation}
    \tr(a^\vee b) = \delta_{a,b},\quad a,b \in \B_A.
\end{equation}
Note that $b^\vee$ is of the same parity as $b$, and $\deg(b^\vee) = 2d-\deg(b)$.  For any homogeneous $a \in A$, we let
\begin{equation}\label{adag}
    a^\dagger := \sum_{b \in \B_A} (-1)^{\bar a \bar b} b a  b^\vee,
\end{equation}
which is well-defined independent of the choice of the basis $\B_A$.  In particular, $1^\dagger \in A_{2d}$ is the \emph{canonical central element} of $A$.  Also let $\B_{C(A)}$ be a homogeneous basis for the cocenter $C(A)$ and $\{\cocenter{b}^\vee : \cocenter{b} \in \B_{C(A)}\}$ be the dual basis for $Z(A)$.

View $\kk[x_1,\dotsc,x_n]$ as a graded superalgebra with each $x_i$ being even of degree $2d$.  Let $P_n(A)$ denote the graded superalgebra $A^{\otimes n} \otimes \kk[x_1,\dotsc,x_n]$, i.e.\ the algebra of polynomials in the commuting variables $x_1,\dotsc,x_n$ with coefficients in $A^{\otimes n}$.
The group $\fS_n$ acts on $P_n(A)$ by automorphisms, permuting the variables $x_1,\dotsc,x_n$ and
the tensor factors of $A^{\otimes n}$ (with the usual sign convention).

For $1 \le i < j \le n$, define
\begin{equation} \label{tau}
    \tau_{i,j}
    := \sum_{b \in \B_A} 1^{\otimes (n-j)} \otimes b \otimes 1^{\otimes (j-i-1)} \otimes b^\vee \otimes 1^{\otimes (i-1)}.
\end{equation}
Note that the $\tau_{i,j}$ do not depend on the choice of basis $\B_A$.  In addition, for all $a \in A$, we have
\begin{equation} \label{plane}
    \sum_{b \in \B_A} \tr (b^\vee a) b
    = a
    = \sum_{b \in \B_A} \tr (a b) b^\vee.
\end{equation}
It follows that
\begin{equation} \label{travel}
    \tau_{i,j} \ba = s_{i,j}(\ba) \tau_{i,j},
    \quad \ba \in A^{\otimes n},
\end{equation}
where $s_{i,j} \in \fS_n$ is the transposition of $i$ and $j$.  In particular, this gives that $\tau_i \ba = s_i(\ba) \tau_i$ where
\begin{equation} \label{taui}
    \tau_i
    := \tau_{i,i+1}
    = \sum_{b \in \B_A} 1^{\otimes (n-i-1)} \otimes b \otimes b^\vee \otimes 1^{\otimes (i-1)}.
\end{equation}

For $i=1,\dotsc,n-1$, we define the \emph{Demazure operator} $\partial_i \colon P_n(A) \rightarrow P_n(A)$ to be the even degree zero linear map defined by
\begin{equation} \label{Demazure2}
    \partial_i (\ba p)
    := \tau_i \ba \frac{p-s_i(p)}{x_{i+1}-x_i} = s_i(\ba) \tau_i
    \frac{p-s_i(p)}{x_{i+1}-x_i},\quad
    \ba \in A^{\otimes n}, p \in \kk[x_1,\dotsc,x_n].
\end{equation}
Equivalently, we have that
\begin{equation} \label{Demazure}
    \partial_i(f) = \frac{\tau_i f - s_i(f) \tau_i}{x_{i+1}-x_i},\quad f \in P_n(A).
\end{equation}
Note also that the Demazure operators satisfy the twisted Leibniz identity
\begin{equation}
    \partial_i(fg) = \partial_i(f) g + s_i(f) \partial_i(g),\quad f,g \in P_n(A).
\end{equation}
(See also \cite[\S4.1]{Sav20}.)

The \emph{affine wreath product algebra} $\AWA{n}$ is the vector space $P_n(A) \otimes \kk \fS_n$ viewed as a graded superalgebra with multiplication defined so that $P_n(A)$ and $\kk \fS_n$ are subalgebras and
\begin{equation} \label{sif}
    s_i f = s_i(f) s_i + \partial_i(f),\quad
    f \in P_n(A),\ i \in \{1,2,\dotsc,n-1\}.
\end{equation}
By \cite[Th.~4.6]{Sav20}, it is also the case that $\AWA{n} \cong
\kk[x_1,\dotsc,x_n] \otimes_\kk \WA{n}$ as a graded vector superspace,
with both tensor factors being subalgebras.
Let
\begin{align} \label{Pinv}
    P_n(A)^{\fS_n\inv}
    &:= \left\{f \in P_n(A) : \sigma(f) = f \text{ for all }\sigma \in \fS_n\right\},
    \\ \label{Panti}
    P_n(A)^{\fS_n\anti}
    &:= \left\{f \in P_n(A) : \sigma(f) = (-1)^{\ell(\sigma)} f \text{ for all }\sigma \in \fS_n\right\}.
\end{align}
Given a homogeneous subspace $B \leq A$, we use the similar notation $P_n(B), P_n(B)^{\fS_n\inv}$ and $P_n(B)^{\fS_n\anti}$ for the analogously-defined subspaces of $P_n(A)$.
By \cite[Th.~4.14]{Sav20}, the center of $\AWA{n}$ in the sense of \cref{center} is the
subalgebra $P_n(Z(A))^{\fS_n\inv}$ of $\fS_n$-invariants in $P_n(Z(A))$, and $\AWA{n}$ is finitely generated as a module over its center.

It follows immediately from \cref{sif} that, in $\AWA{n}$, we have
\begin{equation} \label{sunroof}
  \begin{split}
    s_i f e_{(n)} = (s_i \actplus f) e_{(n)} \quad \text{where} \quad
    s_i \actplus f := s_i(f) + \partial_i(f),
    \\
    s_i f e_{(1^n)} = -(s_i \actminus f) e_{(1^n)} \quad \text{where} \quad
    s_i \actminus f := s_i(f) - \partial_i(f).
  \end{split}
\end{equation}
Transporting the left $\AWA{n}$-action through the linear isomorphism
$P_n(A) \xrightarrow{\cong} \AWA{n} e_{(n)}, f \mapsto f e_{(n)}$, we see that
$P_n(A)$ is a left $\AWA{n}$-module with $P_n(A)$ acting by left
multiplication and $\fS_n$ acting by $\actplus$.  We call this the
\emph{polynomial representation} of $\AWA{n}$.
Similarly, we can transport the $\AWA{n}$-action through the
isomorphism
$P_n(A) \xrightarrow{\cong} \AWA{n} e_{(1^n)}, f \mapsto f e_{(1^n)}$
to obtain a left action of $\AWA{n}$ on $P_n(A)$, with $\fS_n$ acting
by $\actminus$.
Note that, for $\sigma \in \fS_n$, $\ba \in A^{\otimes n}$, and $f \in \kk[x_1,\dotsc,x_n]$, we have
\[
  \sigma \actplus (\ba f) = \sigma(\ba) \left( \sigma \actplus f \right)
  \quad \text{and} \quad
  \sigma \actminus (\ba f) = \sigma(\ba) \left( \sigma \actminus f \right).
\]
Furthermore,
\begin{equation} \label{happy}
  e_{(n)} f e_{(n)}
  = e_{(n)} \left( \frac{1}{n!} \sum_{\sigma \in \fS_n} \sigma \actplus f \right) e_{(n)}
\end{equation}
and
\begin{equation} \label{sad}
  e_{(1^n)} f e_{(1^n)}
  = e_{(1^n)} \left( \frac{1}{n!} \sum_{\sigma \in \fS_n} \sigma \actminus f \right) e_{(1^n)}.
\end{equation}
The $\actplus$- and $\actminus$-actions extend to define actions
of $\fS_n$ on  $A^{\otimes n} \otimes \kk(x_1,\dotsc,x_n)$ with the simple
 transpositions satisfying the same formulae
from \cref{sunroof}.

\begin{lem} \label{dog}
    For $n,r \in \N$, the maps
    \begin{align*}
        P_n(A)^{\fS_n\inv}
        &\to e_{(n)} P_n(A) e_{(n)},\quad f \mapsto e_{(n)} f e_{(n)},
        \\
        P_n(A)^{\fS_n\anti}
        &\to e_{(1^n)} P_n(A) e_{(n)},\quad f \mapsto e_{(1^n)} f e_{(n)},
    \end{align*}
    are isomorphisms of graded vector superspaces.  The same is true when
    $A$ is replaced everywhere by $e A e'$ for homogeneous idempotents
    $e, e' \in A$.
\end{lem}

\begin{proof}
    Since the operator $\partial_i$ lowers polynomial degree, we see that, for any $f \in P_n(A)$ that is homogeneous in the $x_i$,
    \[
        e_{(n)} f e_{(n)}
        \overset{\cref{happy}}{=} e_{(n)} \left( \frac{1}{n!} \sum_{\sigma \in \fS_n} \sigma \actplus f \right) e_{(n)}
        = e_{(n)} \left( \left( \frac{1}{n!} \sum_{\sigma \in \fS_n} \sigma(f) \right) + \text{terms of lower polynomial degree} \right) e_{(n)}.
    \]
    The fact that the first map in the statement of the lemma is an   isomorphism of graded vector spaces then follows by induction on polynomial degree.  The argument for the second map is analogous.  The final assertion in the lemma follows by multiplying each side on the left by $e^{\otimes n}$ and on the right by ${e'}^{\otimes n}$.
\end{proof}

The \emph{affine wreath product category} $\AWC$ is a strict graded monoidal supercategory whose nonzero morphism spaces are the graded superalgebras $\AWA{n}$ for all $n \geq 0$. Formally, it is obtained from the category $\WC$ from the previous section by adjoining one additional generating morphism
\begin{equation} \label{upgen}
    \begin{tikzpicture}[anchorbase]
        \draw[->] (0,0) -- (0,0.6);
        \blackdot{0,0.3};
    \end{tikzpicture}
    \colon \uparrow \to \uparrow
\end{equation}
which is even of degree $2d$, imposing the additional relations
\begin{align} \label{affrel}
    \begin{tikzpicture}[anchorbase]
        \draw[->] (-0.3,-0.3) -- (0.3,0.3);
        \draw[->] (0.3,-0.3) -- (-0.3,0.3);
        \blackdot{-0.12,0.12};
    \end{tikzpicture}
    \ -\
    \begin{tikzpicture}[anchorbase]
        \draw[->] (-0.3,-0.3) -- (0.3,0.3);
        \draw[->] (0.3,-0.3) -- (-0.3,0.3);
        \blackdot{0.15,-0.15};
    \end{tikzpicture}
    &= \sum_{b \in \B_A}
    \begin{tikzpicture}[anchorbase]
        \draw[->] (-0.2,-0.3) -- (-0.2,0.3);
        \draw[->] (0.2,-0.3) -- (0.2,0.3);
        \blacktoken[east]{-0.2,0}{b};
        \blacktoken[west]{0.2,0}{b^\vee};
    \end{tikzpicture}
    \ ,\\\label{affrel1}
    \begin{tikzpicture}[anchorbase]
        \draw[->] (0,0) -- (0,0.7);
        \blackdot{0,0.2};
        \blacktoken[east]{0,0.4}{a};
    \end{tikzpicture}
    &=
    \begin{tikzpicture}[anchorbase]
        \draw[->] (0,0) -- (0,0.7);
        \blackdot{0,0.4};
        \blacktoken[west]{0,0.2}{a};
    \end{tikzpicture}
    \ ,\quad a \in A.
\end{align}
We refer to the morphisms
$
  \begin{tikzpicture}[anchorbase]
    \draw[->] (0,0) -- (0,0.4);
    \blackdot{0,0.2};
  \end{tikzpicture}
$
as \emph{dots}.  It is immediate from this definition that there is an isomorphism of graded superalgebras
\begin{equation} \label{imathaff}
    \imath_n \colon \AWA{n} \to \End_{\AWC}(\uparrow^{\otimes n})
\end{equation}
sending $s_i$ to a crossing of the $i$-th and $(i+1)$-st strings, $x_j$ to a dot on the $j$-th string, and $1^{\otimes {(n-j)}} \otimes a \otimes 1^{\otimes (j-1)}$ to a token labeled $a$ on the $j$-th
string.  As before, we number strings by $1,2,\dotsc$ from right to left.  Using $\imath_n$, we can \emph{identify} the algebra $\AWA{n}$ with $\End_{\AWC}(\uparrow^{\otimes n})$.

We refer to the morphisms in $\AWC$ defined by the elements $\tau_{i,j} \in \AWA{n}_{2d}$ from \cref{tau} as \emph{teleporters}, and denote them by a line segment connecting tokens on strings $i$ and $j$:
\begin{equation}\label{brexit}
    \begin{tikzpicture}[anchorbase]
        \draw[->] (0,-0.4) --(0,0.4);
        \draw[->] (0.5,-0.4) -- (0.5,0.4);
        \teleport{black}{black}{0,0}{0.5,0};
    \end{tikzpicture}
    =
    \begin{tikzpicture}[anchorbase]
        \draw[->] (0,-0.4) --(0,0.4);
        \draw[->] (0.5,-0.4) -- (0.5,0.4);
        \teleport{black}{black}{0,0.2}{0.5,-0.2};
    \end{tikzpicture}
    =
    \begin{tikzpicture}[anchorbase]
        \draw[->] (0,-0.4) --(0,0.4);
        \draw[->] (0.5,-0.4) -- (0.5,0.4);
        \teleport{black}{black}{0,-0.2}{0.5,0.2};
    \end{tikzpicture}
    = \sum_{b \in \B_A}
    \begin{tikzpicture}[anchorbase]
        \draw[->] (0,-0.4) --(0,0.4);
        \draw[->] (0.5,-0.4) -- (0.5,0.4);
        \blacktoken[east]{0,0.15}{b};
        \blacktoken[west]{0.5,-0.15}{b^\vee};
    \end{tikzpicture}
    = \sum_{b \in \B_A}
    \begin{tikzpicture}[anchorbase]
        \draw[->] (0,-0.4) --(0,0.4);
        \draw[->] (0.5,-0.4) -- (0.5,0.4);
        \blacktoken[east]{0,-0.15}{b^\vee};
        \blacktoken[west]{0.5,0.15}{b};
    \end{tikzpicture}\ .
\end{equation}
Note that we do not insist that the tokens in a teleporter are drawn at the same horizontal level, the convention when this is not the case being that $b$ is on the higher of the tokens and $b^\vee$ is on the lower one.  We will also draw teleporters in larger diagrams.  When doing so, we add a sign of $(-1)^{y \bar b}$ in front of the $b$ summand in \cref{brexit}, where $y$ is the sum of the parities of all morphisms in the diagram vertically between the tokens labeled $b$ and $b^\vee$.  For example,
\[
    \begin{tikzpicture}[anchorbase]
        \draw[->] (-1,-0.4) -- (-1,0.4);
        \draw[->] (-0.5,-0.4) -- (-0.5,0.4);
        \draw[->] (0,-0.4) -- (0,0.4);
        \draw[->] (0.5,-0.4) -- (0.5,0.4);
        \blacktoken[east]{-1,0}{a};
        \blacktoken[west]{0,0.1}{c};
        \teleport{black}{black}{-0.5,0.2}{0.5,-0.2};
    \end{tikzpicture}
    =
    \sum_{b \in \B_A} (-1)^{(\bar a + \bar c) \bar b}
    \begin{tikzpicture}[anchorbase]
        \draw[->] (-1,-0.4) -- (-1,0.4);
        \draw[->] (-0.5,-0.4) -- (-0.5,0.4);
        \draw[->] (0,-0.4) -- (0,0.4);
        \draw[->] (0.5,-0.4) -- (0.5,0.4);
        \blacktoken[east]{-1,0}{a};
        \blacktoken[west]{0,0.1}{c};
        \blacktoken[east]{-0.5,0.2}{b};
        \blacktoken[west]{0.5,-0.2}{b^\vee};
    \end{tikzpicture}
    \ .
\]
This convention ensures that one can slide the endpoints of teleporters along strands:
\[
    \begin{tikzpicture}[anchorbase]
        \draw[->] (-1,-0.4) -- (-1,0.4);
        \draw[->] (-0.5,-0.4) -- (-0.5,0.4);
        \draw[->] (0,-0.4) -- (0,0.4);
        \draw[->] (0.5,-0.4) -- (0.5,0.4);
        \blacktoken[east]{-1,0}{a};
        \blacktoken[west]{0,0.1}{c};
        \teleport{black}{black}{-0.5,0.2}{0.5,-0.2};
    \end{tikzpicture}
    =
    \begin{tikzpicture}[anchorbase]
        \draw[->] (-1,-0.4) -- (-1,0.4);
        \draw[->] (-0.5,-0.4) -- (-0.5,0.4);
        \draw[->] (0,-0.4) -- (0,0.4);
        \draw[->] (0.5,-0.4) -- (0.5,0.4);
        \blacktoken[east]{-1,0}{a};
        \blacktoken[east]{0,0.1}{c};
        \teleport{black}{black}{-0.5,-0.2}{0.5,0.2};
    \end{tikzpicture}
    =
    \begin{tikzpicture}[anchorbase]
        \draw[->] (-1,-0.4) -- (-1,0.4);
        \draw[->] (-0.5,-0.4) -- (-0.5,0.4);
        \draw[->] (0,-0.4) -- (0,0.4);
        \draw[->] (0.5,-0.4) -- (0.5,0.4);
        \blacktoken[east]{-1,0}{a};
        \blacktoken[west]{0,0.1}{c};
        \teleport{black}{black}{-0.5,0.2}{0.5,0.2};
    \end{tikzpicture}
    =
    \begin{tikzpicture}[anchorbase]
        \draw[->] (-1,-0.4) -- (-1,0.4);
        \draw[->] (-0.5,-0.4) -- (-0.5,0.4);
        \draw[->] (0,-0.4) -- (0,0.4);
        \draw[->] (0.5,-0.4) -- (0.5,0.4);
        \blacktoken[east]{-1,0}{a};
        \blacktoken[west]{0,0.1}{c};
        \teleport{black}{black}{-0.5,-0.2}{0.5,-0.2};
    \end{tikzpicture}
    \ .
\]
Using teleporters, the relation \cref{affrel} can be written as
\begin{equation} \label{upslides1}
    \begin{tikzpicture}[anchorbase]
        \draw[->] (-0.3,-0.3) -- (0.3,0.3);
        \draw[->] (0.3,-0.3) -- (-0.3,0.3);
        \blackdot{-0.12,0.12};
    \end{tikzpicture}
    \ -\
    \begin{tikzpicture}[anchorbase]
        \draw[->] (-0.3,-0.3) -- (0.3,0.3);
        \draw[->] (0.3,-0.3) -- (-0.3,0.3);
        \blackdot{0.15,-0.15};
    \end{tikzpicture}
    \ =\
    \begin{tikzpicture}[anchorbase]
        \draw[->] (-0.2,-0.3) -- (-0.2,0.3);
        \draw[->] (0.2,-0.3) -- (0.2,0.3);
        \teleport{black}{black}{-0.2,0}{0.2,0};
    \end{tikzpicture}
    \ .
\end{equation}
Composing with the crossing on the top and bottom, we see that we also have the relation
\begin{equation} \label{upslides2}
    \begin{tikzpicture}[anchorbase]
        \draw[->] (-0.3,-0.3) -- (0.3,0.3);
        \draw[->] (0.3,-0.3) -- (-0.3,0.3);
        \blackdot{-0.12,-0.12};
    \end{tikzpicture}
    \ -\
    \begin{tikzpicture}[anchorbase]
        \draw[->] (-0.3,-0.3) -- (0.3,0.3);
        \draw[->] (0.3,-0.3) -- (-0.3,0.3);
        \blackdot{0.15,0.15};
    \end{tikzpicture}
    \ =\
    \begin{tikzpicture}[anchorbase]
        \draw[->] (-0.2,-0.3) -- (-0.2,0.3);
        \draw[->] (0.2,-0.3) -- (0.2,0.3);
        \teleport{black}{black}{-0.2,0}{0.2,0};
    \end{tikzpicture}
    \ .
\end{equation}
As in \cref{travel}, tokens can ``teleport'' across teleporters (justifying the terminology) in the sense that, for $a \in A$, we have
\begin{equation} \label{tokteleport}
    \begin{tikzpicture}[anchorbase]
        \draw[->] (0,-0.5) --(0,0.5);
        \draw[->] (0.5,-0.5) -- (0.5,0.5);
        \blacktoken[west]{0.5,-0.25}{a};
        \teleport{black}{black}{0,0}{0.5,0};
    \end{tikzpicture}
    =
    \begin{tikzpicture}[anchorbase]
        \draw[->] (0,-0.5) --(0,0.5);
        \draw[->] (0.5,-0.5) -- (0.5,0.5);
        \blacktoken[east]{0,0.25}{a};
        \teleport{black}{black}{0,0}{0.5,0};
    \end{tikzpicture}
    \ ,\qquad
    \begin{tikzpicture}[anchorbase]
        \draw[->] (0,-0.5) --(0,0.5);
        \draw[->] (0.5,-0.5) -- (0.5,0.5);
        \blacktoken[east]{0,-0.25}{a};
        \teleport{black}{black}{0,0}{0.5,0};
    \end{tikzpicture}
    =
    \begin{tikzpicture}[anchorbase]
        \draw[->] (0,-0.5) --(0,0.5);
        \draw[->] (0.5,-0.5) -- (0.5,0.5);
        \blacktoken[west]{0.5,0.25}{a};
        \teleport{black}{black}{0,0}{0.5,0};
    \end{tikzpicture}
    \ ,
\end{equation}
where the strings can occur anywhere in a diagram (i.e.\ they do not need to be adjacent).  The endpoints of teleporters slide through crossings and dots, and they can teleport too.  For example we have
\begin{equation} \label{laser}
  \begin{tikzpicture}[anchorbase]
    \draw[->] (-0.4,-0.4) to (0.4,0.4);
    \draw[->] (0.4,-0.4) to (-0.4,0.4);
    \draw[->] (0.8,-0.4) to (0.8,0.4);
    \teleport{black}{black}{-0.15,0.15}{0.8,0.15};
  \end{tikzpicture}
  \ =\
  \begin{tikzpicture}[anchorbase]
    \draw[->] (-0.4,-0.4) to (0.4,0.4);
    \draw[->] (0.4,-0.4) to (-0.4,0.4);
    \draw[->] (0.8,-0.4) to (0.8,0.4);
    \teleport{black}{black}{0.2,-0.2}{0.8,-0.2};
  \end{tikzpicture}
  \ ,\qquad
  \begin{tikzpicture}[anchorbase]
    \draw[->] (-0.2,-0.5) to (-0.2,0.5);
    \draw[->] (0.2,-0.5) to (0.2,0.5);
    \blackdot{-0.2,0};
    \teleport{black}{black}{-0.2,-0.25}{0.2,0};
  \end{tikzpicture}
  \ =\
  \begin{tikzpicture}[anchorbase]
    \draw[->] (-0.2,-0.5) to (-0.2,0.5);
    \draw[->] (0.2,-0.5) to (0.2,0.5);
    \blackdot{-0.2,0};
    \teleport{black}{black}{-0.2,0.25}{0.2,0};
  \end{tikzpicture}
  \ ,\qquad
  \begin{tikzpicture}[anchorbase]
    \draw[->] (-0.4,-0.5) to (-0.4,0.5);
    \draw[->] (0,-0.5) to (0,0.5);
    \draw[->] (0.4,-0.5) to (0.4,0.5);
    \teleport{black}{black}{-0.4,-0.25}{0,-0.25};
    \teleport{black}{black}{0,0}{0.4,0};
  \end{tikzpicture}
  \ =\
  \begin{tikzpicture}[anchorbase]
    \draw[->] (-0.4,-0.5) to (-0.4,0.5);
    \draw[->] (0,-0.5) to (0,0.5);
    \draw[->] (0.4,-0.5) to (0.4,0.5);
    \teleport{black}{black}{-0.4,-0.2}{0.4,0.3};
    \teleport{black}{black}{0,-.1}{0.4,-.1};
  \end{tikzpicture}
  \ =\
  \begin{tikzpicture}[anchorbase]
    \draw[->] (-0.4,-0.5) to (-0.4,0.5);
    \draw[->] (0,-0.5) to (0,0.5);
    \draw[->] (0.4,-0.5) to (0.4,0.5);
    \teleport{black}{black}{-0.4,0.15}{0.4,0.15};
    \teleport{black}{black}{0,-.1}{0.4,-.1};
  \end{tikzpicture}
    \ .
\end{equation}

The next goal is to construct an analogue for $\AWC$ of the monoidal functor $\Delta$ from \cref{stroppel}.  The symmetric product ${\color{blue} \AWC} \odot {\color{red} \AWC}$ can be defined in the same way as in the previous section, adjoining degree zero isomorphisms represented by the mutually-inverse two-color crossings \cref{referencethislatertoo}; these are required also to commute with the colored dots.  However, for the definition of $\Delta$ for arbitrary $A$, it seems to be essential to work with a different grading on this symmetric product category; see \cref{stroppel2} below for further discussion.  Recalling that the trace has degree $-2d$, we denote this modified symmetric product by ${\color{blue} \AWC} \xodot{2d} {\color{red} \AWC}$. It has the same underlying monoidal supercategory as ${\color{blue} \AWC} \odot {\color{red} \AWC}$ but now, rather than being of degree zero, the two-color crossings are even of degrees
\begin{equation}\label{par}
    \deg
    \left( \:
    \begin{tikzpicture}[anchorbase]
        \draw[->,blue] (-0.4,-0.4) to (0.4,0.4);
        \draw[->,red] (0.4,-0.4) to (-0.4,0.4);
    \end{tikzpicture}
    \:\right)
    = 2\dA, \qquad
    \deg
    \left( \:
    \begin{tikzpicture}[anchorbase]
        \draw[->,red] (-0.4,-0.4) to (0.4,0.4);
        \draw[->,blue] (0.4,-0.4) to (-0.4,0.4);
    \end{tikzpicture}
    \:\right)
    = - 2\dA.
\end{equation}
Teleporters can be defined in the same way as \cref{brexit}, including in situations where the token at one end is on a red string and the token at the other end is on a blue string. The two-colored teleporters satisfy analogous properties to \cref{tokteleport} (for the same reasons).

We also need to localize in a similar way to  \cite[Sec.\ 4]{BSW-K0}.  Let
\begin{equation} \label{dashdumb}
  \begin{tikzpicture}[anchorbase]
    \draw[->,blue] (-0.2,-0.4) to (-0.2,0.4);
    \draw[->,red] (0.2,-0.4) to (0.2,0.4);
    \dashdumb{-0.2,0}{0.2,0};
  \end{tikzpicture}
  \ :=\
  \begin{tikzpicture}[anchorbase]
    \draw[->,blue] (-0.2,-0.4) to (-0.2,0.4);
    \draw[->,red] (0.2,-0.4) to (0.2,0.4);
    \reddot{0.2,0};
  \end{tikzpicture}
  \ -\
  \begin{tikzpicture}[anchorbase]
    \draw[->,blue] (-0.2,-0.4) to (-0.2,0.4);
    \draw[->,red] (0.2,-0.4) to (0.2,0.4);
    \bluedot{-0.2,0};
  \end{tikzpicture}
  \qquad \text{and} \qquad
  \begin{tikzpicture}[anchorbase]
    \draw[->,red] (-0.2,-0.4) to (-0.2,0.4);
    \draw[->,blue] (0.2,-0.4) to (0.2,0.4);
    \dashdumb{-0.2,0}{0.2,0};
  \end{tikzpicture}
  \ :=\
  \begin{tikzpicture}[anchorbase]
    \draw[->,red] (-0.2,-0.4) to (-0.2,0.4);
    \draw[->,blue] (0.2,-0.4) to (0.2,0.4);
    \reddot{-0.2,0};
  \end{tikzpicture}
  \ -\
  \begin{tikzpicture}[anchorbase]
    \draw[->,red] (-0.2,-0.4) to (-0.2,0.4);
    \draw[->,blue] (0.2,-0.4) to (0.2,0.4);
    \bluedot{0.2,0};
  \end{tikzpicture}
  \ .
\end{equation}
Then we define $\blue{\AWC} \barodot \red{\AWC}$ to be the strict graded monoidal supercategory obtained by localizing $\blue{\AWC} \xodot{2d} \red{\AWC}$ at these morphisms.  This means that we adjoin new morphisms, the \emph{dot dumbbells}, which we declare to be two-sided inverses to \cref{dashdumb}:
\begin{equation} \label{dotdumb}
  \begin{tikzpicture}[anchorbase]
    \draw[->,blue] (-0.2,-0.4) to (-0.2,0.4);
    \draw[->,red] (0.2,-0.4) to (0.2,0.4);
    \dotdumb{-0.2,0}{0.2,0};
  \end{tikzpicture}
  \ :=
  \left(
    \begin{tikzpicture}[anchorbase]
      \draw[->,blue] (-0.2,-0.4) to (-0.2,0.4);
      \draw[->,red] (0.2,-0.4) to (0.2,0.4);
      \dashdumb{-0.2,0}{0.2,0};
    \end{tikzpicture}
  \right)^{-1}
  \qquad \text{and} \qquad
  \begin{tikzpicture}[anchorbase]
    \draw[->,red] (-0.2,-0.4) to (-0.2,0.4);
    \draw[->,blue] (0.2,-0.4) to (0.2,0.4);
    \dotdumb{-0.2,0}{0.2,0};
  \end{tikzpicture}
  \ :=
  \left(
    \begin{tikzpicture}[anchorbase]
      \draw[->,red] (-0.2,-0.4) to (-0.2,0.4);
      \draw[->,blue] (0.2,-0.4) to (0.2,0.4);
      \dashdumb{-0.2,0}{0.2,0};
    \end{tikzpicture}
  \right)^{-1}
  .
\end{equation}
In fact, one only needs to require that one of these morphisms is invertible, since that implies the invertibility of the other one too. The dot dumbbells are even of degree $-2\dA$ and we have
\begin{equation} \label{dotteleport}
  \begin{tikzpicture}[anchorbase]
    \draw[->,blue] (-0.2,-0.4) to (-0.2,0.4);
    \draw[->,red] (0.2,-0.4) to (0.2,0.4);
    \reddot{0.2,0.15};
    \dotdumb{-0.2,-0.15}{0.2,-0.15};
  \end{tikzpicture}
  \ =\
  \begin{tikzpicture}[anchorbase]
    \draw[->,blue] (-0.2,-0.4) to (-0.2,0.4);
    \draw[->,red] (0.2,-0.4) to (0.2,0.4);
    \bluedot{-0.2,0.15};
    \dotdumb{-0.2,-0.15}{0.2,-0.15};
  \end{tikzpicture}
  \ +\
  \begin{tikzpicture}[anchorbase]
    \draw[->,blue] (-0.2,-0.4) to (-0.2,0.4);
    \draw[->,red] (0.2,-0.4) to (0.2,0.4);
  \end{tikzpicture}
  \quad \text{or, more generally,} \quad
  \begin{tikzpicture}[anchorbase]
    \draw[->,blue] (-0.2,-0.4) to (-0.2,0.4);
    \draw[->,red] (0.2,-0.4) to (0.2,0.4);
    \multreddot[west]{0.2,0.15}{n};
    \dotdumb{-0.2,-0.15}{0.2,-0.15};
  \end{tikzpicture}
  \ =\
  \begin{tikzpicture}[anchorbase]
    \draw[->,blue] (-0.2,-0.4) to (-0.2,0.4);
    \draw[->,red] (0.2,-0.4) to (0.2,0.4);
    \multbluedot[east]{-0.2,0.15}{n};
    \dotdumb{-0.2,-0.15}{0.2,-0.15};
  \end{tikzpicture}
  \ + \sum_{\substack{r,s \ge 0 \\ r + s = n-1}}
  \begin{tikzpicture}[anchorbase]
    \draw[->,blue] (-0.2,-0.4) to (-0.2,0.4);
    \draw[->,red] (0.2,-0.4) to (0.2,0.4);
    \multbluedot[east]{-0.2,0}{r};
    \multreddot[west]{0.2,0}{s};
  \end{tikzpicture}
  \ .
\end{equation}
We also define the \emph{box dumbbells}
\begin{equation} \label{sqdumb}
  \begin{tikzpicture}[anchorbase]
    \draw[->,red] (-0.2,-0.4) to (-0.2,0.4);
    \draw[->,blue] (0.2,-0.4) to (0.2,0.4);
    \sqdumb{-0.2,0}{0.2,0};
  \end{tikzpicture}
  \ :=\
  \begin{tikzpicture}[anchorbase]
    \draw[->,red] (-0.2,-0.4) to (-0.2,0.4);
    \draw[->,blue] (0.2,-0.4) to (0.2,0.4);
    \dashdumb{-0.2,0}{0.2,0};
  \end{tikzpicture}
  \ -
  \begin{tikzpicture}[anchorbase]
    \draw[->,red] (-0.2,-0.4) to (-0.2,0.4);
    \draw[->,blue] (0.2,-0.4) to (0.2,0.4);
    \dotdumb{-0.2,0.15}{0.2,0.15};
    \teleport{red}{blue}{-0.2,-0.05}{0.2,-0.05};
    \teleport{red}{blue}{-0.2,-0.25}{0.2,-0.25};
  \end{tikzpicture}
  \quad \text{and} \quad
  \begin{tikzpicture}[anchorbase]
    \draw[->,blue] (-0.2,-0.4) to (-0.2,0.4);
    \draw[->,red] (0.2,-0.4) to (0.2,0.4);
    \sqdumb{-0.2,0}{0.2,0};
  \end{tikzpicture}
  \ :=\
  \begin{tikzpicture}[anchorbase]
    \draw[->,blue] (-0.2,-0.4) to (-0.2,0.4);
    \draw[->,red] (0.2,-0.4) to (0.2,0.4);
    \dashdumb{-0.2,0}{0.2,0};
  \end{tikzpicture}
  \ -
  \begin{tikzpicture}[anchorbase]
    \draw[->,blue] (-0.2,-0.4) to (-0.2,0.4);
    \draw[->,red] (0.2,-0.4) to (0.2,0.4);
    \dotdumb{-0.2,0.15}{0.2,0.15};
    \teleport{blue}{red}{-0.2,-0.05}{0.2,-0.05};
    \teleport{blue}{red}{-0.2,-0.25}{0.2,-0.25};
  \end{tikzpicture}
  \ ,
\end{equation}
which are even of degree $2\dA$.  As with teleporters, we will often draw dumbbells in more general positions than the ones displayed above.  (See \cite[Sec.\ 4]{BSW-K0} for further discussion.)

Similarly to the previous section, we view $\AWA{n-r} \otimes \AWA{r}$ as a subalgebra of $\AWA{n}$ via the identifications
\begin{gather*}
      1 \otimes s_i \leftrightarrow s_i, \
      s_j \otimes 1 \leftrightarrow s_{r+j}, \
      1 \otimes x_i \leftrightarrow x_i, \
      x_j \otimes 1 \leftrightarrow x_{r+j}, \
      1 \otimes \ba \leftrightarrow 1^{\otimes (n-r)} \otimes \ba, \
      \bb \otimes 1 \leftrightarrow \bb \otimes 1^{\otimes r}.
\end{gather*}
Let $\AWA{n-r} \barotimes \AWA{r}$  denote the Ore localization of $\AWA{n-r} \otimes \AWA{r}$ at the central element
\[
    \prod_{i=1}^r \prod_{j=1}^{n-r} (x_i - x_{r+j}).
\]
This localization is the algebraic counterpart of introducing the dot dumbbells \cref{dotdumb}.  Indeed, generalizing \cref{imathaff}, we have an isomorphism of graded superalgebras
\begin{equation}
    \imath_{r,n} \colon \AWA{n-r} \barotimes \AWA{r} \to \End_{{\color{blue} \AWC} \barodot {\color{red} \AWC}} \left( \upblue^{\otimes (n-r)} \otimes \upred^{\otimes r} \right).
\end{equation}
Combining this with the elements $\{ \sigma_\lambda^{-1} \circ \sigma_\mu : \lambda, \mu \in \cP_{r,n} \}$ (recall the definition of the $\sigma_\lambda$ in \cref{organize}), which give the matrix units, we see that
\begin{equation} \label{shipped}
    \End_{\Add(\blue{\AWC} \barodot \red{\AWC})} \left( (\upblue \oplus \upred)^{\oplus n} \right)
    \cong \bigoplus_{r=0}^n \mathrm{Mat}_{\binom{n}{r}} \left( \AWA{n-r} \barotimes \AWA{r} \right).
\end{equation}

One can forget the grading on $\blue{\AWC} \barodot \red{\AWC}$ to view it as a monoidal supercategory rather than as a graded monoidal supercategory.  We will denote this simply by $\blue{\AWC} \barodotnonumber \red{\AWC}$.  There is a strict monoidal superfunctor
\begin{equation}\label{flippy}
    \flip \colon
    \blue{\AWC} \barodotnonumber \red{\AWC}\to
    \blue{\AWC} \barodotnonumber \red{\AWC}
\end{equation}
defined on diagrams by switching the colors blue and red, then multiplying by $(-1)^w$, where $w$ is the total number of dot dumbbells in the picture.

Dots, tokens, and teleporters commute with all (i.e.\ both dot and box) dumbbells.  All dumbbells commute with each other and all dumbbells commute past two-color crossings.  For instance, we have
\begin{gather*}

  \ ,
\end{gather*}
and similarly for dot dumbbells.  More generally, define
\[
  %
  \ ,
\]
and similarly for other orders of the strings.  Then we have
\begin{gather} \label{hops1}
  %
  \ .
\end{gather}
It follows that
\begin{equation} \label{dnb}
  %
  \ .
\end{equation}

\begin{theo} \label{AWcoprod}
  There is a strict graded monoidal superfunctor
  \[
    \Delta \colon \AWC \to \Add \left( {\color{blue} \AWC} \barodot {\color{red} \AWC} \right)
  \]
  such that $\Delta(\uparrow) = \upblue \oplus \upred$ and
  \begin{gather} \nonumber
    \Delta \left(\:
      %
    \ .
  \end{gather}
\end{theo}

\begin{proof}
  We need to verify that $\Delta$ preserves the relations \cref{wreathrel,affrel1,upslides1}.  Relation \cref{affrel1} and the first three relations in \cref{wreathrel} are immediate.  The fourth relation in \cref{wreathrel} follows easily using \cref{tokteleport}.  The image under $\Delta$ of the crossing squared is
  \[
    %
    \ ,
  \]
  which implies that $\Delta$ preserves the sixth relation of \cref{wreathrel}.  Using the fact that dots commute with dumbbells, the image under $\Delta$ of the left hand side of \cref{upslides1} is
  \[
    %
    \ ,
  \]
  which is the image of the right-hand side of \cref{upslides1}.

  It remains to show that $\Delta$ preserves the braid relation.  The image under $\Delta$ of the two sides of this relation are $8 \times 8$ matrices, corresponding to the 8 choices of colors of the incoming and outgoing strings.  A component is zero unless the number of incoming red strings is equal to the number of outgoing red strings.  In addition, the single-colored components are clearly equal, using the braid relation in ${\color{blue} \AWC}$ and ${\color{red} \AWC}$.  Thus, it suffices to consider the multi-colored components.  First consider the component $\upred \upblue \upblue \to \upblue \upblue \upred$.  Denote this component by $\Delta_{\upred \upblue \upblue}^{\upblue \upblue \upred}$.  Then we have
  \[
    \Delta_{\upred \upblue \upblue}^{\upblue \upblue \upred}
    \left(

    \right).
  \end{align*}
  The remaining 14 cases are similar.
  \details{ \color{black}
  \begin{align*}
    \Delta_{\upblue \upred \upblue}^{\upblue \upred \upblue}
    \left(
      %
    \right),
  \end{align*}
  }
\end{proof}

\begin{rem} \label{coass2}
  As in \cref{coass1}, the categorical comultiplication is coassociative.
\end{rem}

\begin{rem}\label{stroppel2}
    The categorical comultiplication $\Delta$ of \cref{AWcoprod} does \emph{not} specialize to the one of \cite[Th.~4.2]{BSW-K0} when $A=\kk$.  However, when $A$ is supercommutative, there is another way to proceed, working with the localization ${\color{blue} \AWC} \:\overline{\odot}\: {\color{red} \AWC}$ of the original symmetric product ${\color{blue} \AWC} \odot {\color{red} \AWC}$ in which the two-color crossings are even of degree zero.  Note for this category that the involution \cref{flippy} is actually a graded monoidal superfunctor.  Then one can replace \cref{Deltcross} by
    \[
        \Delta
        \left(\:
            \begin{tikzpicture}[anchorbase]
                \draw[->] (-0.4,-0.4) to (0.4,0.4);
                \draw[->] (0.4,-0.4) to (-0.4,0.4);
            \end{tikzpicture}
        \: \right)
        =
        \begin{tikzpicture}[anchorbase]
            \draw[->,blue] (-0.4,-0.4) to (0.4,0.4);
            \draw[->,blue] (0.4,-0.4) to (-0.4,0.4);
        \end{tikzpicture}
        \ +\
        \begin{tikzpicture}[anchorbase]
            \draw[->,red] (-0.4,-0.4) to (0.4,0.4);
            \draw[->,red] (0.4,-0.4) to (-0.4,0.4);
        \end{tikzpicture}
        \ +\
        \begin{tikzpicture}[anchorbase]
        \draw[->,blue] (-0.4,-0.4) to (0.4,0.4);
        \draw[->,red] (0.4,-0.4) to (-0.4,0.4);
        \end{tikzpicture}
        \ +\
        \begin{tikzpicture}[anchorbase]
            \draw[->,red] (-0.4,-0.4) to (0.4,0.4);
            \draw[->,blue] (0.4,-0.4) to (-0.4,0.4);
        \end{tikzpicture}
        \ -\
        \begin{tikzpicture}[anchorbase]
            \draw[->,blue] (-0.4,-0.4) to (0.4,0.4);
            \draw[->,red] (0.4,-0.4) to (-0.4,0.4);
            \teleport{red}{blue}{-0.2,0.2}{0.2,0.2};
            \dotdumb{-0.2,-0.2}{0.2,-0.2};
        \end{tikzpicture}
        \ +\
        \begin{tikzpicture}[anchorbase]
            \draw[->,red] (-0.4,-0.4) to (0.4,0.4);
            \draw[->,blue] (0.4,-0.4) to (-0.4,0.4);
            \teleport{blue}{red}{-0.2,0.2}{0.2,0.2};
            \dotdumb{-0.2,-0.2}{0.2,-0.2};
        \end{tikzpicture}
        \ -
        \begin{tikzpicture}[anchorbase]
            \draw[->,blue] (-0.2,-0.4) to (-0.2,0.4);
            \draw[->,red] (0.2,-0.4) to (0.2,0.4);
            \teleport{blue}{red}{-0.2,-0.2}{0.2,-0.2};
            \dotdumb{-0.2,0.15}{0.2,0.15};
        \end{tikzpicture}
        +
        \begin{tikzpicture}[anchorbase]
            \draw[->,red] (-0.2,-0.4) to (-0.2,0.4);
            \draw[->,blue] (0.2,-0.4) to (0.2,0.4);
            \teleport{red}{blue}{-0.2,-0.2}{0.2,-0.2};
            \dotdumb{-0.2,0.15}{0.2,0.15};
        \end{tikzpicture}
        \ .
  \]
    This formula is more symmetric than \cref{Deltcross} as it satisfies $\flip \circ \Delta = \Delta$. Moreover, in the case $A = \kk$, it exactly recovers the comultiplication of \cite[Th.~4.2]{BSW-K0}.  For another approach to defining these types of categorical comultiplications, which explains how they are related to intertwining operators, see \cite[\S7.3]{BSW-qFrobHeis}.
\end{rem}

\details{There is a more conceptual way to understand the definition
  of the categorical comultiplications, which is similar to the
  approach in the quantum case which will be developed in a subsequent article.}

Now we are going to discuss some results pertaining to the Grothendieck ring of the category $\AWC$; in particular, we will show that  this gives another categorification of the ring $\Sym_\Zq^{\otimes N}$.
We assume for the remainder of the section that the algebra $A$ satisfies \emph{assumption} \eqref{virginia} from the introduction. This is needed in order to be able to prove the following.

\begin{theo} \label{salad}
    Suppose that $A$ satisfies the hypothesis \eqref{virginia} from the introduction.  Then the natural inclusion $\WA{n} \hookrightarrow \AWA{n}$ induces an isomorphism of Grothendieck groups
    \[
        K_0(\pgsmod\WA{n}) \cong K_0(\pgsmod\AWA{n}).
    \]
    More generally, this assertion holds when $\WA{n}$ is replaced by
    $\WA{n_1} \otimes \dotsb \otimes \WA{n_r}$ and $\AWA{n}$ is
    replaced by $\AWA{n_1} \otimes \dotsb \otimes \AWA{n_r} \otimes \Lambda$
    for any $n_1,\dotsc,n_r \ge 0$ and any positively graded polynomial algebra
    $\Lambda$ (possibly of infinite rank) viewed as a purely even graded superalgebra.
\end{theo}

\begin{proof}
    If the grading on $A$ is positive and nontrivial, then the first statement follows, using \cite[Ch.~XII, Prop.~3.3]{Bass}, from the fact that $\deg x_i = 2\dA > 0$, which implies that $\WA{n}$ and $\AWA{n}$ have the same degree zero piece.  Instead, if the grading on $A$ is trivial, then $A=A_{0,\even}$ is semisimple by the assumption \eqref{virginia}.  Thus $A$ is Morita equivalent to $\kk^m$ for some $m \ge 1$.  Hence $\WA{n}$ and $\AWA{n}$ are Morita equivalent to $\WA[\kk^m]{n} \cong \bigoplus_{n_1 + \dotsb + n_m = n} \WA[\kk]{n_1} \otimes \dotsb \otimes \WA[\kk]{n_m}$ and $\AWA[\kk^m]{n} \cong \bigoplus_{n_1 + \dotsb + n_m = n} \AWA[\kk]{n_1} \otimes \dotsb \otimes \AWA[\kk]{n_m}$, respectively.  Since $\WA[\kk]{n} \cong \kk \fS_n$ and $\AWA[\kk]{n}$ is the degenerate affine Hecke algebra, the result then follows from \cite[Prop.~8]{Kho14}.
    \details{
        If $A=A_{0,\even}$ is semisimple, one can also argue as follows. Consider the filtration
        \[
            0 = \AWA{n}_{\le -1} \subseteq \AWA{n}_{\le 0} \subseteq \AWA{n}_{\le 1} \subseteq \dotsb
        \]
        on $\AWA{n}$, where $\AWA{n}_{\le r}$ is spanned by elements of the form $\ba x_1^{m_1} \dotsm x_n^{m_n} \sigma$ for $\ba \in A^{\otimes n}$, $\sigma \in \fS_n$, and $m_1 + \dotsb + m_n \le r$.  This gives $\AWA{n}$ the structure of a filtered superalgebra, with $\AWA{n}_{\le 0} = \WA{n}$.  The associated graded superalgebra $\gr \AWA{n}$ is isomorphic to $\WA[{A[x]}]{n} = P_n(A) \rtimes \fS_n$.  Let $\bigwedge^r(x_1,\dotsc,x_n)$ be the $r$th exterior power of the vector space $\kk x_1 \oplus \dotsb \oplus \kk x_n$ (viewed as a purely even superspace).  We have an exact sequence of right $\WA[{A[x]}]{n}$-supermodules
        \begin{multline*}
            0 \to {\textstyle\bigwedge}^n(x_1,\dotsc,x_n) \otimes \WA[{A[x]}]{n}
            \to {textstyle\bigwedge}^{n-1}(x_1,\dotsc,x_n) \otimes \WA[{A[x]}]{n}
            \to \dotsb \\
            \dotsb \to
            {\textstyle\bigwedge}^1(x_1,\dotsc,x_n) \otimes \WA[{A[x]}]{n}
            \to \WA[{A[x]}]{n}
            \xrightarrow{\rho} \WA{n},
        \end{multline*}
        where $\rho$ is the map given by $x_i \mapsto 0$, $1 \le i \le n$, and the other nonzero maps are $\sum_{i=1}^n \phi_i \otimes x_i$, where $\phi_i$ is contraction with respect to $x_i$.  In particular, $\WA{n}$ has finite Tor dimension as a right $(\gr \AWA{n})$-module.  On the other hand, $\WA{n}$ is semisimple since $A$ is.  Hence $\gr \AWA{n}$ has Tor dimension zero as a right $\WA{n}$-supermodule.  The remainder of the proof follows as in \cite[\S5.2]{Kho14}, which treats the case $A = \kk$.
    }
\end{proof}

Fix a choice of homogeneous idempotents $e_1,\dots,e_N \in A$ as we
did after \cref{coass1}.
Then we can define the idempotents \cref{berries,sugar}, which will
now be viewed as idempotents in $\AWA{n}$.
Thus, for $\blambda \in \cP^N$, $n \in \N$ and $i=1,\dotsc,N$, we have associated objects $S_\blambda, H_{n,i}$ and $E_{n,i}$ defined as in \cref{hide}, but viewed now as objects in $\Kar(\AWC_{q,\pi})$.  When the assumption \eqref{virginia} from the introduction is satisfied, we deduce using \cref{salad,yoneda} that the canonical embedding $\WC \to \AWC$ induces a $\Zq$-algebra isomorphism $K_0(\Kar(\WC_{q,\pi})) \xrightarrow{\cong} K_0(\Kar(\AWC_{q,\pi}))$.  Therefore, we can reformulate \cref{definingg}: there is a $\Zq$-algebra isomorphism
\begin{equation} \label{emerald}
    \gamma \colon \Sym_\Zq^{\otimes N} \xrightarrow{\cong} K_0 \left(\Kar\left( \AWC_{q,\pi}\right) \right),\quad
    s_{\lambda,i} \mapsto [S_{\lambda,i}],\
    h_{n,i} \mapsto [H_{n,i}],\
    e_{n,i} \mapsto [E_{n,i}],
\end{equation}
of course still assuming that the hypothesis \eqref{virginia} holds.

The canonical functors $\blue{\AWC} \to \blue{\AWC} \barodot \red{\AWC}$ and $\red{\AWC} \to \blue{\AWC} \barodot \red{\AWC}$ induce graded monoidal superfunctors between the Karoubi envelopes of the $(Q,\Pi)$-envelopes of these categories.  Hence, we get a $\Zq$-algebra homomorphism
\[
    \epsilon \colon K_0\left(\Kar\left(\blue{\AWC_{q,\pi}}\right)\right) \otimes_\Zq K_0\left(\Kar\left(\red{\AWC_{q,\pi}}\right)\right)
    \to K_0\left( \Kar \left(\left(\blue{\AWC} \barodot \red{\AWC}\right)_{q,\pi}\right)\right).
\]
Also $\Delta$ extends to $\tilde\Delta \colon \Kar(\AWC_{q,\pi}) \rightarrow \Kar((\blue{\AWC}\barodot\red{\AWC})_{q,\pi})$.  The following theorem implies that the diagram
\[
  \begin{tikzcd}
    \Sym^{\otimes N}_\Zq \arrow[r, "\delta"] \arrow[d, "\gamma"]
    &
    \Sym^{\otimes N}_\Zq \otimes_\Zq \Sym^{\otimes N}_\Zq \arrow[d, "\epsilon \circ (\gamma \otimes \gamma)"]
    \\
    K_0\left(\Kar\left(\AWC_{q,\pi}\right)\right) \arrow[r, "{[\tilde\Delta]}"]
    &
    K_0\left(\Kar\left(\left(\blue{\AWC} \barodot \red{\AWC}\right)_{q,\pi}\right)\right)
  \end{tikzcd}
\]
commutes, where $\gamma$ comes from \cref{emerald}.  Thus, again, the functor $\Delta$ categorifies the comultiplication $\delta$.

\begin{prop} \label{coprodcat}
    For each $n \ge 0$ and $1 \le i \le N$, there are degree zero isomorphisms
    \begin{align}\label{bakedbeans}
        \tilde\Delta(H_{n,i}) &\cong \bigoplus_{r=0}^n \blue{H_{n-r,i}} \otimes \red{H_{r,i}},&
        \tilde\Delta(E_{n,i}) &\cong \bigoplus_{r=0}^n \blue{E_{n-r,i}}
        \otimes \red{E_{r,i}}.
    \end{align}
\end{prop}

The proof of \cref{coprodcat} will occupy the remainder of the section.  For $\lambda \in \cP_{r,n}$ and $1 \le i \le r$, $1 \le j \le r-n$, we define the following elements of $\AWA{n-r} \barotimes \AWA{r}$:
\begin{align*}
  y_{i,j}^+ &:=
  \begin{cases}
    \tau_{r+1-i,r+j} + x_{r+1-i} - x_{r+j} & \text{if } j \le \lambda_i, \\
    1 & \text{if } j > \lambda_i,
  \end{cases}
  &
  z_{i,j}^+ &:=
  \begin{cases}
    (x_{r+1-i} - x_{r+j})^{-1} & \text{if } j \le \lambda_i, \\
    1 - \tau_{r+1-i,r+j} (x_{r+1-i} - x_{r+j})^{-1} & \text{if } j > \lambda_i,
  \end{cases}
  \\
   y_{i,j}^- &:=
  \begin{cases}
    \tau_{r+1-i,r+j} - x_{r+1-i} + x_{r+j} & \text{if } j \le \lambda_i, \\
    1 & \text{if } j > \lambda_i,
  \end{cases}
  &
  z_{i,j}^- &:=
  \begin{cases}
    -(x_{r+1-i} - x_{r+j})^{-1} & \text{if } j \le \lambda_i, \\
    1 + \tau_{r+1-i,r+j} (x_{r+1-i} - x_{r+j})^{-1} & \text{if } j > \lambda_i,
  \end{cases}
\end{align*}
and
\[
  y_\lambda^+ := \prod_{\substack{ i=1,\dotsc,r \\ j=1,\dotsc,n-r }} y_{i,j}^+,\quad
  z_\lambda^+ := \prod_{\substack{ i=r,\dotsc,1 \\ j=n-r,\dotsc,1 }} z_{i,j}^+,\quad
  y_\lambda^- := (-1)^{|\lambda|} \prod_{\substack{ i=1,\dotsc,r \\ j=1,\dotsc,n-r }} y_{i,j}^-,\quad
  z_\lambda^- := (-1)^{|\mu|} \prod_{\substack{ i=r,\dotsc,1 \\ j=n-r,\dotsc,1 }} z_{i,j}^-.
\]
Note that, since $A$ is not necessarily supercommutative, the order of these products is important.  However, it is straightforward to verify that, for $\lambda \in \cP_{r,n}$, the products $z_\lambda^\pm$ depend only on the order of the terms $z_{i,j}^\pm$ for $j > \lambda_i$, that is, the terms corresponding to boxes outside the Young diagram $\lambda$.  Furthermore, for these terms, any order such that removing the boxes (starting from $\max_{r,n}$) in the given order results in a sequence of Young diagrams is equivalent.  Similarly, the products $y_\lambda^\pm$ depend only on the order of the boxes inside the Young diagram and any order such that adding the boxes in the given order (starting from $\min_{r,n}$) results in a sequence of Young diagrams is equivalent.

\begin{lem} \label{jack}
  For $0 \le r \le n$ and $\lambda, \mu \in \cP_{r,n}$, we have
  \begin{align*}
    1_\mu \circ \tilde\Delta(e_{(n)}) \circ 1_\lambda
    &= \binom{n}{r}^{-1} \sigma_\mu^{-1} \circ \imath_{r,n} ( z_\mu^+ ) \circ \imath_{r,n} \left( e_{(n-r)} \otimes e_{(r)} \right) \circ \imath_{r,n} ( y_\lambda^+ ) \circ \sigma_\lambda,
    \\
    1_\mu \circ \tilde\Delta(e_{(1^n)}) \circ 1_\lambda
    &= \binom{n}{r}^{-1} \sigma_\mu^{-1} \circ \imath_{r,n} ( z_\mu^- ) \circ \imath_{r,n} \left( e_{(n-r)} \otimes e_{(r)} \right) \circ \imath_{r,n} ( y_\lambda^- ) \circ \sigma_\lambda.
  \end{align*}
\end{lem}

\begin{proof}
  First note that
  \begin{align*}
    \tilde\Delta(e_{(2)})
    &= \frac{1}{2}
    \left(\
      \begin{tikzpicture}[anchorbase]
        \draw[->,blue] (-0.2,-0.4) to (-0.2,0.4);
        \draw[->,blue] (0.2,-0.4) to (0.2,0.4);
      \end{tikzpicture}
      +
      \begin{tikzpicture}[anchorbase]
        \draw[->,blue] (-0.2,-0.4) to (0.2,0.4);
        \draw[->,blue] (0.2,-0.4) to (-0.2,0.4);
      \end{tikzpicture}
    \ \right)
    + \frac{1}{2}
    \left(\
      \begin{tikzpicture}[anchorbase]
        \draw[->,red] (-0.2,-0.4) to (-0.2,0.4);
        \draw[->,red] (0.2,-0.4) to (0.2,0.4);
      \end{tikzpicture}
      +
      \begin{tikzpicture}[anchorbase]
        \draw[->,red] (-0.2,-0.4) to (0.2,0.4);
        \draw[->,red] (0.2,-0.4) to (-0.2,0.4);
      \end{tikzpicture}
    \ \right)
    + \frac{1}{2}
    \left(\
      \begin{tikzpicture}[anchorbase]
        \draw[->,blue] (-0.2,-0.4) to (-0.2,0.4);
        \draw[->,red] (0.2,-0.4) to (0.2,0.4);
      \end{tikzpicture}
      -
      \begin{tikzpicture}[anchorbase]
        \draw[->,blue] (-0.2,-0.4) to (-0.2,0.4);
        \draw[->,red] (0.2,-0.4) to (0.2,0.4);
        \dotdumb{-0.2,0.15}{0.2,0.15};
        \teleport{blue}{red}{-0.2,-0.15}{0.2,-0.15};
      \end{tikzpicture}
      +
      \begin{tikzpicture}[anchorbase]
        \draw[->,blue] (-0.4,-0.4) to (0.4,0.4);
        \draw[->,red] (0.4,-0.4) to (-0.4,0.4);
        \dotdumb{-0.2,0.2}{0.2,0.2};
      \end{tikzpicture}
    \ \right)
    \circ
    \left(\
      \begin{tikzpicture}[anchorbase]
        \draw[->,blue] (-0.2,-0.4) to (-0.2,0.4);
        \draw[->,red] (0.2,-0.4) to (0.2,0.4);
      \end{tikzpicture}
      +
      \begin{tikzpicture}[anchorbase]
        \draw[->,red] (-0.4,-0.4) to (0.4,0.4);
        \draw[->,blue] (0.4,-0.4) to (-0.4,0.4);
        \teleport{blue}{red}{-0.2,0.2}{0.2,0.2};
      \end{tikzpicture}
      +
      \begin{tikzpicture}[anchorbase]
        \draw[->,red] (-0.4,-0.4) to (0.4,0.4);
        \draw[->,blue] (0.4,-0.4) to (-0.4,0.4);
        \dashdumb{-0.2,0.2}{0.2,0.2};
      \end{tikzpicture}
    \ \right),
    \\
    \tilde\Delta(e_{(1^2)})
    &= \frac{1}{2}
    \left(\
      \begin{tikzpicture}[anchorbase]
        \draw[->,blue] (-0.2,-0.4) to (-0.2,0.4);
        \draw[->,blue] (0.2,-0.4) to (0.2,0.4);
      \end{tikzpicture}
      -
      \begin{tikzpicture}[anchorbase]
        \draw[->,blue] (-0.2,-0.4) to (0.2,0.4);
        \draw[->,blue] (0.2,-0.4) to (-0.2,0.4);
      \end{tikzpicture}
    \ \right)
    + \frac{1}{2}
    \left(\
      \begin{tikzpicture}[anchorbase]
        \draw[->,red] (-0.2,-0.4) to (-0.2,0.4);
        \draw[->,red] (0.2,-0.4) to (0.2,0.4);
      \end{tikzpicture}
      -
      \begin{tikzpicture}[anchorbase]
        \draw[->,red] (-0.2,-0.4) to (0.2,0.4);
        \draw[->,red] (0.2,-0.4) to (-0.2,0.4);
      \end{tikzpicture}
    \ \right)
    + \frac{1}{2}
    \left(\
      \begin{tikzpicture}[anchorbase]
        \draw[->,blue] (-0.2,-0.4) to (-0.2,0.4);
        \draw[->,red] (0.2,-0.4) to (0.2,0.4);
      \end{tikzpicture}
      +
      \begin{tikzpicture}[anchorbase]
        \draw[->,blue] (-0.2,-0.4) to (-0.2,0.4);
        \draw[->,red] (0.2,-0.4) to (0.2,0.4);
        \dotdumb{-0.2,0.15}{0.2,0.15};
        \teleport{blue}{red}{-0.2,-0.15}{0.2,-0.15};
      \end{tikzpicture}
      -
      \begin{tikzpicture}[anchorbase]
        \draw[->,blue] (-0.4,-0.4) to (0.4,0.4);
        \draw[->,red] (0.4,-0.4) to (-0.4,0.4);
        \dotdumb{-0.2,0.2}{0.2,0.2};
      \end{tikzpicture}
    \ \right)
    \circ
    \left(\
      \begin{tikzpicture}[anchorbase]
        \draw[->,blue] (-0.2,-0.4) to (-0.2,0.4);
        \draw[->,red] (0.2,-0.4) to (0.2,0.4);
      \end{tikzpicture}
      +
      \begin{tikzpicture}[anchorbase]
        \draw[->,red] (-0.4,-0.4) to (0.4,0.4);
        \draw[->,blue] (0.4,-0.4) to (-0.4,0.4);
        \teleport{blue}{red}{-0.2,0.2}{0.2,0.2};
      \end{tikzpicture}
      -
      \begin{tikzpicture}[anchorbase]
        \draw[->,red] (-0.4,-0.4) to (0.4,0.4);
        \draw[->,blue] (0.4,-0.4) to (-0.4,0.4);
        \dashdumb{-0.2,0.2}{0.2,0.2};
      \end{tikzpicture}
    \ \right).
  \end{align*}
  The lemma in the case $n=2$ follows from these formulae.  For the general case, we proceed by induction on $|\lambda| - |\mu|$.  We give the proof for the first formula, since the proof of the second is analogous.

  For the base case, we have $\lambda = \min_{r,n}$ (so $1_\lambda = \sigma_\lambda = \upblue^{\otimes (n-r)} \otimes \upred^{\otimes r}$) and $\mu = \max_{r,n}$ (so $1_\mu = \upred^{\otimes r} \otimes \upblue^{\otimes (n-r)}$.  Note that
  \[
    e_{(n)} = \frac{1}{n!} \sum_{\rho \in \fS_{n-r} \times \fS_r} \sum_{\sigma \in D} \sigma \rho,
  \]
  where $D$ denotes the set of minimal length $\fS_n / (\fS_{n-r} \times \fS_r)$-coset representatives.  For $\rho \in \fS_{n-r} \times \fS_r$, we have $\tilde\Delta(\rho) \circ 1_\lambda = 1_\lambda \circ \imath_{r,n} (\rho) \circ \sigma_\lambda$.  Therefore
  \[
    1_\mu \circ \tilde\Delta ( e_{(n)} ) \circ 1_\lambda
    = \binom{n}{r}^{-1} \sum_{\sigma \in D} 1_\mu \circ \tilde\Delta(\sigma) \circ 1_\lambda \circ \imath_{r,n} \left( e_{(n-r)} \otimes e_{(r)} \right) \circ \sigma_\lambda.
  \]
  Since $\mu$ is maximal, the term $1_\mu \circ \tilde\Delta(\sigma) \circ 1_\lambda$ above can only be nonzero when $\sigma$ is the longest coset representative.  For this $\sigma$, when computing $\tilde\Delta(\sigma)$, we must replace each crossing
  $
    \begin{tikzpicture}[anchorbase]
      \draw[->] (-0.2,-0.25) to (0.2,0.25);
      \draw[->] (0.2,-0.25) to (-0.2,0.25);
    \end{tikzpicture}
  $
  in a reduced word for $\sigma$ with
  $
    \begin{tikzpicture}[anchorbase]
      \draw[->,blue] (-0.2,-0.25) to (0.2,0.25);
      \draw[->,red] (0.2,-0.25) to (-0.2,0.25);
      \dotdumb{-0.1,-0.125}{0.1,-0.125};
    \end{tikzpicture},
  $
  i.e.\ the terms from \cref{Deltcross} that are colored $\upred \upblue$ at the top and $\upblue \upred$ at the bottom.  It follows for this longest $\sigma$ that
  \begin{align*}
    1_\mu \circ \tilde\Delta (\sigma) \circ 1_\lambda
    &= \prod_{\substack{ 1 \le i \le r \\ 1 \le j \le n-r}} \imath_{r,n} \left( (x_{r+1-i} - x_{r+j})^{-1} \right) \circ \sigma_\lambda.
  \end{align*}
  Since $j > \lambda_i$ and $j \le \mu_i$ for all $i,j$, this completes the proof of the base case.

  For the induction step, consider $\mu,\lambda \in \cP_{r,n}$ such that either $\mu$ is not maximal or $\lambda$ is not minimal, and consider $X := 1_\mu \circ \tilde\Delta (e_{(n)}) \circ 1_\lambda$.  If $\mu$ is not maximal, let $\nu \in \cP_{r,n}$ be obtained from $\mu$ by adding a box.  Let $j$ be the unique index such that
  $
    \sigma_\mu^{-1} =
    \left(
      \begin{tikzpicture}[anchorbase]
        \draw[->,red] (-0.2,-0.25) to (0.2,0.25);
        \draw[->,blue] (0.2,-0.25) to (-0.2,0.25);
      \end{tikzpicture}
    \right)_j
    \circ \sigma_\nu^{-1},
  $
  where the subscript $j$ indicates that we are applying the crossing to the $j$-th and $(j+1)$-st strings.  Using the formula for $Y := 1_\nu \circ \tilde\Delta(e_{(n)}) \circ 1_\lambda$ provided by the induction hypothesis, we are reduced to verifying that
  \begin{equation} \label{clouds1}
    X
    =
    \left(
      \begin{tikzpicture}[anchorbase]
        \draw[->,red] (-0.4,-0.4) to (0.4,0.4);
        \draw[->,blue] (0.4,-0.4) to (-0.4,0.4);
      \end{tikzpicture}
    \right)_j
    \circ
    \left(\
      \begin{tikzpicture}[anchorbase]
        \draw[->,red] (-0.2,-0.4) to (-0.2,0.4);
        \draw[->,blue] (0.2,-0.4) to (0.2,0.4);
      \end{tikzpicture}
      -
      \begin{tikzpicture}[anchorbase]
        \draw[->,red] (-0.2,-0.4) to (-0.2,0.4);
        \draw[->,blue] (0.2,-0.4) to (0.2,0.4);
        \dotdumb{-0.2,0.15}{0.2,0.15};
        \teleport{red}{blue}{-0.2,-0.15}{0.2,-0.15};
      \end{tikzpicture}
    \ \right)_j
    \circ
    \left(
      \begin{tikzpicture}[anchorbase]
        \draw[->,red] (-0.2,-0.4) to (-0.2,0.4);
        \draw[->,blue] (0.2,-0.4) to (0.2,0.4);
        \dashdumb{-0.2,0}{0.2,0};
      \end{tikzpicture}
    \right)_j
    \circ Y.
  \end{equation}
  To see this, we apply $1_\nu \circ \tilde\Delta( - ) \circ 1_\lambda$ to the identity $e_{(n)} = \frac{1}{2} (1 + s_j) e_{(n)}$ to see that
  \[
    Y = \frac{1}{2}
    \left(\
      \begin{tikzpicture}[anchorbase]
        \draw[->,red] (-0.2,-0.4) to (-0.2,0.4);
        \draw[->,blue] (0.2,-0.4) to (0.2,0.4);
      \end{tikzpicture}
      +
      \begin{tikzpicture}[anchorbase]
        \draw[->,red] (-0.2,-0.4) to (-0.2,0.4);
        \draw[->,blue] (0.2,-0.4) to (0.2,0.4);
        \dotdumb{-0.2,0.15}{0.2,0.15};
        \teleport{red}{blue}{-0.2,-0.15}{0.2,-0.15};
      \end{tikzpicture}
    \ \right)_j
    \circ Y + \frac{1}{2}
    \left(
      \begin{tikzpicture}[anchorbase]
        \draw[->,blue] (-0.4,-0.4) to (0.4,0.4);
        \draw[->,red] (0.4,-0.4) to (-0.4,0.4);
        \dotdumb{-0.2,0.2}{0.2,0.2};
      \end{tikzpicture}
    \right)_j
    \circ X,
  \]
  which is equivalent to \cref{clouds1}.

  On the other hand, if $\lambda$ is not minimal, let $\kappa$ be obtained from $\lambda$ by removing a box, and define $j$ so that
  $
    \sigma_\lambda = \sigma_\kappa \circ
    \left(
      \begin{tikzpicture}[anchorbase]
        \draw[->,red] (-0.2,-0.25) to (0.2,0.25);
        \draw[->,blue] (0.2,-0.25) to (-0.2,0.25);
      \end{tikzpicture}
    \right)_j.
  $
  Let $Z := 1_\mu \circ \tilde\Delta(e_{(n)}) \circ 1_\kappa$.  Then we need to show that
  \begin{equation} \label{clouds2}
    X \circ
    \left(
      \begin{tikzpicture}[anchorbase]
        \draw[->,blue] (-0.4,-0.4) to (0.4,0.4);
        \draw[->,red] (0.4,-0.4) to (-0.4,0.4);
      \end{tikzpicture}
    \right)_j
    = Z \circ
    \left(\
      \begin{tikzpicture}[anchorbase]
        \draw[->,blue] (-0.2,-0.4) to (-0.2,0.4);
        \draw[->,red] (0.2,-0.4) to (0.2,0.4);
        \teleport{blue}{red}{-0.2,0}{0.2,0};
      \end{tikzpicture}
      +
      \begin{tikzpicture}[anchorbase]
        \draw[->,blue] (-0.2,-0.4) to (-0.2,0.4);
        \draw[->,red] (0.2,-0.4) to (0.2,0.4);
        \dashdumb{-0.2,0}{0.2,0};
      \end{tikzpicture}
    \ \right)_j.
  \end{equation}
  To see this, we apply $1_\mu \circ \tilde\Delta(-) \circ 1_\kappa$ to the identity $e_{(n)} = e_{(n)} \frac{1}{2} (1 + s_j)$ to obtain
  \[
    Z = \frac{1}{2} Z \circ
    \left(\
      \begin{tikzpicture}[anchorbase]
        \draw[->,blue] (-0.2,-0.4) to (-0.2,0.4);
        \draw[->,red] (0.2,-0.4) to (0.2,0.4);
      \end{tikzpicture}
      -
      \begin{tikzpicture}[anchorbase]
        \draw[->,blue] (-0.2,-0.4) to (-0.2,0.4);
        \draw[->,red] (0.2,-0.4) to (0.2,0.4);
        \dotdumb{-0.2,0.15}{0.2,0.15};
        \teleport{blue}{red}{-0.2,-0.15}{0.2,-0.15};
      \end{tikzpicture}
    \ \right)_j
    + \frac{1}{2} X \circ
    \left(
      \begin{tikzpicture}[anchorbase]
        \draw[->,blue] (-0.4,-0.4) to (0.4,0.4);
        \draw[->,red] (0.4,-0.4) to (-0.4,0.4);
        \dotdumb{-0.2,0.2}{0.2,0.2};
      \end{tikzpicture}
    \right)_j,
  \]
  which is equivalent to
  \[
    X \circ
    \left(
      \begin{tikzpicture}[anchorbase]
        \draw[->,blue] (-0.4,-0.4) to (0.4,0.4);
        \draw[->,red] (0.4,-0.4) to (-0.4,0.4);
        \dotdumb{-0.2,0.2}{0.2,0.2};
      \end{tikzpicture}
    \right)_j
    = Z \circ
    \left(\
      \begin{tikzpicture}[anchorbase]
        \draw[->,blue] (-0.2,-0.4) to (-0.2,0.4);
        \draw[->,red] (0.2,-0.4) to (0.2,0.4);
      \end{tikzpicture}
      +
      \begin{tikzpicture}[anchorbase]
        \draw[->,blue] (-0.2,-0.4) to (-0.2,0.4);
        \draw[->,red] (0.2,-0.4) to (0.2,0.4);
        \dotdumb{-0.2,0.15}{0.2,0.15};
        \teleport{blue}{red}{-0.2,-0.15}{0.2,-0.15};
      \end{tikzpicture}
    \ \right)_j.
  \]
  Composing on the right with
  $
    \left(
      \begin{tikzpicture}[anchorbase]
        \draw[->,blue] (-0.2,-0.25) to (-0.2,0.25);
        \draw[->,red] (0.2,-0.25) to (0.2,0.25);
        \dashdumb{-0.2,0}{0.2,0};
      \end{tikzpicture}
    \right)_j
  $
  gives \cref{clouds2}.
\end{proof}

\begin{lem}
    For $0 \le r \le n$, we have
    \begin{equation} \label{wrap}
        \sum_{\sigma \in \fS_{n-r} \times \fS_r} \sum_{\lambda \in \cP_{r,n}} \sigma \actplus (y_\lambda^+ z_\lambda^+)
        = n!
        = \sum_{\sigma \in \fS_{n-r} \times \fS_r} \sum_{\mu \in \cP_{r,n}} \sigma \actminus (y_\mu^- z_\mu^-).
    \end{equation}
\end{lem}

\begin{proof}
  First note that
  \[
    y_{i,j}^\pm z_{i,j}^\pm = 1 \pm \varepsilon_{i,j}(\lambda) \tau_{r+1-i,r+j} (x_{r+1-i} - x_{r+j})^{-1}
    \quad \text{where} \quad
    \varepsilon_{i,j}(\lambda) =
    \begin{cases}
      1 & \text{if } j \le \lambda_i, \\
      -1 & \text{if } j > \lambda_i.
    \end{cases}
  \]
  Thus, it suffices to show that
  \begin{gather} \label{dubstep1}
      \sum_{\sigma \in \fS_{n-r} \times \fS_r} \sigma \actplus \left( \sum_{\lambda \in \cP_{r,n}} \prod_{\substack{1 \le i \le r \\ 1 \le j \le n-r}} \left(1 + \varepsilon_{i,j}(\lambda) \tau_{r+1-i,r+j}(x_{r+1-i} - x_{r+j})^{-1} \right) \right) = n!
      \quad \text{and}
      \\ \label{dubstep2}
      \sum_{\sigma \in \fS_{n-r} \times \fS_r} \sigma \actminus \left( \sum_{\lambda \in \cP_{r,n}} \prod_{\substack{1 \le i \le r \\ 1 \le j \le n-r}} \left(1 - \varepsilon_{i,j}(\lambda) \tau_{r+1-i,r+j}(x_{r+1-i} - x_{r+j})^{-1} \right) \right) = n!,
  \end{gather}
  where the order of the products is given as follows: all $(i,j)$-terms for $(i,j)$ in the Young diagram $\lambda$ occur before those with $(i,j)$ outside the Young diagram.  The terms corresponding to $(i,j)$ in the Young diagram $\lambda$ are ordered as they are in the definition of $y_\lambda^\pm$ (i.e.\ such that adding the boxes to $\min_{r,n}$ in order yields a sequence of Young diagrams) and the terms corresponding to $(i,j)$ outside the Young diagram $\lambda$ are ordered as they are in the definition of $z_\lambda^\pm$ (i.e.\ such that removing the boxes from $\max_{r,n}$ in order yields a sequence of Young diagrams).

  The proof of \cref{dubstep1} is almost identical to the proof of \cite[Lem.~4.6]{BSW-K0}; one merely replaces $y_{i,j} = (x_{r+1-i} - x_{r+j})^{-1}$ everywhere there by $\tau_{r+1-i,r+j} (x_{r+1-i} - x_{r+j})^{-1}$.  Then \cref{dubstep2} follows by applying the automorphism $x_i \mapsto -x_i$, $1 \le i \le n$, of $A^{\otimes n} \otimes \kk(x_1,\dotsc,x_n)$.
  \details{
    We prove \cref{dubstep1} by induction on $n$.  Decompose the left-hand side of \cref{dubstep1} as the sum $X + Y$, where in $X$ the second sum in \cref{dubstep1} is only over $\lambda \in \cP_{r,n}$ with $\lambda_1 = n-r$, and in $Y$ the second sum in \cref{dubstep1} is only over $\lambda \in \cP_{r,n}$ with $\lambda_1 < n-r$.  Using the induction hypothesis, together with the fact that $\{ 1, s_{n-1}, s_{n-2} s_{n-1}, \dotsc, s_1 \dotsm s_{n-1} \}$ is a set of $\fS_n / \fS_{n-1}$ coset representatives, we see that
    \begin{align*}
      X &= (n-1)! (1 + s_{r-1} + \dotsb + s_1 \dotsm s_{r-1}) \actplus \prod_{j=1,\dotsc,n-r} \left( 1 + \tau_{r,r+j} (x_r - x_{r+j})^{-1} \right), \\
      Y &= (n-1)! (1 + s_{n-1} + \dotsb + s_{r+1} \dotsm s_{n-1}) \actplus \prod_{i=r,\dotsc,1} \left( 1 - \tau_{i,n} (x_i - x_n)^{-1} \right).
    \end{align*}
    For $1 \le i \le r$ and $1 \le j \le n-r$, define
    \[
      \zeta_{i,j} = \tau_{r+1-i,r+j} (x_{r+1-i} - x_{r+j})^{-1}.
    \]
    Using the algebraic analogues of the relations \cref{hops1,hops2,laser}, we have, for $1 \le i \le r$ and $1 \le j \le n-r$,
    \begin{align} \label{dance1} \tag{$*$}
      s_{r+1-q} \zeta_{i,j}
      &=
      \begin{cases}
        \zeta_{i+1,j} s_{r+1-q} - \zeta_{i+1,j} \zeta_{i,j}
        & \text{if } i+1 = q \le r,
        \\
        \zeta_{i,j} s_{r+1-q}
        & \text{if } i+1 < q \le r;
      \end{cases}
      \\ \label{dance2} \tag{$**$}
      s_{r+q} \zeta_{i,j}
      &=
      \begin{cases}
        \zeta_{i,j-1} s_{r+q} + \zeta_{i,j-1} \zeta_{i,j}
        & \text{if } 1 \le q = j-1,
        \\
        \zeta_{i,j} s_{r+q}
        & \text{if } 1 \le q < j-1.
      \end{cases}
    \end{align}

    For $m \geq 1$, let $C_m$ be the set of sequences $\left((i_1,j_1),\dots,(i_m,j_m)\right) \in \left(\{1,\dots,r\}\times\{1,\dots,n-r\}\right)^{m}$ such that either $i_q > i_{q+1}, j_q = j_{q+1}$ or $i_q=i_{q+1},j_q < j_{q+1}$ for each $q=1,\dotsc,m-1$.  Such a sequence may be visualized as a ``hook'' drawn inside the $r \times (n-r)$ rectangle, e.g.\ if $r=4,n=9$ then $((4,1),(4,2), (2,2), (2,4)) \in C_4$ is
    \[
      \begin{tikzpicture}[baseline=-5mm]
        \draw[-] (0,0.2) to (1,0.2) to (1,-.6) to (0,-.6) to (0,0.2);
        \draw[-](.1,-.5) to (.3,-.5) to (.3,-.1) to (.7,-.1);
        \node at (0.1,-0.5) {$\bullet$};
        \node at (0.3,-0.5) {$\bullet$};
        \node at (0.3,-0.1) {$\bullet$};
        \node at (0.7,-0.1) {$\bullet$};
      \end{tikzpicture}
    \]
    Using \cref{dance1} and induction on $i=1,\dots,r$, one shows that
    \[
      s_{r+1-i} \dotsm s_{r-2} s_{r-1} \actplus \prod_{j=1}^{n-r} \left(1+\zeta_{1,j}\right)
      = 1 - \sum_{m \geq 1} \sum_{\substack{\left((i_1,j_1),\dotsc,(i_m,j_m)\right)\in C_m\\i_1 = i}} (-1)^{|\{i_1,\dotsc,i_m\}|} \zeta_{i_1,j_1} \dotsm \zeta_{i_m,j_m}.
    \]
    Hence:
    \[
      X= r(n-1)! -(n-1)!\sum_{m \geq 1} \sum_{\left((i_1,j_1),\dotsc,(i_m,j_m)\right)\in C_m} (-1)^{|\{i_1,\dotsc,i_m\}|} \zeta_{i_1,j_1} \dotsm \zeta_{i_m,j_m}.
    \]
    Similarly, using \eqref{dance2} and induction on $j=n-r,\dotsc,1$, one shows that
    \[
      s_{r+j} \dotsm s_{n-2} s_{n-1} \actplus \prod_{i=1}^r \left(1-\zeta_{i,n-r}\right)
      = 1 + \sum_{m \geq 1} \sum_{\substack{\left((i_1,j_1),\dots,(i_m,j_m)\right)\in C_m\\j_1 =
      j}} (-1)^{|\{i_1,\dotsc,i_m\}|} \zeta_{i_1,j_1} \dotsm \zeta_{i_m,j_m}.
    \]
    Hence
    \[
      Y = (n-r)(n-1)!+(n-1)! \sum_{m \geq 1}
      \sum_{\left((i_1,j_1),\dotsc,(i_m,j_m)\right)\in C_m}
      (-1)^{|\{i_1,\dotsc,i_m\}|} \zeta_{i_1,j_1} \dotsb \zeta_{i_m,j_m}.
    \]
    Adding the above expressions for $X$ and $Y$ gives that $X+Y = n!$.
  }
\end{proof}

\begin{proof}[Proof of \cref{coprodcat}]
    We just go through the proof for the first isomorphism; the argument for the second one is similar, using $z_\mu^-$ and $y_\lambda^-$ instead of $z_\mu^+$ and $y_\lambda^+$, and \cref{sad} instead of \cref{happy}.  Define morphisms $u$ and $v$ in $\Kar((\blue{\AWC} \barodot \red{\AWC})_{q,\pi})$ by
    \[
        u := \sum_{r=0}^n \binom{n}{r}^{-1} \sum_{\mu \in \cP_{r,n}} \sigma_\mu^{-1} \circ \imath_{r,n}(z_\mu^+) \circ \imath_{r,n} \left( e_{(n-r),i} \otimes e_{(r),i} \right),
        \quad
        v := \sum_{r=0}^n \sum_{\lambda \in \cP_{r,n}} \imath_{r,n} \left( e_{(n-r),i} \otimes e_{(r),i} \right) \circ \imath_{r,n}(y_\lambda^+) \circ \sigma_\lambda.
    \]
    \Cref{jack} implies that $u \circ v = \tilde\Delta(e_{(n),i})$.  (We use here the fact that the tokens labeled $e_i$ commute with all crossings and dots.)  We also have
    \begin{align*}
        v \circ u
        &= \sum_{r=0}^n \binom{n}{r}^{-1} \imath_{r,n} \left( e_{(n-r),i} \otimes e_{(r),i} \right) \circ \imath_{r,n} (y_\lambda^+ z_\lambda^+) \circ \imath_{r,n} \left( e_{(n-r),i} \otimes e_{(r),i} \right)
        \\
        &\overset{\mathclap{\cref{happy}}}{=}\
        \sum_{r=0}^n \imath_{r,n} \left( e_{(n-r),i} \otimes e_{(r),i} \right) \circ \imath_{r,n} \left( \frac{1}{n!} \sum_{\sigma \in \fS_{n-r} \times \fS_r} \sigma \actplus (y_\lambda^+ z_\lambda^+) \right) \circ \imath_{r,n} \left( e_{(n-r),i} \otimes e_{(r),i} \right)
        \\
        &\overset{\mathclap{\cref{wrap}}}{=}\
        \sum_{r=0}^n \imath_{r,n} \left( e_{(n-r),i} \otimes e_{(r),i} \right).
    \end{align*}
    It follows that the idempotents $\tilde\Delta(e_{(n),i})$ and $\sum_{r=0}^n \imath_{r,n}(e_{(n-r),i} \otimes e_{(r),i})$ are conjugate. Hence, the objects   $\tilde\Delta(H_n)$ and $\bigoplus_{r=0}^n \blue{H_{n-r}} \otimes \red{H_r}$ of the Karoubi   envelope that they define are isomorphic.  Note finally that $y_\lambda^+$ has degree $2 |\lambda| \dA$, $z_\lambda^+$ has degree $-2 |\lambda| \dA$, and $\sigma_\lambda$ has degree $-2|\lambda|\dA$.  It follows that the morphisms $u$ and $v$ are both homogeneous of degree zero.  So the isomorphism we have constructed is also of degree zero.
\end{proof}

\section{The Frobenius Heisenberg category}

Continue with $A$ being a symmetric graded Frobenius superalgebra with
trace of degree $-2d$; we do not need the hypothesis \eqref{virginia}
until the last paragraph of the section.  We now introduce our main object of study, the Frobenius Heisenberg category, which is a slight modification of the category $\Heis{A}{k}$ defined in \cite{Sav18}; see \cref{expedia}. (Note that in \cite{Sav18} it was assumed that $A$ is positively graded, but actually this is not needed for the definition or the proofs of the other relations used below.)

\begin{defin} \label{star}
  The \emph{Frobenius Heisenberg category} $\Heis{A}{k}$ of central charge $k \in \Z$ is the strict graded monoidal supercategory generated by objects $\uparrow$ and $\downarrow$ and morphisms
  \begin{gather*}
    \begin{tikzpicture}[anchorbase]
      \draw[->] (0,0) -- (0,0.6);
      \blackdot{0,0.3};
    \end{tikzpicture}
    \colon \uparrow \to \uparrow
    \ ,\quad
    \begin{tikzpicture}[anchorbase]
      \draw [->](0,0) -- (0.6,0.6);
      \draw [->](0.6,0) -- (0,0.6);
    \end{tikzpicture}
    \colon \uparrow \otimes \uparrow\ \to \uparrow \otimes \uparrow
    \ ,\quad
    \begin{tikzpicture}[anchorbase]
      \draw[->] (0,0) -- (0,0.6);
      \blacktoken[west]{0,0.3}{a};
    \end{tikzpicture}
    \colon \uparrow \to \uparrow
    \ ,\ a \in A,
    \\
    \begin{tikzpicture}[anchorbase]
      \draw[->] (0,.2) -- (0,0) arc (180:360:.3) -- (.6,.2);
    \end{tikzpicture}
    \ \colon \one \to\ \downarrow\otimes \uparrow
    , \quad
    \begin{tikzpicture}[anchorbase]
      \draw[->] (0,-.2) -- (0,0) arc (180:0:.3) -- (.6,-.2);
    \end{tikzpicture}
    \ \colon \uparrow\otimes \downarrow\ \to \one
    , \quad
    \begin{tikzpicture}[anchorbase]
      \draw[<-] (0,.2) -- (0,0) arc (180:360:.3) -- (.6,.2);
    \end{tikzpicture}
    \ \colon \one \to\ \uparrow\otimes \downarrow
    , \quad
    \begin{tikzpicture}[anchorbase]
      \draw[<-] (0,0) -- (0,.2) arc (180:0:.3) -- (.6,0);
    \end{tikzpicture}
    \ \colon \downarrow\otimes \uparrow\ \to \one,
  \end{gather*}
  subject to certain relations.  The degree of the dot, crossing, and token are as in \cref{Wrupgen,upgen}, while the cups/caps are even of degrees
  \begin{equation} \label{trappist}
    \deg \left(\:
    \begin{tikzpicture}[anchorbase]
      \draw[->] (0,.2) -- (0,0) arc (180:360:.3) -- (.6,.2);
    \end{tikzpicture}
    \: \right) = k \dA,
    \quad
    \deg \left( \:
    \begin{tikzpicture}[anchorbase]
      \draw[->] (0,-.2) -- (0,0) arc (180:0:.3) -- (.6,-.2);
    \end{tikzpicture}
    \:\right) = - k \dA,
    \quad
    \deg \left(
    \:\begin{tikzpicture}[anchorbase]
      \draw[<-] (0,.2) -- (0,0) arc (180:360:.3) -- (.6,.2);
    \end{tikzpicture}
    \:\right) = - k \dA,
    \quad
    \deg \left(
    \:\begin{tikzpicture}[anchorbase]
      \draw[<-] (0,0) -- (0,.2) arc (180:0:.3) -- (.6,0);
    \end{tikzpicture}
    \:\right) = k \dA.
  \end{equation}
  The defining relations are \cref{wreathrel,affrel,affrel1}, plus the following additional relations:
  \begin{gather} \label{squid}
    \begin{tikzpicture}[anchorbase]
      \draw[->] (-0.3,-0.5) to (-0.3,0) to[out=up,in=up,looseness=2] (0,0) to[out=down,in=down,looseness=2] (0.3,0) to (0.3,0.5);
    \end{tikzpicture}
    \ =\
    \begin{tikzpicture}[anchorbase]
      \draw[->] (0,-0.5) to (0,0.5);
    \end{tikzpicture}
    \ ,\qquad
    \begin{tikzpicture}[anchorbase]
      \draw[->] (-0.3,0.5) to (-0.3,0) to[out=down,in=down,looseness=2] (0,0) to[out=up,in=up,looseness=2] (0.3,0) to (0.3,-0.5);
    \end{tikzpicture}
    \ =\
    \begin{tikzpicture}[anchorbase]
      \draw[<-] (0,-0.5) to (0,0.5);
    \end{tikzpicture}
    \ ,
  \\ \label{chew}
  \begin{tikzpicture}[baseline={(0,-0.15)}]
    \jonbubright{0,0}{a}{r};
  \end{tikzpicture}
    \ = -\delta_{r,k-1} \tr(a)
    \quad \text{if } 0 \le r < k,
    \qquad
  \begin{tikzpicture}[baseline={(0,-0.15)}]
    \jonbubleft{0,0}{a}{r};
  \end{tikzpicture}
    \ = \delta_{r,-k-1} \tr(a)
    \quad \text{if } 0 \le r < -k,
    \\ \label{curly}
    \begin{tikzpicture}[anchorbase]
      \draw[->] (0,-0.5) to[out=up,in=0] (-0.3,0.2) to[out=180,in=up] (-0.45,0) to[out=down,in=180] (-0.3,-0.2) to[out=0,in=down] (0,0.5);
    \end{tikzpicture}
    = \delta_{k,0}\
    \begin{tikzpicture}[anchorbase]
      \draw[->] (0,-0.5) -- (0,0.5);
    \end{tikzpicture}
    \quad \text{if } k \le 0,
    \qquad
    \begin{tikzpicture}[anchorbase]
      \draw[->] (0,-0.5) to[out=up,in=180] (0.3,0.2) to[out=0,in=up] (0.45,0) to[out=down,in=0] (0.3,-0.2) to[out=180,in=down] (0,0.5);
    \end{tikzpicture}
    = \delta_{k,0}\
    \begin{tikzpicture}[anchorbase]
      \draw[->] (0,-0.5) -- (0,0.5);
    \end{tikzpicture}
    \quad \text{if } k \ge 0,
    \\ \label{squish}
    \begin{tikzpicture}[anchorbase]
      \draw[<-] (0,0) \braidto (0.5,0.6) \braidto (0,1.2);
      \draw[->] (0.5,0) \braidto (0,0.6) \braidto (0.5,1.2);
    \end{tikzpicture}
    \ =\
    \begin{tikzpicture}[anchorbase]
      \draw[<-] (0,0) -- (0,1.2);
      \draw[->] (0.3,0) -- (0.3,1.2);
    \end{tikzpicture}
    \ + \sum_{r,s \ge 0}
    \begin{tikzpicture}[anchorbase]
      \draw[->] (-0.2,0.6) to (-0.2,0.35) arc(180:360:0.2) to (0.2,0.6);
      \draw[<-] (-0.2,-0.6) to (-0.2,-0.35) arc(180:0:0.2) to (0.2,-0.6);
      \draw[->] (0.8,0) arc(360:0:0.2);
      \multblackdot[east]{0.4,0}{-r-s-2};
      \multblackdot[west]{0.2,0.42}{r};
      \multblackdot[west]{0.2,-0.42}{s};
      \teleport{black}{black}{0.18,0.25}{0.459,0.141};
      \teleport{black}{black}{0.18,-0.25}{0.459,-0.141};
    \end{tikzpicture}
    \ ,\qquad
    \begin{tikzpicture}[anchorbase]
      \draw[->] (0,0) \braidto (0.5,0.6) \braidto (0,1.2);
      \draw[<-] (0.5,0) \braidto (0,0.6) \braidto (0.5,1.2);
    \end{tikzpicture}
    \ =\
    \begin{tikzpicture}[anchorbase]
      \draw[->] (0,0) -- (0,1.2);
      \draw[<-] (0.3,0) -- (0.3,1.2);
    \end{tikzpicture}
    \ + \sum_{r,s \ge 0}
    \begin{tikzpicture}[anchorbase]
      \draw[<-] (-0.2,0.6) to (-0.2,0.35) arc(180:360:0.2) to (0.2,0.6);
      \draw[->] (-0.2,-0.6) to (-0.2,-0.35) arc(180:0:0.2) to (0.2,-0.6);
      \draw[->] (-0.8,0) arc(-180:180:0.2);
      \multblackdot[west]{-0.4,0}{-r-s-2};
      \multblackdot[east]{-0.2,0.42}{r};
      \multblackdot[east]{-0.2,-0.42}{s};
      \teleport{black}{black}{-0.18,0.25}{-0.459,0.141};
      \teleport{black}{black}{-0.18,-0.25}{-0.459,-0.141};
    \end{tikzpicture}
    \ ,
  \end{gather}
  where we have used a dot labeled $r$ to denote the composition of $r$ dots.  The last relation here needs some explanation, since it involves some diagrammatic shorthands which have not yet been defined.  First, we are using the left and right crossings defined by
  \begin{equation} \label{fall}
    \begin{tikzpicture}[anchorbase]
      \draw[->] (0,0) -- (0.6,0.6);
      \draw[<-] (0.6,0) -- (0,0.6);
    \end{tikzpicture}
    \ :=\
    \begin{tikzpicture}[anchorbase]
      \draw[->] (0.2,-0.3) to (-0.2,0.3);
      \draw[<-] (0.6,-0.3) to[out=up,in=45,looseness=2] (0,0) to[out=225,in=down,looseness=2] (-0.6,0.3);
    \end{tikzpicture}
    \ ,\qquad
    \begin{tikzpicture}[anchorbase]
      \draw[<-] (0,0) -- (0.6,0.6);
      \draw[->] (0.6,0) -- (0,0.6);
    \end{tikzpicture}
    \ :=\
    \begin{tikzpicture}[anchorbase]
      \draw[->] (-0.2,-0.3) to (0.2,0.3);
      \draw[<-] (-0.6,-0.3) to[out=up,in=135,looseness=2] (0,0) to[out=-45,in=down,looseness=2] (0.6,0.3);
    \end{tikzpicture}
    \ .
  \end{equation}
  The teleporters appearing on the right hand sides of \cref{squish}
  are defined as in \cref{brexit}. The sums in \cref{squish} also
  involve negatively-dotted bubbles.  To interpret these, one first
  uses \cref{affrel1} to collect the tokens which arise on expanding
  the definitions of the teleporters into a single token; see also
  \cref{sta} below.
Then the negatively-dotted bubbles labeled by $a \in A$ are defined from
  \begin{align} \label{fakel}
    \begin{tikzpicture}[baseline={(0,-0.15)}]
      \jonbubleft{0,0}{a}{r-k};
    \end{tikzpicture}
      &:= \sum_{b_1,\dotsc,b_{r} \in \B_A} \det
      \left(
    \begin{tikzpicture}[baseline={(0,-0.15)}]
      \JONbubright{0,0}{b^\vee_{j-1}b_j}{i-j+k};
    \end{tikzpicture}
    \right)_{i,j=1}^{r+1},\quad\:\:\qquad
    r < k,
    \\ \label{faker}
    \begin{tikzpicture}[baseline={(0,-0.15)}]
      \jonbubright{0,0}{a}{r+k};
    \end{tikzpicture}
    &:= (-1)^{r} \sum_{b_1,\dotsc,b_{r} \in \B_A} \det
    \left(
    \begin{tikzpicture}[baseline={(0,-0.15)}]
      \JONbubleft{0,0}{b^\vee_{j-1}b_j}{i-j-k};
    \end{tikzpicture}
    \right)_{i,j=1}^{r+1},\quad
    r < -k,
  \end{align}
  adopting the convention that $b_0^\vee := a$ and $b_{r+1} := 1$.
  The determinants here mean the usual Laplace expansions keeping the
  non-commuting variables in each monomial ordered in the same way as
  the columns from which they are taken (see \cite[(17)]{Sav18}), and
  we interpret them as $\tr(a)$ if $r+1 = 0$ or as $0$ if $r+1 <
  0$. Note in particular that the sums in \cref{squish} are actually
  finite according to these definitions.
\end{defin}

\begin{rem} \label{expedia}
    Our choices \cref{trappist} of degrees for the cups and caps
    differ from the choices made in \cite{Sav18}.  The current
    choices have the advantage of the grading being compatible with
    the pivotal structure to be discussed shortly (i.e.\ the degree of
    a diagram is invariant under rotation through $180\degree$).  They
    also seem more natural at the level of the Grothendieck  ring; see
    \cref{sec:k0}.
Note also that the only odd
morphisms in $\Heis{A}{k}$ arise from tokens labeled by odd elements
of $A$. This is an important point when interpreting diagrams such as
\cref{squish}, as it means that one really only needs to be careful about
horizontal levels when such odd tokens are involved.
\end{rem}

\begin{lem}
    Up to isomorphism of graded monoidal supercategories, $\Heis{A}{k}$ depends only on the underlying graded superalgebra $A$, and not on its trace map.
\end{lem}

\begin{proof}
    Suppose $(A,\tr_1)$ and $(A,\tr_2)$ are symmetric graded Frobenius superalgebras with the same underlying graded superalgebra $A$.  Then there exists an even, degree zero, invertible element $v \in Z(A)$ such that $\tr_1(a) = \tr_2(va)$ for all $a \in A$.   If $\{b^{\vee,1} : b \in \B_A\}$ and $\{b^{\vee,2} : b \in \B_A\}$ denote the dual bases of $\B_A$ with respect to $\tr_1$ and $\tr_2$, respectively, then $b^{\vee,2} = v b^{\vee,1}$ for $b \in \B_A$.  Then it is straightforward to check that we have an isomorphism of graded monoidal supercategories $\Heis{A,\tr_1}{k} \xrightarrow{\cong} \Heis{A,\tr_2}{k}$ given by
    \[
        \begin{tikzpicture}[anchorbase]
            \draw[->] (0,-0.3) -- (0,0.3);
            \blackdot{0,0};
        \end{tikzpicture}
        \mapsto
        \begin{tikzpicture}[anchorbase]
            \draw[->] (0,-0.3) -- (0,0.3);
            \blackdot{0,0.1};
            \blacktoken[west]{0,-0.1}{v^{-1}};
        \end{tikzpicture}
    \]
    and mapping all other generating morphisms to themselves.
\end{proof}

As explained in the proof of \cite[Th.~1.2]{Sav18}, the defining relations of $\Heis{A}{k}$ imply that the following, which is a $(1+ k\dim A)\times 1$ matrix or a $1 \times (1-k \dim A)$ matrix, respectively, is an isomorphism in $\Add(\Heis{A}{k})$:
\begin{equation} \label{invrel}
  \begin{aligned}
    \left[
      \begin{tikzpicture}[anchorbase]
        \draw [->](0,0) -- (0.6,0.6);
        \draw [<-](0.6,0) -- (0,0.6);
      \end{tikzpicture}
      \quad
      \begin{tikzpicture}[anchorbase]
        \draw[->] (0,0) -- (0,0.5) arc (180:0:.3) -- (0.6,0);
        \multblackdot[west]{0,0.4}{r};
        \blacktoken[west]{0,0.15}{b^\vee};
      \end{tikzpicture}\ ,\
      0 \le r \le k-1,\ b \in \B_A
    \right]^T
    \ &\colon \uparrow \otimes\downarrow \to \downarrow \otimes\uparrow \oplus \one ^{\oplus k \dim A} & \text{if } k \ge 0,
    \\
    \left[
      \begin{tikzpicture}[anchorbase]
        \draw [->](0,0) -- (0.6,0.6);
        \draw [<-](0.6,0) -- (0,0.6);
      \end{tikzpicture}
      \quad
      \begin{tikzpicture}[anchorbase]
        \draw[->] (0,1) -- (0,0.3) arc (180:360:.3) -- (0.6,.8);
        \multblackdot[east]{0.6,0.5}{r};
        \blacktoken[east]{0.6,0.25}{b^\vee};
      \end{tikzpicture}\ ,\
      0 \le r \le -k-1,\ b \in \B_A
    \right]
    \ &\colon \uparrow \otimes\downarrow \oplus \one^{\oplus (-k \dim A)} \to \downarrow \otimes\uparrow & \text{if } k < 0.
  \end{aligned}
\end{equation}
In fact, as in \cite[Def.~1.1]{Sav18}, one can define $\Heis{A}{k}$ equivalently as the strict graded monoidal supercategory generated by morphisms
\begin{equation} \label{lemming}
  \begin{tikzpicture}[anchorbase]
    \draw[->] (0,0) -- (0,0.6);
    \blackdot{0,0.3};
  \end{tikzpicture}
  \ ,\quad
  \begin{tikzpicture}[anchorbase]
    \draw [->](0,0) -- (0.6,0.6);
    \draw [->](0.6,0) -- (0,0.6);
  \end{tikzpicture}
  \ ,\quad
  \begin{tikzpicture}[anchorbase]
    \draw[->] (0,.2) -- (0,0) arc (180:360:.3) -- (.6,.2);
  \end{tikzpicture}
  , \quad
  \begin{tikzpicture}[anchorbase]
    \draw[->] (0,-.2) -- (0,0) arc (180:0:.3) -- (.6,-.2);
  \end{tikzpicture}
  \ , \quad
    \begin{tikzpicture}[anchorbase]
    \draw[->] (0,0) -- (0,0.6);
    \blacktoken[west]{0,0.3}{a};
  \end{tikzpicture}
  \ ,\ a \in A,
\end{equation}
subject only to the relations \cref{wreathrel,affrel1,upslides1,squid}, plus the condition that the morphism \cref{invrel} is invertible, where the right crossing is defined as in \cref{fall}.  In the category defined in this way, there are unique morphisms
\begin{equation} \label{kittens}
  \begin{tikzpicture}[anchorbase]
    \draw[<-] (0,.2) -- (0,0) arc (180:360:.3) -- (.6,.2);
  \end{tikzpicture}
  \quad \text{and} \quad
  \begin{tikzpicture}[anchorbase]
    \draw[<-] (0,0) -- (0,.2) arc (180:0:.3) -- (.6,0);
  \end{tikzpicture}
\end{equation}
such that the other relations in \cref{star} hold, due to the following:

\begin{lem} \label{lock}
    Suppose that $\cC$ is a strict monoidal supercategory containing objects $\uparrow, \downarrow$ and morphisms \cref{lemming} of   the same parities as above which satisfy   \cref{wreathrel,affrel1,upslides1,squid}.  If $\cC$ contains morphisms \cref{kittens} satisfying \cref{squish,chew,curly} for the sideways crossings and the negatively dotted bubbles defined as in \cref{fall,fakel,faker}, then these two morphisms are uniquely determined.
\end{lem}

\begin{proof}
    This follows by the argument in the third-to-last paragraph of the proof of \cite[Th.~1.2]{Sav18}.
\end{proof}

We will need some other relations that follow from the defining relations, which are proved in \cite[Th.~1.3]{Sav18}.  The relation \cref{squid} means that $\downarrow$ is right dual to $\uparrow$ (in the appropriate graded sense).  In fact, we also have
\begin{equation}
    \begin{tikzpicture}[anchorbase]
        \draw[<-] (-0.3,-0.5) to (-0.3,0) to[out=up,in=up,looseness=2] (0,0) to[out=down,in=down,looseness=2] (0.3,0) to (0.3,0.5);
    \end{tikzpicture}
    \ =\
    \begin{tikzpicture}[anchorbase]
        \draw[<-] (0,-0.5) to (0,0.5);
    \end{tikzpicture}
    \ ,\qquad
    \begin{tikzpicture}[anchorbase]
        \draw[<-] (-0.3,0.5) to (-0.3,0) to[out=down,in=down,looseness=2] (0,0) to[out=up,in=up,looseness=2] (0.3,0) to (0.3,-0.5);
    \end{tikzpicture}
    \ =\
    \begin{tikzpicture}[anchorbase]
        \draw[->] (0,-0.5) to (0,0.5);
    \end{tikzpicture}
    \ ,
\end{equation}
and so $\downarrow$ is also left dual to $\uparrow$.  Thus $\Heis{A}{k}$ is \emph{rigid}.  Moreover, we have that
\begin{gather}\label{slc1}
    \begin{tikzpicture}[anchorbase]
        \draw[<-] (0,-0.5) -- (0,0.5);
        \blacktoken[east]{0,0}{a};
    \end{tikzpicture}
    :=
    \begin{tikzpicture}[anchorbase]
        \draw[->] (-0.4,0.5) -- (-0.4,-0.1) arc(180:360:0.2) -- (0,0.1) arc(180:0:0.2) -- (0.4,-0.5);
        \blacktoken[west]{0,0}{a};
    \end{tikzpicture}
    \ =\
    \begin{tikzpicture}[anchorbase]
        \draw[->] (0.4,0.5) -- (0.4,-0.1) arc(360:180:0.2) -- (0,0.1) arc(0:180:0.2) -- (-0.4,-0.5);
        \blacktoken[west]{0,0}{a};
    \end{tikzpicture}
    \ ,\qquad
    \begin{tikzpicture}[anchorbase]
        \draw[<-] (0,-0.5) -- (0,0.5);
        \blackdot{0,0};
    \end{tikzpicture}
    :=
    \begin{tikzpicture}[anchorbase]
        \draw[->] (-0.4,0.5) -- (-0.4,-0.1) arc(180:360:0.2) -- (0,0.1) arc(180:0:0.2) -- (0.4,-0.5);
        \blackdot{0,0};
    \end{tikzpicture}
    \ =\
    \begin{tikzpicture}[anchorbase]
        \draw[->] (0.4,0.5) -- (0.4,-0.1) arc(360:180:0.2) -- (0,0.1) arc(0:180:0.2) -- (-0.4,-0.5);
        \blackdot{0,0};
    \end{tikzpicture}
    \ ,
    \\\label{slc2}
    \begin{tikzpicture}[anchorbase,scale=0.5]
        \draw[->] (-1.5,1.5) .. controls (-1.5,0.5) and (-1,-1) .. (0,0) .. controls (1,1) and (1.5,-0.5) .. (1.5,-1.5);
        \draw[->] (-2,1.5) .. controls (-2,-2) and (1.5,-1.5) .. (0,0) .. controls (-1.5,1.5) and (2,2) .. (2,-1.5);
    \end{tikzpicture}
    \ =\
    \begin{tikzpicture}[anchorbase]
        \draw[<-] (0.2,-0.3) to (-0.2,0.3);
        \draw[<-] (0.6,-0.3) to[out=up,in=45,looseness=2] (0,0) to[out=225,in=down,looseness=2] (-0.6,0.3);
    \end{tikzpicture}
    \ =\
    \begin{tikzpicture}[anchorbase]
        \draw[<-] (-0.2,-0.3) to (0.2,0.3);
        \draw[<-] (-0.6,-0.3) to[out=up,in=135,looseness=2] (0,0) to[out=-45,in=down,looseness=2] (0.6,0.3);
    \end{tikzpicture}
    \ =\
    \begin{tikzpicture}[anchorbase,scale=0.5]
        \draw[->] (1.5,1.5) .. controls (1.5,0.5) and (1,-1) .. (0,0) .. controls (-1,1) and (-1.5,-0.5) .. (-1.5,-1.5);
        \draw[->] (2,1.5) .. controls (2,-2) and (-1.5,-1.5) .. (0,0) .. controls (1.5,1.5) and (-2,2) .. (-2,-1.5);
    \end{tikzpicture}
    \ .
\end{gather}
These relations mean that dots, tokens, and crossings slide over all cups and caps:
\begin{align}\label{bogey}
    \begin{tikzpicture}[anchorbase,scale=.8]
        \draw[->] (0,0) -- (0,0.3) arc (180:0:.3) -- (0.6,0);
        \blacktoken[east]{0,0.28}{a};
    \end{tikzpicture}=
    \begin{tikzpicture}[anchorbase,scale=.8]
        \draw[->] (0,0) -- (0,0.3) arc (180:0:.3) -- (0.6,0);
        \blacktoken[west]{0.6,0.28}{a};
    \end{tikzpicture}
    \ ,\
    \begin{tikzpicture}[anchorbase,scale=.8]
        \draw[->] (0,0) -- (0,-0.3) arc (180:360:.3) -- (0.6,0);
        \blacktoken[east]{0,-0.28}{a};
    \end{tikzpicture}=
    \begin{tikzpicture}[anchorbase,scale=.8]
        \draw[->] (0,0) -- (0,-0.3) arc (180:360:.3) -- (0.6,0);
        \blacktoken[west]{0.6,-0.28}{a};
    \end{tikzpicture}
    \ ,\
    \begin{tikzpicture}[anchorbase,scale=.8]
        \draw[->] (0,0) -- (0,0.3) arc (180:0:.3) -- (0.6,0);
        \multblackdot[east]{0,0.28}{};
    \end{tikzpicture}=
    \begin{tikzpicture}[anchorbase,scale=.8]
        \draw[->] (0,0) -- (0,0.3) arc (180:0:.3) -- (0.6,0);
        \multblackdot[west]{0.6,0.28}{};
    \end{tikzpicture}\ ,\
    \begin{tikzpicture}[anchorbase,scale=.8]
        \draw[->] (0,0) -- (0,-0.3) arc (180:360:.3) -- (0.6,0);
        \multblackdot[east]{0,-0.28}{};
    \end{tikzpicture}=
    \begin{tikzpicture}[anchorbase,scale=.8]
        \draw[->] (0,0) -- (0,-0.3) arc (180:360:.3) -- (0.6,0);
        \multblackdot[west]{0.6,-0.28}{};
    \end{tikzpicture}\ ,\
    \begin{tikzpicture}[anchorbase,scale=.8]
        \draw[<-,thin](.6,.4) to (.1,-.3);
	   \draw[<-] (0.6,-0.3) to[out=140, in=0] (0.1,0.2);
	   \draw[-] (0.1,0.2) to[out = -180, in = 90] (-0.2,-0.3);
    \end{tikzpicture}
    =
    \begin{tikzpicture}[anchorbase,scale=.8]
        \draw[<-](-.5,.4) to (0,-.3);
	   \draw[<-] (0.3,-0.3) to[out=90, in=0] (0,0.2);
	   \draw[-] (0,0.2) to[out = -180, in = 40] (-0.5,-0.3);
    \end{tikzpicture}
    \ ,\
    \begin{tikzpicture}[anchorbase,scale=.8]
        \draw[->](-.5,-.3) to (0,.4);
        \draw[<-] (0.3,0.4) to[out=-90, in=0] (0,-0.1);
        \draw[-] (0,-0.1) to[out = 180, in = -40] (-0.5,0.4);
    \end{tikzpicture}
    \ ,\
    \begin{tikzpicture}[anchorbase,scale=.8]
        \draw[->](.6,-.3) to (.1,.4);
	      \draw[<-] (0.6,0.4) to[out=-140, in=0] (0.1,-0.1);
	\draw[-] (0.1,-0.1) to[out = 180, in = -90] (-0.2,0.4);
    \end{tikzpicture}
    \ .
\end{align}
More precisely, the cups and caps equip $\Heis{A}{k}$ with the structure of a \emph{strictly pivotal} graded supercategory: we have an isomorphism of strict graded monoidal supercategories
\begin{equation} \label{spin}
    \ast \colon \Heis{A}{k} \to \left( (\Heis{A}{k})^\op \right)^\rev.
\end{equation}
For a general morphism $f$ represented by a single diagram with homogeneous tokens, the morphism $f^*$ is given by rotating the diagram through $180^\circ$ then multiplying by $(-1)^{\binom{z}{2}}$
where $z$ is the total number odd tokens in the diagram (including ones in negatively-dotted bubbles).  Here `op' denotes the opposite supercategory and `rev' denotes the reverse supercategory (changing the order of the tensor product); these are defined as for categories, but with appropriate signs.

There is another symmetry
\begin{equation} \label{Omega}
    \Omega_k \colon \Heis{A}{k} \xrightarrow{\cong} \Heis{A}{-k}^\op
\end{equation}
that reflects diagrams (with homogeneous tokens) in the horizontal axis then multiplies by $(-1)^{x+y+\binom{z}{2}}$, where $x$ is the total number of crossings, $y$ is the total number of left cups and caps appearing in the diagram, and $z$ is the number of \emph{odd} tokens; see \cite[Lem.~2.1, Cor.~2.2]{Sav18}.

Since the relations \cref{wreathrel,affrel1,upslides1} hold in $\Heis{A}{k}$, we have a strict graded monoidal superfunctor $\imath \colon \AWC \to \Heis{A}{k}$ sending diagrams in $\AWC$ to the same diagrams in $\Heis{A}{k}$.  It will follow from the basis theorem (\cref{basis}) that $\imath$ is actually an inclusion, but we do not need that fact here.  The functor $\imath$ induces a homomorphism of graded algebras
\begin{equation} \label{imath}
    \imath_n \colon \AWA{n} \to \End_{\Heis{A}{k}}(\uparrow^{\otimes n})
\end{equation}
defined in the same way as \cref{imathaff}.  Viewing $A^\op$ as a graded Frobenius algebra with trace map $\tr$ being the same underlying linear map as for $A$, there is also a homomorphism of graded algebras
\begin{equation} \label{jmath}
    \jmath_n \colon \AWA[A^\op]{n} \to \End_{\Heis{A}{k}}(\downarrow^{\otimes n})
\end{equation}
sending $-s_i$ to the crossing of the $i$-th and $(i+1)$-st strings, $x_j$ to a dot on the $j$-th string, and $1^{\otimes {(n-j)}} \otimes a \otimes 1^{\otimes (j-1)}$ to a token labeled $a$ on the $j$-th string.

The basis $\B_A$ is also a basis for $A^\op$, and the dual basis with respect to the Frobenius form on $A^\op$ is the just same as the dual basis for $A$. It follows that the elements $\tau_{i,j} \in \AWA[A^\op]{n}$ are given by exactly the same formula \cref{tau} as for $\AWA[A]{n}$, and we can use similar diagrams to \cref{brexit}, with downward strings replacing the upward ones, to represent the images of these elements under $\jmath_n$.  Tokens teleport across these new downward teleporters
in just the same way as in \cref{tokteleport}.  We may also draw teleporters with endpoints on oppositely oriented strings, interpreting them as usual by putting the label $b$ on the higher of the tokens and $b^\vee$ on the lower one, summing over all $b \in \B_A$. The way that dots teleport across these is slightly different; for example, we have that
\[

  \ .
\]
In all teleporters we continue to add a sign corresponding to odd tokens appearing elsewhere in the diagram vertically between the endpoints of the teleporter, as we did for teleporters on upward strands.

We will sometimes draw diagrams with tokens and dots at the critical point of cups and caps; in light of \cref{bogey}, there is no ambiguity.  We will also often draw dots and tokens on bubbles on the downward as well as the upward strand, extending this also to the negatively-dotted bubbles \cref{fakel,faker} in the obvious way. Due to the strictly pivotal structure, we have that
\[
  %
,
  \quad
  r \in \Z,\ a,b \in A.
\]
Consequently, a bubble labeled by token $a \in A$ actually only depends on $\cocenter{a} \in C(A)$.  We will also come across bubbles joined to another string by a teleporter. Any token on such a bubble can be teleported away from the bubble, and then the resulting ``telebubble'' can then be written more simply  as a sum involving the basis $\B_{C(A)}$ for $C(A)$ and the dual basis $\{\cocenter{b}^\vee : \cocenter{b} \in \B_{C(A)}\}$ for $Z(A)$. For example:
\[
    %
.
\]
We also have
\[
    %
.
\]
Also, if there is more than one teleporter joining the bubble to another string, one can use teleporting to remove all but one of them. For example:
\begin{equation}\label{sta}
  %
  \ ,
\end{equation}
where $a^\dagger$ is as in \cref{adag}.

The final relations to be recalled from \cite[Th.~1.3]{Sav18} are as follows. The bubbles (both genuine and negatively-dotted) satisfy the \emph{infinite Grassmannian relations}: for $a,b \in A$, we have
\begin{gather} \label{weed}
  \cbubble{a}{r} = -\delta_{r,k-1} \tr(a)
  \quad \text{if } r \le k-1,
  \quad
  \ccbubble{a}{r} = \delta_{r,-k-1} \tr(a)
  \quad \text{if } r \le -k-1,
  \\ \label{infgrass}
  \sum_{r \in \Z}
  %
  \ = - \delta_{n,0} \tr(ab) 1_\one.
\end{gather}
We also have the \emph{curl relations}
\begin{equation} \label{curls}
  %
  \ ,
\end{equation}
the \emph{alternating braid relation}
\begin{equation} \label{altbraid}
  %
  \ ,
\end{equation}
and the \emph{bubble slide relations}
\begin{equation} \label{bubslide}
  %
 \ .
\end{equation}
(For these last two relations, we have used \cref{sta} to simplify the formulae from \cite{Sav18}.)

As in \cite{BSW-HKM}, when working with dotted bubbles, it is often helpful to assemble them all into a single generating function.  For an indeterminate $u$ and $a \in A$, let
\begin{align} \label{hedge1}
  %
  &:=
  \sum_{n \in \Z} \ccbubble{a}{n} u^{-n-1}
  \in \tr(a) u^k 1_\one + u^{k-1} \End_{\Heis{A}{k}}(\one) \llbracket u^{-1} \rrbracket.
\end{align}
Then the infinite Grassmannian relation \cref{infgrass} implies that
\begin{equation} \label{eyes}
  %
.
\end{equation}
In order to write other relations in terms of generating functions, we need to unify our notation for dots and tokens, adopting the notation
\[
  %
\]
to denote the composition of the endomorphisms of $\uparrow$ defined by a token labeled by $a \in A$ and a dot of multiplicity $n$.  Then we can also label tokens by polynomials $a_n x^n + \dotsb + a_1 x + a_0 \in A[x]$, meaning the sum of morphisms defined by the tokens labeled $a_n x^n, a_1 x,\dotsc, x_0$, or even by Laurent series in $A[x]\Laurent{u^{-1}}$.  For example:
\[
    %
    + \dotsb.
\]
The caveat here is that, whereas tokens labeled by elements of $A\Laurent{u^{-1}}$ slide through crossings and can be teleported, the more general tokens labeled by elements of $A[x]\Laurent{u^{-1}}$ no longer have these properties in general.  For a Laurent series $p(u)$, we let $[p(u)]_{u^r}$ denote its $u^r$-coefficient, and we write $[p(u)]_{u^{<0}}$ for $\sum_{n < 0} [p(u)]_{u^n} u^n$.  Then the above relations imply the following:
\begin{gather}\label{wednesdayam}
  %
  \ .
\end{gather}

For the remainder of this section, we fix $l,m \in \Z$ and set $k=l+m$.  We define a strict graded monoidal supercategory $\Heis[blue]{A}{l} \xodot{2d} \Heis[red]{A}{m}$ as in the previous section. Thus, its underlying strict monoidal supercategory is the symmetric product of $\blue{\Heis[blue]{A}{l}}$ and $\red{\Heis[red]{A}{m}}$. This has the two-color crossings of all possible orientations (up, right, left and down):
\[
    %
    \ \right)^{-1}
    \ .
\]
These satisfy various commuting relations, including ``pitchfork relations'' that are two-colored versions of the last two relations from \cref{bogey}.  We view $\Heis[blue]{A}{l} \xodot{2d} \Heis[red]{A}{m}$ as a graded supercategory by declaring that the parities and degrees of its generating morphisms are defined as usual for the one-color generators, and the parities and degrees of the two-color crossings are as in \cref{par}.  Then we pass to the localization $\Heis[blue]{A}{l} \barodot \Heis[red]{A}{m}$ by adjoining dot dumbbells like in \cref{dotdumb}.  We also get dot dumbbells with endpoints on downward as well as upward strings; this is similar to the situation discussed in \cite[Sec.~4]{BSW-K0}.  Also we introduce the shorthands that are the box dumbbells as in \cref{sqdumb}, but again we now allow the strings to be oriented downward as well as upward.  Finally, we introduce the following \emph{internal bubbles} in $\Heis[blue]{A}{l} \barodot \Heis[red]{A}{m}$:
\begin{gather} \label{intblue}
  \begin{tikzpicture}[anchorbase]
    \draw[->,red] (0,-0.5) to (0,0.5);
    \intleft{blue}{0,0};
  \end{tikzpicture}
  :=
  \sum_{r \ge 0}
  \begin{tikzpicture}[anchorbase]
    \draw[blue,->] (-0.5,0) arc(90:450:0.2);
    \multbluedot[east]{-0.7,-0.2}{-r-1};
    \draw[->,red] (0,-0.5) to (0,0.5);
    \multreddot[west]{0,0.2}{r};
    \teleport{blue}{red}{-0.3,-0.2}{0,-0.2};
  \end{tikzpicture}
  +
  \begin{tikzpicture}[anchorbase]
    \draw[->,blue] (-0.2,0) arc(180:540:0.2);
    \draw[->,red] (0.5,-0.5) to (0.5,0.5);
    \dotdumb{0.141,0.141}{0.5,0.141};
    \teleport{blue}{red}{0.141,-0.141}{0.5,-0.141};
  \end{tikzpicture}
  \ ,\qquad
  \begin{tikzpicture}[anchorbase]
    \draw[->,red] (0,-0.5) to (0,0.5);
    \intright{blue}{0,0};
  \end{tikzpicture}
  :=
  \sum_{r \ge 0}
  \begin{tikzpicture}[anchorbase]
    \draw[blue,->] (0.5,0) arc(90:-270:0.2);
    \multbluedot[west]{0.7,-0.2}{-r-1};
    \draw[->,red] (0,-0.5) to (0,0.5);
    \multreddot[east]{0,0.2}{r};
    \teleport{blue}{red}{0.3,-0.2}{0,-0.2};
  \end{tikzpicture}
  +
  \begin{tikzpicture}[anchorbase]
    \draw[->,blue] (0.2,0) arc(0:-360:0.2);
    \draw[->,red] (-0.5,-0.5) to (-0.5,0.5);
    \dotdumb{-0.141,0.141}{-0.5,0.141};
    \teleport{blue}{red}{-0.141,-0.141}{-0.5,-0.141};
  \end{tikzpicture}
  \ ,
  \\ \label{intred}
  \begin{tikzpicture}[anchorbase]
    \draw[->,blue] (0,-0.5) to (0,0.5);
    \intleft{red}{0,0};
  \end{tikzpicture}
  :=
  \sum_{r \ge 0}
  \begin{tikzpicture}[anchorbase]
    \draw[red,->] (-0.5,0) arc(90:450:0.2);
    \multreddot[east]{-0.7,-0.2}{-r-1};
    \draw[->,blue] (0,-0.5) to (0,0.5);
    \multbluedot[west]{0,0.2}{r};
    \teleport{red}{blue}{-0.3,-0.2}{0,-0.2};
  \end{tikzpicture}
  -
  \begin{tikzpicture}[anchorbase]
    \draw[->,red] (-0.2,0) arc(180:540:0.2);
    \draw[->,blue] (0.5,-0.5) to (0.5,0.5);
    \dotdumb{0.141,0.141}{0.5,0.141};
    \teleport{red}{blue}{0.141,-0.141}{0.5,-0.141};
  \end{tikzpicture}
  \ ,\qquad
  \begin{tikzpicture}[anchorbase]
    \draw[->,blue] (0,-0.5) to (0,0.5);
    \intright{red}{0,0};
  \end{tikzpicture}
  :=
  \sum_{r \ge 0}
  \begin{tikzpicture}[anchorbase]
    \draw[red,->] (0.5,0) arc(90:-270:0.2);
    \multreddot[west]{0.7,-0.2}{-r-1};
    \draw[->,blue] (0,-0.5) to (0,0.5);
    \multbluedot[east]{0,0.2}{r};
    \teleport{red}{blue}{0.3,-0.2}{0,-0.2};
  \end{tikzpicture}
  -
  \begin{tikzpicture}[anchorbase]
    \draw[->,red] (0.2,0) arc(0:-360:0.2);
    \draw[->,blue] (-0.5,-0.5) to (-0.5,0.5);
    \dotdumb{-0.141,0.141}{-0.5,0.141};
    \teleport{red}{blue}{-0.141,-0.141}{-0.5,-0.141};
  \end{tikzpicture}
  \ .
\end{gather}
(One can also formulate these definitions using generating functions, thereby eliminating the summations in \cref{intblue,intred}, but we did not find that this led to significant benefits in the arguments below.)  Note that the internal bubbles are even with degrees
\[
  \deg \left(
    \begin{tikzpicture}[anchorbase]
      \draw[->,red] (0,-0.5) to (0,0.5);
      \intleft{blue}{0,0};
    \end{tikzpicture}
  \right)
  = 2 l \dA,\quad
  \deg \left(
    \begin{tikzpicture}[anchorbase]
      \draw[->,red] (0,-0.5) to (0,0.5);
      \intright{blue}{0,0};
    \end{tikzpicture}
  \right)
  = - 2 l \dA,\quad
  \deg \left(
    \begin{tikzpicture}[anchorbase]
      \draw[->,blue] (0,-0.5) to (0,0.5);
      \intleft{red}{0,0};
    \end{tikzpicture}
  \right)
  = 2 m \dA,\quad
  \deg \left(
    \begin{tikzpicture}[anchorbase]
      \draw[->,blue] (0,-0.5) to (0,0.5);
      \intright{red}{0,0};
    \end{tikzpicture}
  \right)
  = - 2 m \dA.
\]
Internal bubbles commute with tokens (using \cref{tokteleport} and the fact that tokens slide around bubbles), dots, and other internal bubbles.

There is a strictly pivotal structure on $\Heis[blue]{A}{l} \barodot \Heis[red]{A}{m}$, which is defined in a similar way the one on $\Heis{A}{k}$ discussed above.  In particular, we have the duals of the internal bubbles:
\[
    \begin{tikzpicture}[anchorbase]
        \draw[<-,red] (0,-0.5) to (0,0.5);
        \intleft{blue}{0,0};
    \end{tikzpicture}
    :=
    \left(\
        \begin{tikzpicture}[anchorbase]
            \draw[->,red] (0,-0.5) to (0,0.5);
            \intleft{blue}{0,0};
        \end{tikzpicture}
    \ \right)^*
    \ ,\qquad
    \begin{tikzpicture}[anchorbase]
        \draw[<-,red] (0,-0.5) to (0,0.5);
        \intright{blue}{0,0};
    \end{tikzpicture}
    :=
    \left(\
        \begin{tikzpicture}[anchorbase]
            \draw[->,red] (0,-0.5) to (0,0.5);
            \intright{blue}{0,0};
        \end{tikzpicture}
    \ \right)^*
    \ ,\qquad
    \begin{tikzpicture}[anchorbase]
        \draw[<-,blue] (0,-0.5) to (0,0.5);
        \intleft{red}{0,0};
    \end{tikzpicture}
    :=
    \left(\
        \begin{tikzpicture}[anchorbase]
            \draw[->,blue] (0,-0.5) to (0,0.5);
            \intleft{red}{0,0};
        \end{tikzpicture}
    \ \right)^*
    \ ,\qquad
    \begin{tikzpicture}[anchorbase]
        \draw[<-,blue] (0,-0.5) to (0,0.5);
        \intright{red}{0,0};
    \end{tikzpicture}
    :=
    \left(\
        \begin{tikzpicture}[anchorbase]
            \draw[->,blue] (0,-0.5) to (0,0.5);
            \intright{red}{0,0};
        \end{tikzpicture}
    \ \right)^*
    \ .
\]

As well as the duality $*$ arising from this strictly pivotal structure, there are two other useful symmetries:
\begin{align}\label{wormy}
    \flip &\colon \Heis[blue]{A}{l} \barodotnonumber \Heis[red]{A}{m}\xrightarrow{\cong} \Heis[blue]{A}{m} \barodotnonumber \Heis[red]{A}{l},
    \\
    \Omega_{{\color{blue} l}|{\color{red} m}} &\colon \Heis[blue]{A}{l} \barodot \Heis[red]{A}{m} \xrightarrow{\cong} \left( \Heis[blue]{A}{-l} \barodot \Heis[red]{A}{-m} \right)^\op.
\end{align}
In \cref{wormy}, we view the domain and codomain as monoidal supercategories, not graded, as we did in \cref{flippy}.  The symmetry $\flip$ is a monoidal superfunctor rather than a graded monoidal functor since it preserves the parities but not the degrees of two-color crossings. It is defined on diagrams by switching the colors blue and red, then multiplying by $(-1)^w$, where $w$ is the total number of dumbbells in the picture. It interchanges the internal bubbles in \cref{intblue} with those in \cref{intred}.  The isomorphism $\Omega_{{\color{blue} l}|{\color{red} m}}$, which does preserve degrees, is defined in a similar way to \cref{Omega}.  It sends a morphism $f$ represented by a single diagram (with homogeneous tokens) to $\Omega_{{\color{blue} l}|{\color{red} m}}(f)^\op$ where $\Omega_{{\color{blue}  l}|{\color{red} m}}(f)$ is the morphism defined by reflecting the diagram in a horizontal axis then multiplying by $(-1)^{x+y+\binom{z}{2}}$, where $x$ is the number of one-color crossings, $y$ is the number of left cups and caps (including ones in negatively-dotted and internal bubbles), and $z$ is the total number of \emph{odd} tokens.

The following seven lemmas are analogous to Lemmas 5.6 to 5.12 in \cite{BSW-K0}, and the proofs are similar, the main difference coming from the fact that we need to add teleporters in appropriate places.  We provide the details for the first three and omit the remaining proofs.

\begin{lem} \label{wormhole}
  We have
  \[

    \ .
  \end{multline*}
  Thus we have
  \[
    %
    \ .
  \]
\end{lem}

\begin{proof}
  By the definitions \cref{intblue,intred}, the left-hand side is
  \begin{multline*}
    \sum_{r \ge 0}
    %
    ,
  \end{multline*}
  which simplifies to the sum on the right-hand side.
\end{proof}

\begin{lem}
  We have
  \[
    %
    \ .
  \]
\end{lem}

\begin{proof}
  Using the definition \cref{intred}, the left-hand side is equal to
  \begin{multline*}
    \sum_{r \ge 0}
    %
    \ .
  \end{multline*}
  By the definition \cref{intred} and \cref{tokteleport}, this is equal to the right-hand side.
\end{proof}

\begin{lem} \label{desk}
  We have
  \[
    %
    \ .
  \]
\end{lem}

\begin{proof}
  The proof is analogous to that of \cite[Lem.~5.9]{BSW-K0}.
\end{proof}

\begin{lem} \label{centipede}
  We have
  \[
    %
    \ .
  \]
\end{lem}

\begin{proof}
  The proof is analogous to that of \cite[Lem.~5.10]{BSW-K0}.
\end{proof}

\begin{lem} \label{snail}
  We have
  \[
    %
    \ .
  \]
\end{lem}

\begin{proof}
  The proof is analogous to that of \cite[Lem.~5.11]{BSW-K0}.
\end{proof}

\begin{lem} \label{ladder}
  We have
  \[
    %
    \ .
  \]
\end{lem}

\begin{proof}
  The proof is analogous to that of \cite[Lem.~5.12]{BSW-K0}.
\end{proof}

Now we can define an analogue of the categorical comultiplication $\Delta$ from \cref{AWcoprod}.  We do this in two steps: first, we define it at the level of supercategories; then, in the corollary, we upgrade the result to graded supercategories by passing to the $(Q,\Pi)$-envelope.

\begin{theo} \label{bunny}
  There is a unique strict monoidal superfunctor
  \[
    \cDelt{l}{m} \colon
    \Heis{A}{k} \to \Add \left( \Heis[blue]{A}{l} \barodotnonumber \color{red} \Heis[red]{A}{m} \right)
  \]
  such that $\uparrow \mapsto \upblue \oplus \upred$, $\downarrow \mapsto \downblue \oplus \downred$, and on morphisms
  \begin{gather} \label{fuzzy1}
    \begin{tikzpicture}[anchorbase]
      \draw[->] (0,-0.3) to (0,0.3);
      \blackdot{0,0};
    \end{tikzpicture}
    \mapsto\
    \begin{tikzpicture}[anchorbase]
      \draw[->,blue] (0,-0.3) to (0,0.3);
      \bluedot{0,0};
    \end{tikzpicture}
    \ +\
    \begin{tikzpicture}[anchorbase]
      \draw[->,red] (0,-0.3) to (0,0.3);
      \reddot{0,0};
    \end{tikzpicture}
    \ ,\qquad
    \begin{tikzpicture}[anchorbase]
      \draw[->] (0,-0.3) to (0,0.3);
      \blacktoken[west]{0,0}{a};
    \end{tikzpicture}
    \mapsto\
    \begin{tikzpicture}[anchorbase]
      \draw[->,blue] (0,-0.3) to (0,0.3);
      \bluetoken[west]{0,0}{a};
    \end{tikzpicture}
    +\
    \begin{tikzpicture}[anchorbase]
      \draw[->,red] (0,-0.3) to (0,0.3);
      \redtoken[west]{0,0}{a};
    \end{tikzpicture}
    \ ,
    \\ \label{fuzzy3}
    \begin{tikzpicture}[anchorbase]
      \draw[->] (-0.4,-0.4) to (0.4,0.4);
      \draw[->] (0.4,-0.4) to (-0.4,0.4);
    \end{tikzpicture}
    \mapsto
    \begin{tikzpicture}[anchorbase]
      \draw[->,blue] (-0.4,-0.4) to (0.4,0.4);
      \draw[->,blue] (0.4,-0.4) to (-0.4,0.4);
    \end{tikzpicture}
    \ +\
    \begin{tikzpicture}[anchorbase]
      \draw[->,red] (-0.4,-0.4) to (0.4,0.4);
      \draw[->,red] (0.4,-0.4) to (-0.4,0.4);
    \end{tikzpicture}
    \ +\
    \begin{tikzpicture}[anchorbase]
      \draw[->,blue] (-0.4,-0.4) to (0.4,0.4);
      \draw[->,red] (0.4,-0.4) to (-0.4,0.4);
      \dotdumb{-0.2,0.2}{0.2,0.2};
    \end{tikzpicture}
    \ +\
    \begin{tikzpicture}[anchorbase]
      \draw[->,red] (-0.4,-0.4) to (0.4,0.4);
      \draw[->,blue] (0.4,-0.4) to (-0.4,0.4);
      \sqdumb{-0.2,0.2}{0.2,0.2};
    \end{tikzpicture}
    \ -
    \begin{tikzpicture}[anchorbase]
      \draw[->,blue] (-0.2,-0.4) to (-0.2,0.4);
      \draw[->,red] (0.2,-0.4) to (0.2,0.4);
      \teleport{blue}{red}{-0.2,-0.2}{0.2,-0.2};
      \dotdumb{-0.2,0.15}{0.2,0.15};
    \end{tikzpicture}
    +
    \begin{tikzpicture}[anchorbase]
      \draw[->,red] (-0.2,-0.4) to (-0.2,0.4);
      \draw[->,blue] (0.2,-0.4) to (0.2,0.4);
      \teleport{red}{blue}{-0.2,-0.2}{0.2,-0.2};
      \dotdumb{-0.2,0.15}{0.2,0.15};
    \end{tikzpicture}
    \ ,
  \\ \label{fuzzy2}
    \begin{tikzpicture}[anchorbase]
      \draw[->] (0,-.2) -- (0,0) arc (180:0:.3) -- (.6,-.2);
    \end{tikzpicture}
    \mapsto
    \begin{tikzpicture}[anchorbase]
      \draw[->,blue] (0,-.2) -- (0,0) arc (180:0:.3) -- (.6,-.2);
    \end{tikzpicture}
    -
    \begin{tikzpicture}[anchorbase]
      \draw[->,red] (-0.3,-0.5) -- (-0.3,0) arc (180:0:0.3) -- (0.3,-0.5);
      \intright{blue}{-0.3,-0.1};
    \end{tikzpicture}
    \ ,\qquad
    \begin{tikzpicture}[anchorbase]
      \draw[->] (0,.2) -- (0,0) arc (180:360:.3) -- (.6,.2);
    \end{tikzpicture}
    \mapsto
    \begin{tikzpicture}[anchorbase]
      \draw[->,blue] (0,.2) -- (0,0) arc (180:360:.3) -- (.6,.2);
    \end{tikzpicture}
    +
    \begin{tikzpicture}[anchorbase]
      \draw[->,red] (-0.3,0.5) -- (-0.3,0) arc (180:360:0.3) -- (0.3,.5);
      \intleft{blue}{-0.3,0.1};
    \end{tikzpicture}
    \ .
  \end{gather}
  Moreover, $\cDelt{l}{m}$ satisfies the following for all $a \in A$ and $n \in \Z$:
  \begin{gather}  \label{gloves}
    \begin{tikzpicture}[anchorbase]
      \draw[<-] (0,-.2) -- (0,0) arc (180:0:.3) -- (.6,-.2);
    \end{tikzpicture}
    \mapsto
    \begin{tikzpicture}[anchorbase]
      \draw[<-,blue] (-0.3,-0.5) -- (-0.3,0) arc (180:0:0.3) -- (0.3,-0.5);
      \intleft{red}{0.3,-0.1};
    \end{tikzpicture}
    +
    \begin{tikzpicture}[anchorbase]
      \draw[<-,red] (0,-.2) -- (0,0) arc (180:0:.3) -- (.6,-.2);
    \end{tikzpicture}
    \ ,\qquad
    \begin{tikzpicture}[anchorbase]
      \draw[<-] (0,.2) -- (0,0) arc (180:360:.3) -- (.6,.2);
    \end{tikzpicture}
    \mapsto -
    \begin{tikzpicture}[anchorbase]
      \draw[<-,blue] (-0.3,0.5) -- (-0.3,0) arc (180:360:0.3) -- (0.3,.5);
      \intright{red}{0.3,0.1};
    \end{tikzpicture}
    +
    \begin{tikzpicture}[anchorbase]
      \draw[<-,red] (0,.2) -- (0,0) arc (180:360:.3) -- (.6,.2);
    \end{tikzpicture}
    \ ,
    \\ \label{shoes}
    \ccbubble{a}{n}
    \mapsto \sum_{r \in \Z}\
    \begin{tikzpicture}[anchorbase]
      \draw[->,blue] (0,0.55) arc(90:450:0.2);
      \draw[->,red] (0,-0.55) arc(-90:270:0.2);
      \multbluedot[west]{0.2,0.35}{r};
      \multreddot[west]{0.2,-0.35}{n-r-1};
      \bluetoken[east]{-0.2,0.35}{a};
      \teleport{blue}{red}{0,0.15}{0,-0.15};
    \end{tikzpicture}
    \ ,\qquad
    \cbubble{a}{n}
    \mapsto - \sum_{r \in \Z}
    \begin{tikzpicture}[anchorbase]
      \draw[<-,blue] (0,0.55) arc(90:450:0.2);
      \draw[<-,red] (0,-0.55) arc(-90:270:0.2);
      \multbluedot[west]{0.2,0.35}{r};
      \multreddot[west]{0.2,-0.35}{n-r-1};
      \bluetoken[east]{-0.2,0.35}{a};
      \teleport{blue}{red}{0,0.15}{0,-0.15};
    \end{tikzpicture}
    \ .
  \end{gather}
\end{theo}

\begin{proof}
  We define $\cDelt{l}{m}$ on the generating morphisms from \cref{star} by \cref{fuzzy1,fuzzy3,fuzzy2,gloves}, and must check that the images under $\cDelt{l}{m}$ of the defining relations \cref{squid,wreathrel,affrel1,upslides1,squish,chew,curly} all hold in the category $\Add (\Heis[blue]{A}{l} \barodotnonumber \Heis[red]{A}{m})$.  In addition, we will verify that $\cDelt{l}{m}$ satisfies \cref{shoes}.  Using \cref{lock}, this is sufficient to prove the theorem as stated.

  We already verified \cref{wreathrel,affrel1,upslides1} in \cref{AWcoprod}.  The relation \cref{squid} is straightforward to verify since all of the matrices involved are diagonal.  One just needs to use \cref{wormhole} and its image under the symmetry \cref{wormy}.

  Now we check \cref{shoes}.  Note in terms of bubble generating functions that it is equivalent to
  \begin{equation} \label{circus}
    \cDelt{l}{m}
    \left(

    \ .
  \end{equation}
  First suppose that $k \ge 0$ and consider the clockwise bubble $\cbubble{a}{n}$.  If this is a negatively-dotted bubble, i.e.\ $n < 0$, then it is a scalar by the definition \cref{faker} and the assumption on $k$.  This scalar is zero (and \cref{shoes} is trivial to verify) unless $k=0$ and $n=-1$.  When $k=0$, we have that
  \[
    \sum_{r \in \Z}
    %
,
  \]
  again verifying \cref{shoes}.  Now, still assuming $k \ge 0$, consider \cref{shoes} for the counterclockwise bubbles.  We have already proved that
  \[
    g(a;u) :=    \cDelt{l}{m}
    \left(
      %
 \in \tr(a) u^{-k}1_\one + u^{-k-1} \End_{\Heis[blue]{A}{l} \barodotnonumber
    \color{red} \Heis[red]{A}{m} }(\one)\llbracket
    u^{-1}\rrbracket
    \ .
  \]
  Applying \cref{eyes}, we have that
  \begin{align*}
    \sum_{b \in \B_A} \cDelt{l}{m}
    \left(
      %
    = \tr(a) 1_\one.
  \end{align*}
  Now we observe that there is a unique morphism $f(a;u) \in \tr(a) u^k 1_\one + u^{k-1} \End_{\Heis[blue]{A}{l} \barodotnonumber \color{red} \Heis[red]{A}{m} }(\one)\llbracket u^{-1}\rrbracket$ such that $\sum_{b \in \B_A} f(ab;u) g(b^\vee;u) = \tr(a) 1_\one$; see the proof of Lemma~\ref{controversy} below for a similar situation.   So we have done enough to prove that
  \[
    f(a;u) =
    \cDelt{l}{m}
    \left(
      %
    \ .
  \]
  This completes the proof of \cref{shoes} when $k \ge 0$.  The case $k \le 0$ is analogous; one starts by proving the relation for the counterclockwise bubble (using the relation obtained from \cref{dizzy} by applying the symmetry $\Omega_{{\color{blue} l}|{\color{red} m}}$), and then use \cref{eyes} to prove it for the clockwise bubble.

  Next we check \cref{chew}.  For $a \in A$ and $0 \le n < k$, we have
  \[
    \cDelt{l}{m}
    \left(
      %
    \overset{\cref{weed}}{=} - \delta_{n,k-1} \sum_{b \in \B_A} \tr(ab) \tr(b^\vee)
    = - \delta_{n,k-1} \tr(a),
  \]
  verifying the first relation in \cref{chew}.  The proof of the second relation is similar.

  Next we verify \cref{curly}.  We first consider the right curl, so
  $k \ge 0$.  Using \cref{desk} and its image under $\flip$, we have
  \begin{multline*}
    \cDelt{l}{m}
    \left(\
      %
      \
    \right).
  \end{multline*}
  The proof for the left curl, when $k \le 0$, is similar; it uses the identity obtained from \cref{desk} by applying the symmetry $\Omega_{{\color{blue} l}|{\color{red} m}}$ then rotating through $180\degree$.

  Finally, we must verify \cref{squish}.  We give the details for the first relation.  The proof of the second is analogous, using the identities obtained from \cref{centipede,snail,ladder} by applying $* \circ \Omega_{\blue{l}|\red{m}}$.  First note that, by definition \cref{fall}, $\cDelt{l}{m}$ sends
  \begin{align*}
    %
    \ .
  \end{align*}
  With this, it is straightforward to compute the image under $\cDelt{l}{m}$ of the left-hand side of \cref{squish}.  We can also compute the image of the right-hand side using \cref{shoes}.  Comparing the matrix entries of the resulting morphisms, we are reduced to verifying the identities
  \begin{gather*}
    -
    %
    \ ,
  \end{gather*}
  plus the images of the first two under the symmetry flip, followed by rotation through $180\degree$.  To prove the first identity, first simplify it by multiplying on the bottom of the left string by a clockwise internal bubble and using \cref{wormhole}.  The resulting identity then follows from \cref{centipede}.  The second identity follows from \cref{snail}.  For the third identity, we have
  \[
    %
    \ ,
  \]
  where the last equality follows from \cref{ladder,wormhole}.  The proof of the fourth identity is analogous, using the images of \cref{ladder,wormhole} under the symmetry flip.
\end{proof}

The functor in \cref{bunny} does not respect grading. In the next
corollary, we explain how to rectify this
 by passing to the $(Q,\Pi)$-envelope, recalling the notation (\ref{covid2}).

\begin{cor} \label{bunnycor}
    There is a unique strict graded monoidal functor
    \[
        \tilde\Delta_{\blue{l}|\red{m}} \colon
        \Heisenv{A}{k} \to \Add \left(\left( \Heis[blue]{A}{l} \barodot \Heis[red]{A}{m}\right)_{q,\pi} \right)
    \]
    such that $Q^r \Pi^t \uparrow \mapsto Q^r \Pi^t \upblue \oplus Q^{r-ld} \Pi^t \upred$, $Q^r \Pi^t \downarrow \mapsto Q^{r+md} \Pi^t \downblue \oplus Q^r \Pi^t \downred$, and on morphisms
    \begin{gather} \label{fuzzygraded1}
        \begin{tikzpicture}[anchorbase]
          \draw[->] (0,-0.3) to (0,0.3);
          \blackdot{0,0};
          \shiftline{-0.2,-0.3}{0.2,-0.3}{r,t};
          \shiftline{-0.2,0.3}{0.2,0.3}{s,t};
        \end{tikzpicture}
        \mapsto
        \begin{tikzpicture}[anchorbase]
          \draw[->,blue] (0,-0.3) to (0,0.3);
          \bluedot{0,0};
          \shiftline{-0.2,-0.3}{0.2,-0.3}{r,t};
          \shiftline{-0.2,0.3}{0.2,0.3}{s,t};
        \end{tikzpicture}
        \ +\
        \begin{tikzpicture}[anchorbase]
          \draw[->,red] (0,-0.3) to (0,0.3);
          \reddot{0,0};
          \shiftline{-0.2,-0.3}{0.2,-0.3}{r-ld,t};
          \shiftline{-0.2,0.3}{0.2,0.3}{s-ld,t};
        \end{tikzpicture}
        \ ,\qquad
        \begin{tikzpicture}[anchorbase]
          \draw[->] (0,-0.3) to (0,0.3);
          \blacktoken[west]{0,0}{a};
          \shiftline{-0.2,-0.3}{0.2,-0.3}{r,t};
          \shiftline{-0.2,0.3}{0.2,0.3}{s,t};
        \end{tikzpicture}
        \mapsto
        \begin{tikzpicture}[anchorbase]
          \draw[->,blue] (0,-0.3) to (0,0.3);
          \bluetoken[west]{0,0}{a};
          \shiftline{-0.2,-0.3}{0.2,-0.3}{r,t};
          \shiftline{-0.2,0.3}{0.2,0.3}{s,t};
        \end{tikzpicture}
        +\
        \begin{tikzpicture}[anchorbase]
          \draw[->,red] (0,-0.3) to (0,0.3);
          \redtoken[west]{0,0}{a};
          \shiftline{-0.2,-0.3}{0.2,-0.3}{r-ld,t};
          \shiftline{-0.2,0.3}{0.2,0.3}{s-ld,t};
        \end{tikzpicture}
        \ ,
        \\ \label{fuzzygraded3}
        \begin{tikzpicture}[anchorbase]
          \draw[->] (-0.4,-0.4) to (0.4,0.4);
          \draw[->] (0.4,-0.4) to (-0.4,0.4);
          \shiftline{-0.5,-0.4}{0.5,-0.4}{r,t};
          \shiftline{-0.5,0.4}{0.5,0.4}{s,t};
        \end{tikzpicture}\!
        \mapsto
        \begin{tikzpicture}[anchorbase]
          \draw[->,blue] (-0.4,-0.4) to (0.4,0.4);
          \draw[->,blue] (0.4,-0.4) to (-0.4,0.4);
          \shiftline{-0.5,-0.4}{0.5,-0.4}{r,t};
          \shiftline{-0.5,0.4}{0.5,0.4}{s,t};
        \end{tikzpicture}
        \!\! +\!
        \begin{tikzpicture}[anchorbase]
          \draw[->,red] (-0.4,-0.4) to (0.4,0.4);
          \draw[->,red] (0.4,-0.4) to (-0.4,0.4);
          \shiftline{-0.5,-0.4}{0.5,-0.4}{r-2ld,t};
          \shiftline{-0.5,0.4}{0.5,0.4}{s-2ld,t};
        \end{tikzpicture}
        \!\! +\!
        \begin{tikzpicture}[anchorbase]
          \draw[->,blue] (-0.4,-0.4) to (0.4,0.4);
          \draw[->,red] (0.4,-0.4) to (-0.4,0.4);
          \dotdumb{-0.2,0.2}{0.2,0.2};
          \shiftline{-0.5,-0.4}{0.5,-0.4}{r-ld,t};
          \shiftline{-0.5,0.4}{0.5,0.4}{s-ld,t};
        \end{tikzpicture}
        \!\! +\!
        \begin{tikzpicture}[anchorbase]
          \draw[->,red] (-0.4,-0.4) to (0.4,0.4);
          \draw[->,blue] (0.4,-0.4) to (-0.4,0.4);
          \sqdumb{-0.2,0.2}{0.2,0.2};
          \shiftline{-0.5,-0.4}{0.5,-0.4}{r-ld,t};
          \shiftline{-0.5,0.4}{0.5,0.4}{s-ld,t};
        \end{tikzpicture}
        \!\! -
        \begin{tikzpicture}[anchorbase]
          \draw[->,blue] (-0.2,-0.4) to (-0.2,0.4);
          \draw[->,red] (0.2,-0.4) to (0.2,0.4);
          \teleport{blue}{red}{-0.2,-0.2}{0.2,-0.2};
          \dotdumb{-0.2,0.15}{0.2,0.15};
          \shiftline{-0.3,-0.4}{0.3,-0.4}{r-ld,t};
          \shiftline{-0.3,0.4}{0.3,0.4}{s-ld,t};
        \end{tikzpicture}
        \!+
        \begin{tikzpicture}[anchorbase]
          \draw[->,red] (-0.2,-0.4) to (-0.2,0.4);
          \draw[->,blue] (0.2,-0.4) to (0.2,0.4);
          \teleport{red}{blue}{-0.2,-0.2}{0.2,-0.2};
          \dotdumb{-0.2,0.15}{0.2,0.15};
          \shiftline{-0.3,-0.4}{0.3,-0.4}{r-ld,t};
          \shiftline{-0.3,0.4}{0.3,0.4}{s-ld,t};
        \end{tikzpicture}
        \! ,
        \\ \label{fuzzygraded2}
        \begin{tikzpicture}[anchorbase]
          \draw[->] (0,-.2) -- (0,0) arc (180:0:.3) -- (.6,-.2);
          \shiftline{-0.1,-0.2}{0.7,-0.2}{r,t};
          \shiftline{-0.1,0.5}{0.7,0.5}{s,t};
        \end{tikzpicture}
        \mapsto 
        \begin{tikzpicture}[anchorbase]
          \draw[->,blue] (0,-.2) -- (0,0) arc (180:0:.3) -- (.6,-.2);
          \shiftline{-0.1,-0.2}{0.7,-0.2}{r+md,t};
          \shiftline{-0.1,0.5}{0.7,0.5}{s,t};
        \end{tikzpicture}
        -
        \begin{tikzpicture}[anchorbase]
          \draw[->,red] (-0.3,-0.5) -- (-0.3,0) arc (180:0:0.3) -- (0.3,-0.5);
          \intright{blue}{-0.3,-0.1};
          \shiftline{-0.4,-0.5}{0.4,-0.5}{r-ld,t};
          \shiftline{-0.4,0.5}{0.4,0.5}{s,t};
        \end{tikzpicture}
        \ ,\qquad
        \begin{tikzpicture}[anchorbase]
          \draw[->] (0,.2) -- (0,0) arc (180:360:.3) -- (.6,.2);
          \shiftline{-0.1,-0.5}{0.7,-0.5}{r,t};
          \shiftline{-0.1,0.2}{0.7,0.2}{s,t};
        \end{tikzpicture}
        \mapsto
        \begin{tikzpicture}[anchorbase]
          \draw[->,blue] (0,.2) -- (0,0) arc (180:360:.3) -- (.6,.2);
          \shiftline{-0.1,-0.5}{0.7,-0.5}{r,t};
          \shiftline{-0.1,0.2}{0.7,0.2}{s+md,t};
        \end{tikzpicture}
        +
        \begin{tikzpicture}[anchorbase]
          \draw[->,red] (-0.3,0.5) -- (-0.3,0) arc (180:360:0.3) -- (0.3,.5);
          \intleft{blue}{-0.3,0.1};
          \shiftline{-0.4,-0.5}{0.4,-0.5}{r,t};
          \shiftline{-0.4,0.5}{0.4,0.5}{s-ld,t};
        \end{tikzpicture}
        \! ,\\
\label{fuzzygraded4}
        \begin{tikzpicture}[anchorbase]
          \draw[<-] (0,-.2) -- (0,0) arc (180:0:.3) -- (.6,-.2);
          \shiftline{-0.1,-0.2}{0.7,-0.2}{r,t};
          \shiftline{-0.1,0.5}{0.7,0.5}{s,t};
        \end{tikzpicture}
        \mapsto
        \begin{tikzpicture}[anchorbase]
          \draw[<-,blue] (-0.3,-0.5) -- (-0.3,0) arc (180:0:0.3) -- (0.3,-0.5);
          \intleft{red}{0.3,-0.1};
          \shiftline{-0.4,-0.5}{0.4,-0.5}{r+md,t};
          \shiftline{-0.4,0.5}{0.4,0.5}{s,t};
        \end{tikzpicture}
        +
        \begin{tikzpicture}[anchorbase]
          \draw[<-,red] (0,-.2) -- (0,0) arc (180:0:.3) -- (.6,-.2);
          \shiftline{-0.1,-0.2}{0.7,-0.2}{r-ld,t};
          \shiftline{-0.1,0.5}{0.7,0.5}{s,t};
        \end{tikzpicture}
        \ ,\qquad
        \begin{tikzpicture}[anchorbase]
          \draw[<-] (0,.2) -- (0,0) arc (180:360:.3) -- (.6,.2);
          \shiftline{-0.1,-0.5}{0.7,-0.5}{r,t};
          \shiftline{-0.1,0.2}{0.7,0.2}{s,t};
        \end{tikzpicture}
        \mapsto
        \begin{tikzpicture}[anchorbase]
          \draw[<-,blue] (-0.3,0.5) -- (-0.3,0) arc (180:360:0.3) -- (0.3,.5);
          \intright{red}{0.3,0.1};
          \shiftline{-0.4,-0.5}{0.4,-0.5}{r,t};
          \shiftline{-0.4,0.5}{0.4,0.5}{s+md,t};
        \end{tikzpicture}
        +
        \begin{tikzpicture}[anchorbase]
          \draw[<-,red] (0,.2) -- (0,0) arc (180:360:.3) -- (.6,.2);
          \shiftline{-0.1,-0.5}{0.7,-0.5}{r,t};
          \shiftline{-0.1,0.2}{0.7,0.2}{s-ld,t};
        \end{tikzpicture}
        \ .
    \end{gather}
\end{cor}

\begin{proof}
    There is a strict graded monoidal functor $\Delta_{\blue{l}|\red{m}} \colon \Heis{A}{k} \to \Add (( \Heis[blue]{A}{l} \barodot \Heis[red]{A}{m})_{q,\pi})$ defined on objects by $\uparrow \mapsto Q^0 \Pi^0 \upblue \oplus Q^{-ld} \Pi^0 \upred$, $\downarrow \mapsto Q^{md} \Pi^0 \downblue \oplus Q^0 \Pi^0 \downred$, and on generating morphisms by the same expressions as on the right hand sides of \cref{fuzzygraded1,fuzzygraded3,fuzzygraded2,fuzzygraded4} taking $t=0$.  To see this, note first that the degrees of cups/caps are consistent.  Then we must verify that the images of the relations from \cref{star} are satisfied in $\Add (( \Heis[blue]{A}{l} \barodot \Heis[red]{A}{m})_{q,\pi})$.  These relations hold in the strict monoidal supercategory $\Add (\Heis[blue]{A}{l} \barodot     \Heis[red]{A}{m} )$ thanks to \cref{bunny}; so the same relations also hold in $\Add ((\Heis[blue]{A}{l} \barodot \Heis[red]{A}{m})_{q,\pi} )$.  Finally, we get induced the functor in the statement of the corollary by applying the universal property of $(Q,\Pi)$-envelope.
\end{proof}

\begin{rem} \label{coass3}
    As in \cref{coass1,coass2}, the categorical comultiplication $\cDelt{l}{m}$, hence $\tilde\Delta_{\blue{l}|\red{m}}$, is coassociative in the appropriate sense.
\end{rem}

In the last result of the section, we upgrade \cref{coprodcat} to the
Frobenius Heisenberg category; see also \cref{camping} below for
further discussion of the significance of this result.  As before, we assume for
this that $A$ satisfies the hypothesis \cref{virginia} from the
introduction, and make a choice of idempotents as we did after \cref{coass1}.
Recall the homomorphisms $\imath_n$ and $\jmath_n$ from \cref{imath,jmath}.  For $\blambda \in \cP^N$ with $|\blambda|=n$, we have the associated object
\begin{equation}\label{dungeons}
    S_\blambda^+ := \left( \uparrow^{\otimes n}, \imath_n(e_\blambda) \right) \in \Kar\left(\Heis{A}{k}_{q,\pi}\right).
\end{equation}
For $n \in \N$, $1 \le i \le N$, define
\begin{equation}\label{dragons}
    H_{n,i}^+ := \left( \uparrow^{\otimes n}, \imath_n(e_{(n),i}) \right)
    \qquad
    E_{n,i}^+ := \left( \uparrow^{\otimes n}, \imath_n(e_{(1^n),i}) \right).
\end{equation}
Let $S_\blambda^- := (S_\blambda^+)^*$, $H_{n,i}^- := (H_{n,i}^+)^*$,
and $E_{n,i}^- := (E_{n,i}^+)^*$.
Note that there are degree zero isomorphisms
\begin{equation}\label{dd}
    S_\blambda^- \cong \left( \downarrow^{\otimes n},
    \jmath_n(e_{\blambda^T}) \right),
    \qquad
     E_{n,i}^- := \left( \downarrow^{\otimes n}, \jmath_n(e_{(n),i}) \right),
    \qquad
    H_{n,i}^- := \left( \downarrow^{\otimes n}, \jmath_n(e_{(1^n),i}) \right),
\end{equation}
hence
\begin{equation} \label{peaches}
    \Omega_k (S_\blambda^\pm) \cong S_{\blambda^T}^\mp,\qquad
    \Omega_k (E_{n,i}^\pm) \cong H_{n,i}^\mp,\qquad
    \Omega_k (H_{n,i}^\pm) \cong E_{n,i}^\mp.
\end{equation}
 The transpose appears here because the sign arising when $\jmath_n$ is
applied to a transposition is the opposite of the one that arises when
$*$ is applied to a crossing.

\begin{theo} \label{morebunny}
    Extending the functor $\tilde\Delta_{\blue{l}|\red{m}}$ from \cref{bunnycor} to the Karoubi envelopes in the canonical way, there are degree zero isomorphisms
    \begin{equation} \label{thedrop}
        \begin{aligned}
            \tilde\Delta_{\blue{l}|\red{m}} (H_{n,i}^+) &\cong \bigoplus_{r=0}^n Q^{-lr \dA} \blue{H_{n-r,i}^+} \otimes \red{H_{r,i}^+},
            &
            \tilde\Delta_{\blue{l}|\red{m}} (H_{n,i}^-) &\cong \bigoplus_{r=0}^n Q^{m(n-r) \dA } \blue{H_{n-r,i}^-} \otimes \red{H_{r,i}^-},
            \\
            \tilde\Delta_{\blue{l}|\red{m}} (E_{n,i}^+) &\cong \bigoplus_{r=0}^n Q^{-lr \dA} \blue{E_{n-r,i}^+} \otimes \red{E_{r,i}^+},
            &
            \tilde\Delta_{\blue{l}|\red{m}} (E_{n,i}^-) &\cong \bigoplus_{r=0}^n Q^{m (n-r)\dA} \blue{E_{n-r,i}^-} \otimes \red{E_{r,i}^-},
        \end{aligned}
    \end{equation}
    for $n \ge 0$ and $1 \le i \le N$,
\end{theo}

\begin{proof}
    For the $+$ sign, this follows from \cref{coprodcat} using also the grading shifts in the definition of $\tilde\Delta_{\blue{l}|\red{m}}$, noting that \cref{fuzzygraded3} is the same as \cref{Deltcross} (except for grading shifts).  Then for the $-$ sign, it follows by taking right duals (using the right cups and caps), again taking into account the grading shifts.
\end{proof}

\section{Generalized cyclotomic quotients\label{sec:GCQ}}

Throughout this section, we forget the grading to view $\Heis{A}{k}$
as a strict monoidal supercategory rather than as a graded monoidal
supercategory.  We will construct some module supercategories over it
known as generalized cyclotomic quotients.  In a similar fashion, one
can construct graded module supercategories over $\Heis{A}{k}$, but
these require additional homogeneity assumptions on the defining
parameters which are too restrictive for the subsequent applications
in the paper.

We fix a supercommutative superalgebra $R$.
Defining the $R$-superalgebra $\AR := R \otimes_\kk A$, we can extend
the trace map to an $R$-linear map $\tr_R = \id \otimes \tr \colon \AR
\to R$.  We can then base change to obtain $\HeisR{A}{k} = R
\otimes_\kk \Heis{A}{k}$, which is a strict $R$-linear monoidal
supercategory.
Extra care is needed when working with this since scalars in the base
ring $R$ are not necessarily even, resulting in potential additional
signs.
Morphism spaces in $\HeisR{A}{k}$ are naturally
left $R$-supermodules; since $R$ is supercommutative they may also be viewed as right
$R$-supermodules via the usual formula $v r = (-1)^{\bar v \bar r} rv$.
In diagrams for morphisms in $\HeisR{A}{k}$, we can label tokens by
elements of $\AR$.  The generating function formalism used
in the relations
\cref{wednesdayam,curlgen,wednesdaylunch,wednight,wednesdaypm}
extends in the obvious way to the base ring $R$.

\begin{lem}
  For a polynomial $p(u) \in \AR[u]$, we have
  \begin{align} \label{road1}
    \begin{tikzpicture}[anchorbase]
      \draw[->] (0,-0.4) to (0,0.4);
      \blacktoken[west]{0,0}{p(x)};
    \end{tikzpicture}
    &=
    \left[
      \begin{tikzpicture}[anchorbase]
        \draw[->] (0,-0.4) to (0,0.4);
        \blacktoken[east]{0,0.15}{(u-x)^{-1}};
        \blacktoken[west]{0,-0.15}{p(u)};
      \end{tikzpicture}
    \right]_{u^{-1}},
    &
    \begin{tikzpicture}[anchorbase]
      \draw[<-] (0,-0.4) to (0,0.4);
      \blacktoken[west]{0,0}{p(x)};
    \end{tikzpicture}
    &=
    \left[
      \begin{tikzpicture}[anchorbase]
        \draw[<-] (0,-0.4) to (0,0.4);
        \blacktoken[east]{0,0.15}{(u-x)^{-1}};
        \blacktoken[west]{0,-0.15}{p(u)};
      \end{tikzpicture}
    \right]_{u^{-1}},
    \\ \label{road2}
    \begin{tikzpicture}[anchorbase]
      \draw[->] (0,0.2) arc(90:-270:0.2);
      \blacktoken[west]{0.2,0}{p(x)};
    \end{tikzpicture}
    &= -
    \left[
      \begin{tikzpicture}[anchorbase]
        \bubgenright{0,0}{u};
        \blacktoken[west]{0.2,0}{p(u)};
      \end{tikzpicture}
    \right]_{u^{-1}},
    &
     \begin{tikzpicture}[anchorbase]
      \draw[<-] (0,0.2) arc(90:-270:0.2);
      \blacktoken[west]{0.2,0}{p(x)};
    \end{tikzpicture}
    &=
    \left[
      \begin{tikzpicture}[anchorbase]
        \bubgenleft{0,0}{u};
        \blacktoken[west]{0.2,0}{p(u)};
      \end{tikzpicture}
    \right]_{u^{-1}},
    \\ \label{road3}
    \begin{tikzpicture}[anchorbase]
      \draw[->] (0,-0.5) to[out=up,in=180] (0.3,0.2) to[out=0,in=up] (0.45,0) to[out=down,in=0] (0.3,-0.2) to[out=180,in=down] (0,0.5);
      \blacktoken[west]{0.45,0}{p(x)};
    \end{tikzpicture}
    &=
    \left[
      \begin{tikzpicture}[anchorbase]
        \draw[->] (0,-0.5) -- (0,0.5);
        \draw[->] (0.5,0.2) arc(90:-270:0.2);
        \node at (0.5,0) {\dotlabel{u}};
        \blacktoken[east]{0,0.25}{(u-x)^{-1}};
        \teleport{black}{black}{0.3,0}{0,0};
        \blacktoken[west]{0.7,0}{p(u)};
      \end{tikzpicture}
    \right]_{u^{-1}}
    &
    \begin{tikzpicture}[anchorbase]
      \draw[<-] (0,-0.5) to[out=up,in=180] (0.3,0.2) to[out=0,in=up] (0.45,0) to[out=down,in=0] (0.3,-0.2) to[out=180,in=down] (0,0.5);
      \blacktoken[west]{0.45,0}{p(x)};
    \end{tikzpicture}
    &=
    \left[
      \begin{tikzpicture}[anchorbase]
        \draw[<-] (0,-0.5) -- (0,0.5);
        \draw[<-] (0.5,0.2) arc(90:-270:0.2);
        \node at (0.5,0) {\dotlabel{u}};
        \blacktoken[east]{0,0.25}{(u-x)^{-1}};
        \teleport{black}{black}{0.3,0}{0,0};
        \blacktoken[west]{0.7,0}{p(u)};
      \end{tikzpicture}
    \right]_{u^{-1}}.
  \end{align}
\end{lem}

\begin{proof}
    By linearity, it suffices to consider the case $p(x) = ax^n$, $a \in A_R$, $n \ge 0$.  In that case, \cref{road1,road2} follow immediately from computing the $u^{-1}$ coefficient on the right-hand side.  Then \cref{road3} follows by using \cref{road1,curlgen}.
\end{proof}

Consider the center $Z(\AR) = R \otimes_\kk Z(A)$ of $\AR$.  Its even
part $Z(\AR)_\even$ is a commutative algebra.
Suppose that
\begin{equation}\label{fg}
f(u) = f_0 u^l + f_1 u^{l-1} + \dotsb + f_l
    \quad \text{and} \quad
    g(u) = g_0 u^m + g_1 u^{m-1} + \dotsb + g_m
\end{equation}
are monic polynomials in $Z(\AR)_\even[u]$ of polynomial degrees $l,m
\ge 0$; in particular, $f_0=g_0=1$.
Define $k := m-l$ and
\begin{equation} \label{harbinger}
    \OO(u) := g(u)/f(u) \in u^k + u^{k-1} Z(\AR)_\even \llbracket u^{-1} \rrbracket.
\end{equation}
Then define $\OO^{(r)}, \tilde{\OO}^{(r)} \in Z(\AR)_\even$ so that
\begin{equation}
    \OO(u) = \sum_{r \in \Z} \OO^{(r)} u^{-r-1},\qquad
    \OO(u)^{-1} = - \sum_{r \in \Z} \tilde{\OO}^{(r)} u^{-r-1}.
\end{equation}

\begin{lem} \label{Chicago}
    The left tensor ideal $\cIR(f|g)$ of $\HeisR{A}{k}$ generated by
    \begin{equation} \label{Chicago1}
        \begin{tikzpicture}[anchorbase]
            \draw[->] (0,-0.25) to (0,0.25);
            \blacktoken[east]{0,0}{f(x)};
        \end{tikzpicture}
        \quad \text{and} \quad
        \ccbubble{a}{r} - \tr_R(\OO^{(r)} a) 1_\one,\quad -k \le r < l,\ a \in \AR,
    \end{equation}
    is equal to the left tensor ideal of $\HeisR{A}{k}$ generated by
    \begin{equation} \label{Chicago2}
        \begin{tikzpicture}[anchorbase]
            \draw[<-] (0,-0.25) to (0,0.25);
            \blacktoken[east]{0,0}{g(x)};
        \end{tikzpicture}
        \quad \text{and} \quad
        \cbubble{a}{r} - \tr_R(\tilde{\OO}^{(r)} a) 1_\one,\quad k \le r < m,\ a \in \AR.
    \end{equation}
    Furthermore, this ideal contains
    \begin{equation} \label{Chicago3}
        \ccbubble{a}{r} - \tr_R(\OO^{(r)} a) 1_\one
        \quad \text{and} \quad
        \cbubble{a}{r} - \tr_R(\tilde{\OO}^{(r)} a) 1_\one
        \quad \text{for all } r \in \Z,\ a \in \AR.
    \end{equation}
\end{lem}

\begin{proof}
    We first show by induction on $r$ that $\cIR(f|g)$ contains $\ccbubble{a}{r} - \tr_R(\OO^{(r)} a) 1_\one$ for all $r \in \Z$, $a \in \AR$.  The result holds for $r < l$ by \cref{chew} and the definition of $\cIR(f|g)$.  Now suppose $r \ge l$.  By \cref{harbinger}, we have $\OO(u) f(u) = g(u)$.  This is a polynomial in $u$, and so its $u^{l-r-1}$-coefficient is zero.  Thus
    \begin{equation} \label{jefe}
        \sum_{s=0}^l \OO^{(r-s)} f_s = 0.
    \end{equation}
    Then, by our induction hypothesis, we have
    \[
        \ccbubble{a}{r} - \tr_R(\OO^{(r)} a) 1_\one
        \overset{\cref{jefe}}{=} \ccbubble{a}{r} + \sum_{s=1}^l \tr_R \left( \OO^{(r-s)} f_s a \right) 1_\one
        = \sum_{s=0}^l \ccbubble{f_s a}{r-s}
        =
        \begin{tikzpicture}[baseline={(0,-0.15)}]
          \draw[->] (0,0.2) arc(90:450:0.2);
          \multblackdot[west]{0.2,0}{x^{r-l} f(x)};
          \blacktoken[north]{0,-0.2}{a};
        \end{tikzpicture}
        \in \cIR(f|g).
    \]

    For $a \in \AR$, let $\tr_R(\OO(u)a) = \sum_{r \in \Z} \tr_R(\OO^{(r)}a) u^{-r-1} \in u^k + u^{k-1} R \llbracket u^{-1} \rrbracket$.  We have shown that
    \[
        \begin{tikzpicture}[anchorbase]
            \bubgenleft{0,0}{u};
            \blacktoken[east]{-0.2,0}{a};
        \end{tikzpicture}
        = \tr_R \left( \OO(u) a \right) 1_\one
        \mod{\cIR(f|g)}.
    \]
    For $a,b \in \AR$, we have
    \[
        \sum_{c \in \B_A} \tr_R \left( \OO(u) ac \right) \tr_R \left( \OO(u)^{-1} c^\vee b \right)
        \overset{\cref{plane}}{=} \tr_R \left( \OO(u) \OO(u)^{-1} ab \right)
        = \tr_R(ab)
        \overset{\cref{eyes}}{=} 
        \begin{tikzpicture}[anchorbase]
            \bubgenleft{-0.4,0}{u};
            \bubgenright{0.4,0}{u};
            \teleport{black}{black}{-0.2,0}{0.2,0};
            \blacktoken[east]{-0.6,0}{a};
            \blacktoken[west]{0.6,0}{b};
        \end{tikzpicture}
        \ .
    \]
    This implies that, for $a \in \AR$,
    \[
        \begin{tikzpicture}[anchorbase]
            \bubgenright{0,0}{u};
            \blacktoken[east]{-0.2,0}{a};
        \end{tikzpicture}
        = \tr_R \left( \OO(u)^{-1} a \right) 1_\one
        \mod{\cIR(f|g)}
    \]
    (see the proof of \cref{controversy} below for a similar argument), proving \cref{Chicago3}.

    Now note that, in the quotient of $\HeisR{A}{k}$ by $\cIR(f|g)$, we have
    \[
        \begin{tikzpicture}[anchorbase]
            \draw[<-] (0,-0.4) to (0,0.4);
            \blacktoken[east]{0,0}{g(x)};
        \end{tikzpicture}
        \overset{\cref{road1}}{=}
        \left[
            \begin{tikzpicture}[anchorbase]
                \draw[<-] (0,-0.4) to (0,0.4);
                \blacktoken[east]{0,0.15}{(u-x)^{-1}};
                \blacktoken[west]{0,-0.15}{g(u)};
            \end{tikzpicture}
        \right]_{u^{-1}}
        \overset{\cref{harbinger}}{=}
        \left[
            \begin{tikzpicture}[anchorbase]
                \draw[<-] (0,-0.4) to (0,0.4);
                \blacktoken[east]{0,0.15}{(u-x)^{-1}};
                \blacktoken[west]{0,-0.15}{\OO(u) f(u)};
            \end{tikzpicture}
        \right]_{u^{-1}}
        \overset{\cref{Chicago3}}{\underset{\cref{plane}}{=}}
        \left[
            \begin{tikzpicture}[anchorbase]
                \draw[<-] (0,-0.4) to (0,0.4);
                \bubgenleft{0.6,0}{u};
                \blacktoken[west]{0.8,0}{f(u)};
                \teleport{black}{black}{0,0}{0.4,0};
                \blacktoken[east]{0,0.2}{(u-x)^{-1}};
            \end{tikzpicture}
        \right]_{u^{-1}}
        \overset{\cref{road3}}{=}
        \begin{tikzpicture}[anchorbase]
            \draw[<-] (0,-0.5) to[out=up,in=180] (0.3,0.2) to[out=0,in=up] (0.45,0) to[out=down,in=0] (0.3,-0.2) to[out=180,in=down] (0,0.5);
            \blacktoken[west]{0.45,0}{f(x)};
        \end{tikzpicture}
        = 0.
  \]
  So we have shown that the left tensor ideal generated by \cref{Chicago1} contains the elements \cref{Chicago2}.  A similar argument shows that the left tensor ideal generated by \cref{Chicago2} contains the elements \cref{Chicago1}, completing the proof.
\end{proof}

\begin{defin}
The \emph{generalized cyclotomic quotient} of $\HeisR{A}{k}$
corresponding to the supercommutative algebra $R$ and the polynomials
$f(u)$, $g(u)$ as fixed above
is the $R$-linear quotient category
\[
  \GCQ{f}{g} := \HeisR{A}{k} / \cIR(f|g).
\]
This is not a monoidal supercategory, rather, it is
a (left) module supercategory over $\HeisR{A}{k}$.
Unless the polynomials $f(u)$ and $g(u)$ are themselves homogeneous
when $u$ is given degree $2d$, the ideal $\cIR(f|g)$ is not
homogeneous, and this quotient category is no longer graded.
\end{defin}

If $\cV$ and $\cW$ are two $R$-linear supercategories, we define $\cV \boxtimes_R \cW$ to be the $R$-linear supercategory whose objects are pairs $(X,Y)$ of objects $X \in \cV$ and $Y \in \cW$, and with morphisms
\[
  \Hom_{\cV \boxtimes_R \cW} \big( (X_1, Y_1), (X_2, Y_2) \big)
  := \Hom_\cV (X_1, X_2) \otimes_R \Hom_\cW (Y_1, Y_2).
\]
Composition in $\cV \boxtimes_R \cW$ is defined by \cref{interchange}.
If $\cV$ and $\cW$ are module supercategories over $\HeisR[blue]{A}{l}$ and $\HeisR[red]{A}{m}$, respectively, then $\cV \boxtimes_R \cW$ is naturally a module supercategory over the $R$-linear symmetric product $\HeisR[blue]{A}{l} \odot_R \HeisR[red]{A}{m}$.  If, in addition, the action of the morphism
\begin{equation}\label{tea}
    \begin{tikzpicture}[anchorbase]
        \draw[->,blue] (-0.2,-0.4) to (-0.2,0.4);
        \draw[->,red] (0.2,-0.4) to (0.2,0.4);
        \dashdumb{-0.2,0}{0.2,0};
    \end{tikzpicture}
    \ =\
    \begin{tikzpicture}[anchorbase]
        \draw[->,blue] (-0.2,-0.4) to (-0.2,0.4);
        \draw[->,red] (0.2,-0.4) to (0.2,0.4);
        \reddot{0.2,0};
    \end{tikzpicture}
    \ -\
    \begin{tikzpicture}[anchorbase]
        \draw[->,blue] (-0.2,-0.4) to (-0.2,0.4);
        \draw[->,red] (0.2,-0.4) to (0.2,0.4);
        \bluedot{-0.2,0};
    \end{tikzpicture}
\end{equation}
is invertible on all objects of $\cV \boxtimes_R \cW$, then the
categorical action induces an action of the localization
$\HeisR[blue]{A}{l} \barodotnonumber_R \HeisR[red]{A}{m}$.  Then, using
the $R$-linearization of the categorical comultiplication $\cDelt{l}{m}$ of \cref{bunny}, we
have that $\cV \boxtimes_R \cW$ is a module supercategory over
$\HeisR{A}{l+m}$.

Since dual bases of $A$ over $\kk$ are also dual bases of $\AR$ over
$R$, we have $\AWAR{n} = R \otimes_\kk \AWA{n}$.  The \emph{cyclotomic
  wreath product algebra} $\CWAR{n}{f}$ associated to $f(u)$ as in
\cref{fg} is the
quotient $\AWAR{n}/(f(x_1))$.  We adopt the convention that
$\CWAR{0}{f} = \AWAR{0} =R$.  By \cite[Th.~6.11]{Sav20},
\begin{equation} \label{AW-basis}
    \left\{ x_1^{\alpha_1} \dotsm x_n^{\alpha_n} \bb \sigma :
      \alpha_1,\dotsc,\alpha_n < l,\ \bb \in \B_A^{\otimes n},\ \sigma
      \in \fS_n \right\}
\end{equation}
is an $R$-basis of $\CWAR{n}{f}$.
It follows that the natural homomorphism $\CWAR{n}{f} \to \CWAR{n+1}{f}$ is injective.

For $f(u),g(u)$ as in \cref{fg}, Define
\[
    \cV(f) := \bigoplus_{n \ge 0} \psmod \CWAR{n}{f},\quad
    \cV(g)^\vee := \bigoplus_{n \ge 0} \psmod \CWAopR{n}{g}.
\]
As explained in \cite[\S1]{Sav18} (except that we use right modules
here instead of left modules), we have a strict $R$-linear monoidal superfunctor
\begin{equation} \label{upaction}
    \Psi_f \colon \HeisR{A}{-l} \to \SEnd_R \left( \cV(f) \right)
\end{equation}
sending $\uparrow$ (resp.\ $\downarrow$) to the additive endofunctor that is the induction superfunctor $\ind_n^{n+1} = - \otimes_{\CWAR{n}{f}} \CWAR{n+1}{f}$ (resp.\ the restriction superfunctor $\res^n_{n-1}$) on $\psmod \CWAR{n}{f}$.  On generating morphisms, $\Psi_f$ sends
\begin{itemize}
  \item $
    \begin{tikzpicture}[anchorbase]
      \draw[->] (0,-0.25) to (0,0.25);
      \blackdot{0,0};
    \end{tikzpicture}
    $
    to the $R$-linear supernatural transformation defined on a projective $\CWAR{n}{f}$-supermodule $M$ by the map
    \[
      M \otimes_{\CWAR{n}{f}} \CWAR{n+1}{f}
      \to M \otimes_{\CWAR{n}{f}} \CWAR{n+1}{f},\quad
      v \otimes w \mapsto v \otimes x_{n+1} w;
    \]

  \item $
    \begin{tikzpicture}[anchorbase]
      \draw[->] (0,-0.25) to (0,0.25);
      \blacktoken[east]{0,0}{a};
    \end{tikzpicture}
    $
    to the $R$-linear supernatural transformation defined on a projective $\CWAR{n}{f}$-supermodule $M$ by the map
    \[
      M \otimes_{\CWAR{n}{f}} \CWAR{n+1}{f}
      \to M \otimes_{\CWAR{n}{f}} \CWAR{n+1}{f},\quad
      v \otimes w \mapsto (-1)^{\bar{a} \bar{v}} v \otimes (a \otimes 1^{\otimes n}) w;
    \]

  \item $
    \begin{tikzpicture}[anchorbase]
      \draw[->] (-0.2,-0.25) to (0.2,0.25);
      \draw[->] (0.2,-0.25) to (-0.2,0.25);
    \end{tikzpicture}
    $
    to the $R$-linear supernatural transformation defined on a projective $\CWAR{n}{f}$-supermodule $M$ by the map
    \[
      M \otimes_{\CWAR{n}{f}} \CWAR{n+2}{f}
      \to M \otimes_{\CWAR{n}{f}} \CWAR{n+2}{f},\quad
      v \otimes w \mapsto v \otimes s_{n+1} w;
    \]

  \item $
    \begin{tikzpicture}[anchorbase]
      \draw[->] (-0.2,0.2) to (-0.2,0) arc(180:360:0.2) to (0.2,0.2);
    \end{tikzpicture}
    $
    and
    $
    \begin{tikzpicture}[anchorbase]
      \draw[->] (-0.2,-0.2) to (-0.2,0) arc(180:0:0.2) to (0.2,-0.2);
    \end{tikzpicture}
    $
    to the $R$-linear supernatural transformations defined by the unit and counit of the adjunction making $(\ind_n^{n+1}, \res_n^{n+1})$ into an adjoint pair.
\end{itemize}
The superfunctor \cref{upaction} makes $\cV(f)$ into a module
supercategory over $\HeisR{A}{-l}$.

Similarly, we make $\cV(g)^\vee$ into a module supercategory over $\HeisR{A}{m}$ via the strict $R$-linear monoidal superfunctor
\begin{equation} \label{downaction}
    \Psi_g^\vee \colon \HeisR{A}{m} \to \SEnd_R \left( \cV(g)^\vee \right)
\end{equation}
sending $\downarrow$ (resp.\ $\uparrow$) to the endosuperfunctor defined on $\psmod \CWAopR{n}{g}$ by the induction functor $\ind_n^{n+1}$ (resp.\ the restriction functor $\res^n_{n-1}$).  On generating morphisms, $\Psi_f$ sends
\begin{itemize}
    \item $
        \begin{tikzpicture}[anchorbase]
            \draw[<-] (0,-0.25) to (0,0.25);
            \blackdot{0,0};
        \end{tikzpicture}
    $
    to the $R$-linear supernatural transformation defined on a projective $\CWAopR{n}{g}$-supermodule $M$ by the map
    \[
        M \otimes_{\CWAopR{n}{g}} \CWAopR{n+1}{g}
        \to M \otimes_{\CWAopR{n}{g}} \CWAopR{n+1}{g},\quad
        v \otimes w \mapsto v \otimes x_{n+1} w ;
    \]

    \item $
        \begin{tikzpicture}[anchorbase]
            \draw[<-] (0,-0.25) to (0,0.25);
            \blacktoken[east]{0,0}{a};
        \end{tikzpicture}
    $
    to the $R$-linear supernatural transformation defined on a projective $\CWAopR{n}{g}$-supermodule $M$ by the map
    \[
        M \otimes_{\CWAopR{n}{g}} \CWAopR{n+1}{g}
        \to M \otimes_{\CWAopR{n}{g}} \CWAopR{n+1}{g},\quad
        v \otimes w \mapsto (-1)^{\bar{a} \bar{v}} v \otimes (a \otimes 1^{\otimes n}) w;
    \]

    \item $
        \begin{tikzpicture}[anchorbase]
            \draw[<-] (-0.2,-0.25) to (0.2,0.25);
            \draw[<-] (0.2,-0.25) to (-0.2,0.25);
        \end{tikzpicture}
    $
    to the $R$-linear supernatural transformation defined on a projective $\CWAopR{n}{g}$-supermodule $M$ by the map
    \[
        M \otimes_{\CWAopR{n}{g}} \CWAopR{n+2}{g}
        \to M \otimes_{\CWAopR{n}{g}} \CWAopR{n+2}{g},\quad
        v \otimes w \mapsto - v \otimes s_{n+1} w;
    \]

    \item $
        \begin{tikzpicture}[anchorbase]
            \draw[<-] (-0.2,0.2) to (-0.2,0) arc(180:360:0.2) to (0.2,0.2);
        \end{tikzpicture}
    $
    and
    $
    \begin{tikzpicture}[anchorbase]
        \draw[<-] (-0.2,-0.2) to (-0.2,0) arc(180:0:0.2) to (0.2,-0.2);
    \end{tikzpicture}
    $
    to the $R$-linear supernatural transformations defined by the unit and counit of the adjunction making $(\ind_n^{n+1}, \res_n^{n+1})$ into an adjoint pair.
\end{itemize}

On combining \cref{upaction,downaction}, the $R$-linear supercategory
\begin{equation}\label{Vfg}
 \cV(f|g) := \cV(f) \boxtimes_R \cV(g)^\vee
\approx  \bigoplus_{m,n\geq 0} \psmod
\CWAR{m}{f}\otimes_R \CWAopR{n}{g}
\end{equation}
becomes a module supercategory over
$\HeisR[blue]{A}{l} \odot_R \HeisR[red]{A}{m}$.
We would like the morphism \cref{tea} to act invertibly too so that we
can make $\cV(f|g)$ into a module supercategory over $\HeisR{A}{k}$, but for
this we need to impose a genericity condition on the defining
parameters.
To formulate this, {\em we assume for the remainder of the section}
that the
base ring $R$ is a finite-dimensional supercommutative superalgebra
over an algebraically closed field $\KK\supseteq \kk$;
by ``eigenvalue'' we mean eigenvalue in this field $\KK$.
Noting that $\AR \otimes_R \AR$ is
finite dimensional over $\KK$,
let $\Gamma_R$ be the abelian subgroup of $(\KK,+)$ generated by the eigenvalues of
the endomorphism of
$\AR \otimes_R \AR$ defined by left multiplication
by $\tau_1 = \sum_{b \in \B_A} b \otimes b^\vee$.
For example, if $\dA \neq 0$, then $\tau_1$ is necessarily nilpotent
for degree reasons, hence $\Gamma_R = \{0_\KK\}$; on the other hand, if $A
= \kk$ then $\tau_1 = 1 \otimes 1$ so $\Gamma_R = \Z \cdot 1_\KK$.
By Schur's lemma, $Z(\AR)_\even$ acts on any irreducible $\AR$-supermodule
$L$ via a central
character $\chi_L \colon Z(\AR)_\even \to \KK$.  Hence, for $f(u), g(u)$ as in (\ref{fg}), we have $\chi_L(f(u)), \chi_L(g(u)) \in
\KK[u]$. Let $\Xi_f$ and $\Xi_g$ denote the subsets of $\KK$ consisting of the
roots of $\chi_L(f(u))$ or $\chi_L(g(u))$, respectively, for all
irreducible $\AR$-supermodules $L$.

\begin{defin}\label{genericity}
We say that $f(u), g(u)$ are a {\em generic pair} if
the sets $\Xi_f+\Gamma_R$ and $\Xi_g+\Gamma_R$ are disjoint.
\end{defin}

\begin{lem}
    Assume that $V$ is an $\AWAR{2}$-supermodule.
Let $\lambda_2$ be an eigenvalue of $x_2$ on $V$.
Then $\lambda_2 \in \lambda_1+\Gamma_R$ for some eigenvalue $\lambda_1$ of $x_1$ on
$V$.
\end{lem}

\begin{proof}
Since $x_1$, $x_2$, and $\tau_1$ all commute, we can choose a vector
$v \in V$ in the $\lambda_2$-eigenspace of $x_2$ that is also an
eigenvector of $x_1$ with some eigenvalue $\lambda_1$ and an
eigenvector of $\tau_1$ with some eigenvalue $\gamma \in \Gamma_R$.
First suppose that $v$ is an eigenvector for $s_1$.  Since $s_1^2 = 1$, this means that $s_1 v = \pm v$.  Then we have
  \[
    \lambda_2 v
    = \pm s_1 x_2 v
    = (\pm x_1 s_1 \pm \tau_1) v
    = (\lambda_1 \pm \gamma) v,
  \]
hence $\lambda_2 \in \lambda_1+\Gamma_R$.
  On the other hand, suppose $v$ is not an eigenvector of $s_1$.  Then, by the relation $x_1 s_1 = s_1 x_2 - \tau_1$, the matrix describing the action of $x_1$ on the subspace with basis $\{v,s_1v\}$ is
  \[
    \begin{bmatrix}
      \lambda_1 & -\gamma \\
      0 & \lambda_2
    \end{bmatrix}.
  \]
  Hence $\lambda_2$ is also an eigenvalue of $x_1$.
\end{proof}

\begin{cor} \label{hot}
    The eigenvalues of $x_1,\dotsc,x_n$ on any $\CWAR{n}{f}$-supermodule
    lie in $\Xi_f + \Gamma_R$.
\end{cor}

\begin{cor} \label{bike}
    Suppose that $f(u), g(u) \in Z(\AR)_\even[u]$ are a generic pair
    in the sense of \cref{genericity}.  In the categorical action of
    $\HeisR[blue]{A}{-l} \odot_R \HeisR[red]{A}{m}$ on $\cV(f|g)$, the
    morphism \cref{tea} acts invertibly.
\end{cor}

\begin{proof}
\cref{hot} and the genericity assumption imply that
the set of
eigenvalues of $x_1,\dots,x_n$ on any
finite-dimensional $\CWAR{n}{f}$-supermodule is disjoint from the
set of eigenvalues of $x_1,\dots,x_m$ on any finite-dimensional $\CWAopR{m}{g}$-supermodule.
Consequently, the commuting endomorphisms
defined by evaluating
$$    \begin{tikzpicture}[anchorbase]
        \draw[->,blue] (-0.2,-0.4) to (-0.2,0.4);
        \draw[->,red] (0.2,-0.4) to (0.2,0.4);
        \bluedot{-0.2,0};
    \end{tikzpicture}
\qquad\text{ and }\qquad
\begin{tikzpicture}[anchorbase]
        \draw[->,blue] (-0.2,-0.4) to (-0.2,0.4);
        \draw[->,red] (0.2,-0.4) to (0.2,0.4);
        \reddot{0.2,0};
    \end{tikzpicture}
$$
on an object of
$\mathcal{V}(f|g)$ have disjoint spectra.
Hence, all eigenvalues of the endomorphism
defined by \cref{tea}
lie in $\kk^\times$.
Hence, this endomorphism is invertible.
\end{proof}

Assuming the genericity condition holds,
\cref{bike} is exactly what is needed to see that we can make
$\cV(f|g)$ into a
$\HeisR{A}{k}$-module supercategory as explained after \cref{tea}.
Thus, recalling that $k=m-l$, we have constructed a strict $R$-linear monoidal superfunctor
\begin{equation} \label{tiki}
    \Psi_{f|g} \colon \HeisR{A}{k} \to \SEnd_R \big( \cV(f|g) \big).
\end{equation}
    Recall the generating functions
    $
        \begin{tikzpicture}[baseline={(0,-0.1)}]
            \bubgenleft{0,0}{u};
            \blacktoken[east]{-0.2,0}{a};
        \end{tikzpicture}
    $
    and
    $
        \begin{tikzpicture}[baseline={(0,-0.1)}]
            \bubgenright{0,0}{u};
            \blacktoken[west]{0.2,0}{a};
        \end{tikzpicture}
    $
    defined in \cref{hedge1,hedge2}.

\begin{lem} \label{bbq}
We have that
    \begin{align*}
        \Psi_{f|g}
        \left(
            \begin{tikzpicture}[anchorbase]
                \bubgenleft{0,0}{u};
                \blacktoken[east]{-0.2,0}{a};
            \end{tikzpicture}\
        \right)_{(\CWAR{0}{f}, \CWAopR{0}{g})}
        &=
        \tr_R \left( g(u) f(u)^{-1} a \right) \in u^{k} R \llbracket
          u^{-1} \rrbracket,
\\
        \Psi_{f|g}
        \left(\
            \begin{tikzpicture}[anchorbase]
                \bubgenright{0,0}{u};
                \blacktoken[west]{0.2,0}{a};
            \end{tikzpicture}
        \right)_{(\CWAR{0}{f}, \CWAopR{0}{g})}
        &=
        \tr_R \left( f(u) g(u)^{-1} a \right) \in u^{-k} R \llbracket u^{-1} \rrbracket.
    \end{align*}
\end{lem}

\begin{proof}
    Applying \cref{Chicago} with $g(u)=1$, we get that
    \[
        \Psi_f
        \left(
            \begin{tikzpicture}[anchorbase]
                \bubgenleft[blue]{0,0}{u};
                \blacktoken[east]{-0.2,0}{a};
            \end{tikzpicture}\
        \right)_{\CWAR{0}{f}}
        = \tr_R \left( f(u)^{-1} a \right),
        \qquad
        \Psi_f
        \left(\
            \begin{tikzpicture}[anchorbase]
                \bubgenright[blue]{0,0}{u};
                \blacktoken[west]{0.2,0}{a};
            \end{tikzpicture}
        \right)_{\CWAR{0}{f}}
        = \tr_R \left( f(u) a \right).
    \]
    Similarly, applying it with $f(u)=1$, we get that
    \[
        \Psi_g^\vee
        \left(
            \begin{tikzpicture}[anchorbase]
                \bubgenleft[red]{0,0}{u};
                \blacktoken[east]{-0.2,0}{a};
            \end{tikzpicture}\
        \right)_{\CWAopR{0}{g}}
        = \tr_R \left( g(u) a \right),
        \qquad
        \Psi_g^\vee
        \left(\
            \begin{tikzpicture}[anchorbase]
                \bubgenright[red]{0,0}{u};
                \blacktoken[west]{0.2,0}{a};
            \end{tikzpicture}
        \right)_{\CWAopR{0}{g}}
        = \tr_R \left( g(u)^{-1} a \right).
    \]
    Thus, by \cref{circus}, we have
    \begin{align*}
        \Psi_{f|g}
        \left(
            \begin{tikzpicture}[anchorbase]
                \bubgenleft{0,0}{u};
                \blacktoken[west]{0.2,0}{a};
            \end{tikzpicture}
        \right)_{(\CWAR{0}{f}, \CWAopR{0}{g})}
        &= \sum_{b \in \B_A}
        \tr_R \left( f(u)^{-1} ab \right) \tr_R \left( g(u) b^\vee \right)
        \overset{\cref{plane}}{=}
        \tr_R \left( g(u) f(u)^{-1} a \right),
        \\
        \Psi_{f|g}
        \left(
            \begin{tikzpicture}[anchorbase]
                \bubgenright{0,0}{u};
                \blacktoken[west]{0.2,0}{a};
            \end{tikzpicture}
        \right)_{(\CWAR{0}{f}, \CWAopR{0}{g})}
        &= \sum_{b \in \B_A}
        \tr_R \left( f(u) b^\vee a \right) \tr_R \left( g(u)^{-1} b \right)
        \overset{\cref{plane}}{=}
        \tr_R \left( f(u) g(u)^{-1} a \right). \qedhere
    \end{align*}
\end{proof}

\begin{theo} \label{GCQ}
    Suppose that $f(u), g(u) \in Z(\AR)_\even[u]$ are a generic pair
    as in \cref{genericity}.
Let $$\Ev \colon \SEnd_R ( \cV(f|g) ) \to \cV(f|g)$$
be the $R$-linear functor defined by evaluation on $(\CWAR{0}{f}, \CWAopR{0}{g}) \in \cV(f|g)$.  Then $\Ev \circ \Psi_{f|g}$ factors through the generalized cyclotomic quotient $\GCQ{f}{g}$ to induce an equivalence of $\HeisR{A}{k}$-module supercategories
    \[
        \psi_{f|g} \colon \Kar(\GCQ{f}{g}_{\pi}) \to \cV(f|g).
    \]
\end{theo}

\begin{proof}
    It is clear that
    $
        \begin{tikzpicture}[anchorbase]
            \draw[->] (0,-0.25) to (0,0.25);
            \blacktoken[east]{0,0}{f(x)};
        \end{tikzpicture}
    $
    acts as zero on $(\CWAR{0}{f}, \CWAopR{0}{g})$.  Together with \cref{bbq}, this implies that $\Ev \circ \Psi_{f|g}$ factors through the quotient $\GCQ{f}{g}$ to induce an $R$-linear superfunctor $\GCQ{f}{g} \to \cV(f|g)$.  Since $\cV(f|g)$ is an additive Karoubian $(Q,\Pi)$-category, this extends to induce the superfunctor $\psi_{f|g}$ from the statement of the theorem.  To show that $\psi_{f|g}$ is an equivalence, it remains to prove that it is full, faithful and dense.  This argument is almost identical to the one in the proof of \cite[Th.~9.5]{BSW-qheis}.
\end{proof}

\begin{rem}
    When $g(u)=1$, the genericity assumption is vacuous.  Hence \cref{GCQ} yields an equivalence of categories $\psi_{f|1} \colon \Kar(\GCQ{f}{1}_{\pi}) \to \cV(f)$.  So the generalized cyclotomic quotient $\GCQ{f}{1}$ is Morita equivalent to the usual cyclotomic quotient, that is, the locally unital superalgebra $\bigoplus_{n \ge 0} \CWAR{n}{f}$.
\end{rem}

\section{Basis theorem\label{sec:basis}}

Recall that $C(A)$ denotes the cocenter of $A$ with homogeneous basis $\B_{C(A)}$.  Let $\Sym(A)$ denote the graded symmetric superalgebra generated by the graded vector superspace $C(A)[x]$, where $x$ here is an even indeterminate of degree $2d$.  For $n \in \Z$ and $a \in A$, let $e_n(a) \in \Sym(A)$ denote
\begin{equation}
    e_n(a) :=
    \begin{cases}
        0 & \text{if $n < 0$}, \\
        \tr(a) & \text{if $n=0$}, \\
        \cocenter{a} x^{n-1} & \text{if $n > 0$}.
    \end{cases}
\end{equation}
This defines a parity-preserving homogeneous linear map $e_n \colon A
\rightarrow \Sym(A)$ of degree $2(n-1)d$.
For example, in the special case that $A = \kk$, the algebra $\Sym(A)$
is a polynomial algebra generated freely by the elements
$\{e_n(1) : n > 0\}$. In this case, $\Sym(A)$ may be identified with the algebra of
symmetric functions so that $e_n(1)$ corresponds to the $n$th
elementary symmetric function; then the elements $h_n(1)$ defined by
the following lemma correspond to the complete symmetric functions.

\begin{lem}\label{controversy}
    For each $n \in \Z$, there is a unique homogeneous parity-preserving linear map $h_n \colon A \rightarrow \Sym(A)$ of degree $2(n-1)d$ such that
    \begin{equation} \label{banana}
        \sum_{c \in \B_A} e(ac;-u) h(c^\vee b; u) = \tr(ab),
        \qquad \text{for all }a,b \in A,
    \end{equation}
    where we are using the generating functions
    \begin{equation} \label{moon}
        e(a;u) := \sum_{n \geq 0} e_n(a) u^{-n}, \quad
        h(a;u) := \sum_{n \geq 0} h_n(a) u^{-n}
        \in \Sym(A)\llbracket u^{-1}\rrbracket.
    \end{equation}
\end{lem}

\begin{proof}
    In light of \cref{plane}, it suffices to consider the case where $b=1$.  Then the generating function identity \cref{banana} is equivalent to the identities
    \[
        \sum_{r=0}^n \sum_{c \in \B_A} (-1)^r e_r (ac) h_{n-r}(c^\vee b) = \delta_{n,0} \tr(ab),\quad a,b \in C(A),\ n \in \N.
    \]
    Using \cref{plane}, this is equivalent to $h_0(a)=\tr(a)$ and, for $n \ge 1$,
    \[
        h_n(a) = - \sum_{r=1}^n \sum_{c \in \B_A} e_r(ac) h_{n-r}(c^\vee).
    \]
    This uniquely determines the $h_n(a)$ recursively.
\end{proof}

It follows from \cref{controversy,eyes} that we have a well-defined homomorphism of graded algebras
\begin{equation} \label{beta}
      \beta \colon \Sym(A) \to \End_{\Heis{A}{k}}(\one),\quad
      e_n(a) \mapsto (-1)^{n-1} \ccbubble{a}{n-k-1},\quad
      h_n(a) \mapsto \cbubble{a}{n+k-1}.
\end{equation}
Let $X = X_n \otimes \dotsb \otimes X_1$ and $Y = Y_m \otimes \dotsb \otimes Y_1$ be objects of $\Heis{A}{k}$ for $X_i, Y_j \in \{\uparrow, \downarrow\}$.  An \emph{$(X,Y)$-matching} is a bijection between the sets
\[
    \{i : X_i = \uparrow\} \sqcup \{j : Y_j = \downarrow\}
    \quad \text{and} \quad
    \{i : X_i = \downarrow\} \sqcup \{j : Y_j = \uparrow\}.
\]
A \emph{reduced lift} of an $(X,Y)$-matching is a string diagram representing a morphism $X \to Y$ such that
\begin{itemize}
    \item the endpoints of each string are points which correspond under the given matching;
    \item there are no floating bubbles and no dots or tokens on any string;
    \item there are no self-intersections of strings and no two strings cross each other more than once.
\end{itemize}
For each $(X,Y)$ matching, fix a set $D(X,Y)$ consisting of a choice
of reduced lift for each $(X,Y)$-matching.  Then let $D_\circ(X,Y)$
denote the set of all morphisms that can be obtained from the elements
of $D(X,Y)$ by adding a nonnegative number of dots
and one element of $\B_A$ near to the terminus of each string.

Using the homomorphism $\beta$ from \cref{beta}, we have that, for $X,Y \in \Heis{A}{k}$, $\Hom_{\Heis{A}{k}}(X,Y)$ is a right $\Sym(A)$-module under the action
\[
    \phi \theta := \phi \otimes \beta(\theta),\quad
    \phi \in \Hom_{\Heis{A}{k}}(X,Y),\ \theta \in \Sym(A).
\]

\begin{theo} \label{basis}
    For $X,Y \in \Heis{A}{k}$, the morphism space $\Hom_{\Heis{A}{k}}(X,Y)$ is a free right $\Sym(A)$-module with basis $D_\circ(X,Y)$.
\end{theo}

\begin{proof}
    We prove the result when $k \le 0$.  The result for $k > 0$ then follows by applying the symmetry $\Omega_k$ of \cref{Omega}.  Let $X,Y$ be two objects in $\Heis{A}{k}$.

    That $D_\circ(X,Y)$ spans $\Hom_{\Heis{A}{k}}(X,Y)$ as a right $\Sym(A)$-module follows from the fact that we have an algorithm permitting us to express any diagram representing a morphism $X \to Y$ as a linear combination of elements of $D_\circ(X,Y)$.  We use induction on the number of crossings.  Dots can be moved past crossings modulo diagrams with fewer crossings, and so we can assume that all dots are near the termini of their strings.  Then we can use the relations \cref{wreathrel,squish,curls,altbraid} to move the strings into the same configuration as one of the chosen reduced lifts modulo diagrams with fewer crossings plus some bubbles.  Finally, bubbles can be moved to the right-hand side of the diagram using \cref{bubslide}, where they become scalars in $\Sym(A)$.

    It remains to prove linear independence of $D_\circ(X,Y)$.  We
    begin with the case $X = Y =\ \uparrow^{\otimes n}$.  Consider a
    linear relation $\sum_{s=1}^N \phi_s \beta(\theta_s)$ for some
    $\phi_s \in D_\circ(X,Y)$ and $\theta_s \in \Sym(A)$.  Enumerate
    $\B_{C(A)}$ as $\cocenter{b}_1,\dotsc,\cocenter{b}_r$ so that $\cocenter{b}_1,\dotsc,\cocenter{b}_{r'}$ are
    even and $\cocenter{b}_{r'+1},\dotsc,\cocenter{b}_r$ are odd.  Choose $l \ge m > 0$ such that
    \begin{itemize}
        \item $k=m-l$;
        \item the multiplicities of dots in all $\phi_s$ arising in the linear relation are $< l$;
        \item all of the elements $\theta_s \in \Sym(A)$ are polynomials in $e_i(b_j)$, $1 \le i \le m$, $1 \le j \le r$.
    \end{itemize}

    Let $w_{i,j}$, $1 \le i \le m$, $1 \le j \le r$ be indeterminates,
    with $w_{i,j}$ even for $1 \le j \le r'$ and odd for $r'+1 \leq j \le
    r$.  Let $\KK$ be the algebraic closure of $\kk(w_{i,j} : 1 \le i
    \le m,\, 1 \le j \le r')$ and define $R$ to be the free
    supercommutative $\KK$-superalgebra generated by $w_{i,j}$, $1 \le i
    \le m$, $r'+1 \leq j \le r$.  Since the $w_{i,j}$ are odd for $r'+1 \leq j
    \le r$, $R$ is finite dimensional over $\KK$.  We will now work
    with algebras/categories that are linear over $R$, as in
    \cref{sec:GCQ}.  Recall that
    $\cocenter{b}_1^\vee,\dotsc,\cocenter{b}_r^\vee$ denotes the basis
    of $Z(A)$ that is dual to the basis $\cocenter{b}_1,\dotsc,\cocenter{b}_r$ of $C(A)$ under the pairing of \cref{slipstream}. Consider the cyclotomic wreath product algebras $\CWAR{n}{f}$ and $\CWAopR{n}{g}$ associated to the polynomials
    \[
        f(u) := u^l, \quad
        g(u) := u^m + (w_{1,1} \cocenter{b}_1^\vee + \dotsb + w_{1,r} \cocenter{b}_r^\vee) u^{m-1} + \dotsb + (w_{m,1} \cocenter{b}_1^\vee + \dotsb + w_{m,r} \cocenter{b}_r^\vee)
        \quad \in Z(\AR)_\even[u].
    \]
We claim that $f(u), g(u)$ are a generic pair in the sense of
\cref{genericity}.
To see this, note
    since $\KK$ is algebraically closed that it contains the algebraic
    closure $\overline{\kk}$ of $\kk$.  As $A_R \otimes_R A_R$ is defined over this
    algebraic closure, all elements of $\Gamma_R$ belong to
    $\overline{\kk}$.
Thus, since $\Xi_f = \{0\}$ by the definition of $f(u)$, it suffices
to show that no element of $\overline{\kk}$ is a root of the polynomial
$\chi_L(g(u))$
for any irreducible $\AR$-supermodule
    $L$. This follows because the evaluation of $\chi_L(g(u))$ at any element of
    $\bar{\kk}$ involves at least one of the even indeterminates
    $w_{i,j}$ with nonzero coefficient, hence is nonzero.

Hence, by \cref{bike}, we can use the superfunctor $\Psi_{f|g}$ of
\cref{tiki} to make $\cV(f|g)$ into a $\HeisR{A}{k}$-module
supercategory.  Since $\kk$ embeds into $R$, we have a canonical monoidal superfunctor $\Heis{A}{k} \to \HeisR{A}{k}$.  Hence we can view $\cV(f|g)$ as a module supercategory over $\Heis{A}{k}$.  We now evaluate the relation $\sum_{s=1}^N \phi_s \otimes \beta(\theta_s) = 0$ on $(\CWAR{0}{f}, \CWAopR{0}{g}) \in \cV(f|g)$ to obtain a relation in $\CWAR{n}{f}$.  It follows from \cref{AW-basis} and the choice of $l$ that the images of $\phi_1,\dotsc,\phi_N$ in $\CWAR{n}{f}$ are linearly independent over $R$.  Thus the image of $\beta(\theta_s) \in R$ is zero for each $s$.  Now, by our choice of $m$, each $\theta_s$ is a supercommuting polynomial in $e_i(b_j)$, $1 \le i \le m$ for $1 \le j \le r$.  Since
    \[
        g(u)f(u)^{-1} = u^k + (w_{1,1} \cocenter{b}_1^\vee + \dotsb + w_{1,r} \cocenter{b}_r^\vee) u^{k-1} + \dotsb + (w_{m,1} \cocenter{b}_1^\vee + \dotsb + w_{m,r} \cocenter{b}_r^\vee) u^{k-m},
    \]
    it follows from \cref{beta,bbq} that the image of $(-1)^{i-1} \beta(e_i(b_j))$ in $R$ is $w_{i,j}$.  Thus the $\theta_i$, $1 \le i \le r$, are zero.

    We have now proved the linear independence when $X = Y = \uparrow^{\otimes n}$.  The general case reduces to this case as in the proof of \cite[Prop.~5]{Kho14}.  See the proof of \cite[Th.~6.4]{BSW-K0} for further details.  The only difference in the current setting is that one picks up some tokens at each step in the proof.  Since tokens slide through all crossings and dots, these have only a minor effect on the argument.
\end{proof}

\begin{cor} \label{blue}
    The homomorphism $\beta \colon \Sym(A) \to \End_{\Heis{A}{k}} (\one)$ of \cref{beta} is an isomorphism.
\end{cor}

\Cref{basis} implies Conjectures 1, 2, and 3 of \cite{CL12}, which
concerns the special case where $A$ is a zigzag algebra (cf. \cref{motivatingeg} below).  Note that \cite[Conj.~1]{CL12} was already proved in \cite[Cor.~8.15]{RS17}.  See also Conjectures~8.10 and~8.16 of \cite{RS17}.

\section{Morita theorem\label{sec:morita}}

We say that graded supercategories $\cA$ and $\cB$ are \emph{graded Morita equivalent} if there is a graded
superequivalence $F \colon \Kar(\cA_{q,\pi}) \to \Kar(\cB_{q,\pi})$ between their
graded Karoubi envelopes.
Suppose that we are given a family $\be = (e_X)_{X \in \Xi}$ of
homogeneous idempotents $e_X \in \End_{\cA}(X)$ for some subset $\Xi$
of the object set of $\cA$.
Let $\be \cA \be$ be the full subcategory of $\Kar(\cA)$ generated by
the objects $\{(X,e_X)\colon X \in \Xi\}$. In other words, $\be \cA \be$
is the
graded supercategory with object set $\Xi$, morphisms
\[
    \Hom_{\be \cA \be}(X,Y) := e_Y \Hom_\cA(X,Y) e_X \quad
    \text{for all } X,Y \in \Xi,
\]
and composition induced by that of $\cA$.
Also let $\cA \be \cA$  be the two-sided ideal of $\cA$ generated by the
morphisms
$\left\{e_X : X \in \Xi\right\}$.
The notation $\cA \be \cA = \cA$ in the following lemma indicates that
every morphism $f: X \rightarrow Y$ in $\cA$ can be written as a
linear combination of morphisms of the form $f_Y \circ e_Z \circ f_X$
for $Z \in \Xi$, $f_X : X \rightarrow Z$ and $f_Y:Z \rightarrow Y$.

\begin{lem}\label{zoom}
If $\cA \be \cA = \cA$ then the graded supercategories $\cA$ and $\be
\cA \be$ are graded Morita equivalent.
\end{lem}

\begin{proof}
If $A$ and $B$ are the locally unital algebras associated to $\cA$ and
$\cB := \be \cA \be$ as in \cref{yoneda2}, then $\cA$ and $\cB$ are graded Morita
equivalent if and only if $A$ and $B$ are graded Morita equivalent
superalgebras in
the usual sense, i.e.\ there is an
equivalence between their module categories
$\gsmod A$ and $\gsmod B$ that respects degrees and parities of
morphisms.
Also the assumption $\cA \be \cA = \cA$ implies that $\left\{e_X A : X
\in \Xi\right\}$ is a projective generating family for $\gsmod A$.
With this dictionary in mind, the lemma follows from a standard
algebraic result, e.g.\ see \cite[Theorem 2.4]{BD}.
\end{proof}

Now return to the usual setup, so $A$ is a symmetric graded Frobenius
superalgebra as in \cref{sec:affwreath}.
Let $e \in A$ be a homogeneous idempotent such that $A = A e A$;
equivalently, $eA$ is a projective generator for $\gsmod A$.
Consider the decomposition
\[
    A = eAe \oplus (1-e)A(1-e) \oplus eA(1-e) \oplus (1-e)Ae.
\]
Under the bilinear form induced by the trace map of $A$, the first two
summands are self-dual, while the last two summands are dual to each
other.  In particular, the trace map of $A$ restricts to a
nondegenerate trace on $eAe$, and so $eAe$ is itself a graded
Frobenius superalgebra, which is graded Morita equivalent to $A$.

Let $\B_{eAe}$ be a homogeneous basis of $eAe$, and let $\B_A'$ be a homogeneous basis of $(1-e)A(1-e) \oplus eA(1-e) \oplus (1-e)Ae$.  Then $\B_A := \B_{eAe} \sqcup \B_A'$ is a homogeneous basis of $A$, and we have
\begin{equation} \label{vaporize}
    \sum_{b \in \B_A}
    \begin{tikzpicture}[anchorbase]
        \draw[->] (-0.2,-0.4) -- (-0.2,0.4);
        \draw[->] (0.2,-0.4) -- (0.2,0.4);
        \blacktoken[east]{-0.2,-0.2}{e};
        \blacktoken[east]{-0.2,0}{b};
        \blacktoken[east]{-0.2,0.2}{e};
        \blacktoken[west]{0.2,-0.2}{e};
        \blacktoken[west]{0.2,0}{b^\vee};
        \blacktoken[west]{0.2,0.2}{e};
    \end{tikzpicture}
    =
    \sum_{b \in \B_{eAe}}
    \begin{tikzpicture}[anchorbase]
        \draw[->] (-0.2,-0.4) -- (-0.2,0.4);
        \draw[->] (0.2,-0.4) -- (0.2,0.4);
        \blacktoken[east]{-0.2,0}{b};
        \blacktoken[west]{0.2,0}{b^\vee};
    \end{tikzpicture}
    \ .
\end{equation}
In other words, placing tokens labeled $e$ at both ends of both strands turns the teleporter in $\Heis{A}{k}$ into the teleporter in $\Heis{eAe}{k}$.

Define
$e_\uparrow =
\begin{tikzpicture}[anchorbase]
    \draw[->] (0,-0.15) --(0,0.2);
    \blacktoken[east]{0,0}{e};
\end{tikzpicture}
$
and
$e_\downarrow =
\begin{tikzpicture}[anchorbase]
    \draw[<-] (0,-0.2) -- (0,0.15);
    \blacktoken[east]{0,0}{e};
\end{tikzpicture}
$.
Let $e_{X_n \otimes \dotsb \otimes X_1} := e_{X_n} \otimes
\dotsb \otimes e_{X_1}$ for $X_1,\dotsc,X_n \in \{\uparrow,
\downarrow\}$; in particular, $e_{\one} := 1_\one$.  Then set $\be =
(e_X)_{X \in \Heis{A}{k}}$, where here $X$ ranges over all objects in $\Heis{A}{k}$, and define  the
graded supercategory $\be \Heis{A}{k} \be$
as in the opening paragraph of the
section.
In fact, since we have that $e_Y \otimes e_X = e_{Y \otimes X}$ for all $X, Y
\in \Heis{A}{k}$,
this is a monoidal subcategory of $\Kar(\Heis{A}{k})$, so it is also a
graded monoidal supercategory.

\begin{lem} \label{cards}
The inclusion $eAe \hookrightarrow A$ induces an isomorphism
between the graded monoidal supercategories $\Heis{eAe}{k}$ and $\be
\Heis{A}{k} \be$.
\end{lem}

\begin{proof}
By definition, the categories $\Heis{eAe}{k}$ and $\be \Heis{A}{k}
\be$ have the same objects.
The natural graded monoidal superfunctor $F \colon \Heis{eAe}{k} \to
\be \Heis{A}{k} \be$ is the identity on objects and, on morphisms, it sends the diagram
for a morphism in $\Heis{eAe}{k}$ to the morphism in $\Heis{A}{k}$
represented by the same picture.
This actually belongs to the
subcategory $\be \Heis{A}{k} \be$ since we can use the idempotent
property $e^2 = e$ to
add an additional token labeled $e$ to both ends of all strands
appearing in the diagram.
To see that $F$ is well defined, one
needs to check that the relations are preserved, which follows
ultimately due to \cref{vaporize}.
Finally, to see that $F$ is an isomorphism, consider
the composition $eAe \hookrightarrow A \twoheadrightarrow
C(A)$.  In $C(A)$, we have $eAe = AeA = A$, and so this composition is
surjective.
Since Morita equivalent superalgebras have isomorphic centers, it then
follows from \cref{slipstream} that the inclusion $eAe \hookrightarrow
A$ induces an isomorphism  $C(eAe) \xrightarrow{\cong} C(A)$ of graded vector superspaces.
It remains to apply
\cref{basis} to see that $F$ defines an isomorphism on all morphism spaces.
\end{proof}

\begin{theo}\label{moritathm}
Let $F \colon \Heis{eAe}{k} \rightarrow \Kar(\Heis{A}{k}_{q,\pi})$
be the composition of the isomorphism
$\Heis{eAe}{k} \xrightarrow{\cong}
\be \Heis{A}{k} \be$ from \cref{cards} and the natural inclusion
$\be \Heis{A}{k} \be \hookrightarrow \Kar(\Heis{A}{k}_{q,\pi})$, so that
$F$ takes $\uparrow \mapsto (\uparrow, e)$ and $\downarrow
\mapsto (\downarrow, e^\op)$.
The canonical extension of $F$ to the graded Karoubi envelope
defines a graded monoidal superequivalence
\[
\tilde F \colon
\Kar\left(\Heis{eAe}{k}_{q,\pi}\right) \rightarrow \Kar\left(\Heis{A}{k}_{q,\pi}\right).
\]
In particular, $\Heis{eAe}{k}$ and $\Heis{A}{k}$ are graded Morita equivalent.
\end{theo}

\begin{proof}
    Since $AeA = A$, it follows easily that, for all $X \in \Heis{A}{k}$, $\End_{\Heis{A}{k}}(X) e_X \End_{\Heis{A}{k}}(X)$ contains the identity morphism $1_X$, and hence $\Heis{A}{k} \be \Heis{A}{k} = \Heis{A}{k}$.  Then the theorem follows from \cref{cards,zoom}.
\end{proof}

As mentioned in the introduction, it follows from \cref{moritathm} that, when $A$ is semisimple and purely even with trivial grading, the graded Karoubi envelope of $\Heis{A}{k}$ is monoidally equivalent to the graded Karoubi envelope of the symmetric product of $N$ copies of $\Heis{\kk}{k}$, where $N$ is the number of pairwise inequivalent irreducible $A$-supermodules.  This reduces us to the setting of \cite{MS18,Bru18,BSW-K0}.

\begin{rem}\label{color}
Assume that the hypothesis \cref{virginia}
from the introduction is satisfied and fix idempotents
$e_1,\dots,e_N$ as was done after \cref{coass1}.
Let $e := e_1+\cdots+e_N$, so that $A$ is graded Morita equivalent to
the basic graded superalgebra $eAe$. Then
\cref{moritathm} can be used to replace the Frobenius Heisenberg
category $\Heis{A}{k}$ by $\Heis{eAe}{k}$. The latter
is the full monoidal subcategory
of $\Kar(\Heis{A}{k})$ generated by  the objects $(\uparrow, e_i)$ and
$(\downarrow, e_i)$ for $i=1,\dots,N$. It may be viewed as a simpler
``colored'' version of the original category. For this, one denotes a string labeled by the idempotent $e_i$ instead
by adding an additional color $i$ to the string. Then tokens labeled
by elements of $e_j A e_i$
change the color of a string from $i$ to $j$.
This is particularly convenient when $A$ is a zigzag
algebra like in \cref{motivatingeg}. For this strings are colored by
 vertices in the quiver, and one only needs tokens for each
 edge of this quiver.
\end{rem}

\section{The lattice Heisenberg algebra\label{sec:lattic}}

We assume throughout this section that $A$ satisfies condition
\cref{virginia} from the introduction (see also \cref{Dagger}).  Since
the trace map has degree $-2\dA$, it follows that $A$ has grading $A =
\bigoplus_{i=0}^{2\dA} A_i$, $\dA \ge 0$, with $A_{2d} \ne 0$.  If
$\dA = 0$, then $A = A_0$ is purely even and semisimple.  In this
case, it makes sense to forget gradings and work in the setting of
strict $\kk$-linear monoidal categories.  This amounts to setting
$q=1$ and $\pi=1$ everywhere below.  Since the constructions of the
current section depend only on $A$ up to graded Morita equivalence, we will also assume here that $A$ is a basic algebra.

The (normalized) \emph{Cartan pairing} is the Hermitian form
\begin{equation}\label{cartanform}
    \langle -, - \rangle \colon K_0(\pgsmod A) \times K_0(\pgsmod A) \to \Zq,\quad
    \langle [P], [P'] \rangle := q^{-\dA}\grdim \Hom_A(P,P').
\end{equation}
Here, \emph{Hermitian} means that it is sesquilinear, i.e.\ $\langle v, z w \rangle = z \langle v, w \rangle = \langle \bar{z} v, w \rangle$, and conjugate-symmetric, i.e.\ $\langle v, w \rangle = \overline{\langle w, v \rangle}$, where $\bar{\ } \colon \Zq \rightarrow \Zq, q \mapsto q^{-1}, \pi \mapsto \pi$; the latter property is a consequence of the assumption that $A$ is graded Frobenius.
\details{
    It suffices to consider the basis of $K_0(\pgsmod A)$ given by the classes $[e_i A]$, $1 \le i \le N$, of the indecomposable projective $A$-supermodules.  For $1 \le i,j \le N$, the trace map on $A$ induces a nondegenerate $\kk$-bilinear pairing of graded $\kk$-vector spaces
    \[
        e_i A e_j \times e_j A e_i \to \kk,\quad
        (e_i a e_j, e_j b e_i) \mapsto \tr(e_i a e_j b e_i).
    \]
    It follows that $\Hom_A(e_i A, e_j A) \cong e_j A e_i$ is dual to $\Hom_A(e_j A, e_i A) \cong e_i A e_j$.  Since the trace map has degree $-2\dA$, this implies the desired conjugate-symmetry.
}
Choose a decomposition $1 = e_1 + \dotsc + e_N$ of the unit element of
$A$ as a sum of pairwise orthogonal primitive homogeneous idempotents
like after \cref{coass1}. Then the \emph{graded Cartan matrix} of $A$ is the Hermitian matrix $\left(\langle i,j\rangle\right)_{1 \leq i,j \leq N}$ defined from
\begin{equation}\label{grcartanmatrix}
    \langle i,j \rangle_A
    := \langle [e_i A], [e_j A] \rangle
    =q^{-d} \grdim e_i A e_j.
\end{equation}

\begin{eg}\label{motivatingeg}
    As a motivating example, the reader may like to think about the case when $A$ is a \emph{zigzag algebra} in the sense of \cite{HK}. For example, the zigzag algebra of type $\mathrm{A}_4$ is defined by the quiver and relations
    \begin{align*}
        \xymatrix{
        \stackrel{1}{\bullet}
        \ar@/^/[r]^{a_{21}}&\ar@/^/[l]^{a_{12}}
        \stackrel{2}{\bullet}
        \ar@/^/[r]^{a_{32}}&\ar@/^/[l]^{a_{23}}
        \stackrel{3}{\bullet}
        \ar@/^/[r]^{a_{43}}&\ar@/^/[l]^{a_{34}}
        \stackrel{4}{\bullet}}
        &\qquad\qquad
        \begin{array}{l}
        a_{12} a_{23} = a_{23} a_{34} = a_{32} a_{21} =
        a_{43}a_{32} =0,\\
        a_{23}a_{32}+a_{21}a_{12}=a_{34}a_{43}+a_{32}a_{23}=0,
        \end{array}
    \end{align*}
    with positive grading defined by path length, $\Z/2$-grading induced by the $\Z$-grading, and $d=1$.  The vertex idempotents $e_1,e_2,e_3,e_4$ give a basis for $A_0 = A_{0,\even}$, and the central elements $c_1 = a_{12}a_{21}$, $c_2 = a_{23}a_{32}$, $c_3 = a_{34} a_{43}$, $c_4 = -a_{43} a_{34}$ give a basis for $A_2 = A_{2,\even}$.  The trace sends $c_i \mapsto 1$.  The graded Cartan matrix is
    \[
        \begin{pmatrix}
            q+q^{-1}&1&0&0\\
            1&q+q^{-1}&1&0\\
            0&1&q+q^{-1}&1\\
            0&0&1&q+q^{-1}
        \end{pmatrix}.
    \]
    Also in this case, the center $Z(A)$ and cocenter $C(A)$ are of
    dimension five, with dual bases $\{e_1+e_2+e_3+e_4, c_1,c_2,c_3,c_4\}$ and $\{\cocenter{c}_1 = \cocenter{c}_2= \cocenter{c}_3 = \cocenter{c}_4, \cocenter{e}_1, \cocenter{e}_2, \cocenter{e}_3, \cocenter{e}_4\}$.
\end{eg}

We can associate to $A$ a natural lattice, and hence a lattice
Heisenberg algebra, as we now describe.
For $z = \sum_{n \in \Z, p \in \Z/2} z_{n,p} q^n \pi^p \in \Zq$, define $\extdim{r}{z} \in \Zq$ for $r \in \N$ by
\begin{equation} \label{moose}
    \prod_{n \in \Z} \frac{(1+q^{n} u)^{z_{n,\bar 0}}}{(1-q^{n}\pi
      u)^{z_{n,\bar 1}}}
    = \sum_{r \ge 0} \extdim{r}{z} u^r.
\end{equation}
The significance of this definition is that for a finite dimensional graded vector space $W$, we have
\[
    \grdim \left({\textstyle\bigwedge^r}(W)\right) = \extdim{r}{\grdim W},
\]
where the exterior superalgebra $\bigwedge(W)$ is the quotient of the tensor superalgebra $T(W) = \bigoplus_{r \in \N} W^{\otimes r}$ by the two-sided ideal generated by the elements $v \otimes w + (-1)^{\bar{v} \bar{w}} w \otimes v$; then $\bigwedge^r(W)$ is the image of $W^{\otimes r}$ in $\bigwedge(W)$.  It follows from the definition \cref{moose} that $\extdim{r}{n} = \binom{n}{r}$ for $n \in \Z$.  We also have $\extdim{r}{\bar{z}} = \overline{\extdim{r}{z}}$, where the involution $\bar{\ }$ is the one introduced below \cref{cartanform}, and
\begin{equation} \label{poke}
    \sum_{\substack{r,t \ge 0 \\ r+t=n}} \extdim{r}{z} \extdim{t}{-z} = 0
    \qquad \text{for } z \in \Zq,\ n \ge 1.
\end{equation}
\details{
    This follows from the identity
    \[
        1 = \prod_{n \in \Z} \frac{(1+q^{2n} u)^{z_{2n}}}{(1-q^{2n+1} u)^{z_{2n+1}}} \frac{(1+q^{2n} u)^{-z_{2n}}}{(1-q^{2n+1} u)^{-z_{2n+1}}}
        = \left( \sum_{r \ge 0} \extdim{r}{z} u^r \right) \left( \sum_{t \ge 0} \extdim{r}{-z} u^t \right).
    \]
}

For $n \in \Z$, define the $\Q$-algebra homomorphism
\[
    \theta_n \colon \Qq \to \Qq,\quad
    \theta_n(q) = q^n.
\]
Recalling \cref{dogs,cats}, for $0 \neq k \in \Z$, we form the \emph{Heisenberg double} $\Sym_\Qq^{\otimes N} \#_{\Qq} \Sym_\Qq^{\otimes N}$ with respect to the sesquilinear Hopf pairing determined by
\begin{equation} \label{Hopfpair}
    \langle -, - \rangle_k \colon \Sym_\Qq^{\otimes N} \otimes_\Qq \Sym_\Qq^{\otimes N},\quad
    \langle p_{m,i}, p_{n,i} \rangle_k = \delta_{m,n} n \theta_n \left( q^{-d} \qint{k}_d \langle i,j \rangle_A \right),
\end{equation}
where
$\qint{k}_d
    := q^{(k-1)\dA} + q^{(k-3)\dA} + \dotsb + q^{-(k-3)\dA} + q^{-(k-1)\dA}$.
By definition, $\Sym_\Qq^{\otimes N} \#_{\Qq} \Sym_\Qq^{\otimes N}$ is the $\Qq$-module $\Sym_\Qq^{\otimes N} \otimes_{\Qq} \Sym_\Qq^{\otimes N}$ with associative multiplication defined by
\[
    (e \otimes f)(g \otimes h)
    := \sum_{(f),(g)} \langle f_{(1)}, g_{(2)} \rangle_k e g_{(1)} \otimes f_{(2)} h,
\]
recalling the comultiplication \cref{Symcomult}.  For $f \in \Sym_\Qq^{\otimes N}$, we write $f^-$ and $f^+$ for the elements $f \otimes 1$ and $1 \otimes f$ of $\Sym_\Qq^{\otimes N} \#_{\Qq} \Sym_\Qq^{\otimes N}$, respectively.  Then $\Sym_\Qq^{\otimes N} \#_{\Qq} \Sym_\Qq^{\otimes N}$ is generated by the elements $\{p_{n,i}^\pm : n \in \N,\ 1 \le i \le N\}$, subject to the relations
\begin{equation} \label{fitting}
    \begin{gathered}
        p_{0,i}^+ = p_{0,i}^- = 1,\quad
        p_{m,i}^+ p_{n,j}^+ = p_{n,j}^+ p_{m,i}^+,\quad
        p_{m,i}^- p_{n,j}^- = p_{n,j}^- p_{m,i}^-,
        \\
        p_{m,i}^+ p_{n,j}^- = p_{n,j}^- p_{m,i}^+ + \delta_{m,n} n \theta_n \left( \qint{k}_d \langle i,j \rangle_A \right).
  \end{gathered}
\end{equation}

The pairing under \cref{Hopfpair} of $h_{m,i}$ and $h_{n,j}$ lies in $\Zq$, as follows for example by comparing the coefficients appearing in \cite[Th.~5.3]{Sua17} to \cite[(2.2)]{Sua17}.  Hence we can restrict to obtain a biadditive form $\langle -, - \rangle_k \colon \Sym_\Zq^{\otimes N} \times \Sym_\Zq^{\otimes N} \to \Zq$.  The associated Heisenberg double
\begin{equation}\label{hd}
    \rHeis_k(A) := \Sym_\Zq^{\otimes N} \#_{\Zq} \Sym_\Zq^{\otimes N}
\end{equation}
is a natural $\Zq$-form for $U(\fh_L) \cong \Sym_\Qq^{\otimes N} \#_{\Qq} \Sym_\Qq^{\otimes N}$.  We call $\rHeis_k(A)$ a \emph{lattice Heisenberg algebra}.  Then $\rHeis_k(A)$ is generated as a $\Zq$-algebra by the elements $\{h_{n,i}^+, e_{n,i}^- : n \in \N,\ 1 \le i \le N\}$ subject to the relations
\begin{equation} \label{synth1}
    \begin{gathered}
        h_{0,i}^+ = e_{0,i}^- = 1,\quad
        h_{m,i}^+ h_{n,j}^+ = h_{n,j}^+ h_{m,i}^+,\quad
        e_{m,i}^- e_{n,j}^- = e_{n,j}^- e_{m,i}^-,
        \\
        h_{m,i}^+ e_{n,j}^- = \sum_{r=0}^{\min(m,n)} \extdim{r}{\qint{k}_d \langle i,j \rangle_A} e_{n-r,j}^- h_{m-r,i}^+.
    \end{gathered}
\end{equation}
We refer the reader to \cite[\S5]{Sua17} and \cite[App.~A]{LRS18}
where this and other presentations are computed. When $k=0$, the form
\cref{Hopfpair} is zero.  In this case we simply define
$\Sym_\Qq^{\otimes N} \#_{\Qq} \Sym_\Qq^{\otimes N} =
\Sym_\Qq^{\otimes N} \otimes_{\Qq} \Sym_\Qq^{\otimes N}$ (tensor
product of graded algebras).  The presentation \cref{synth1} continues
to hold in this case.
The next lemma gives an equivalent formulation of this presentation
which is more convenient for the present purposes.

\begin{lem} \label{marahau}
    As a $\Zq$-algebra, $\rHeis_k(A)$ is generated by $\{h_{n,i}^+, e_{n,i}^- : n \in \N,\ 1 \le i \le N\}$ subject to the relations
    \begin{equation} \label{synth2}
        \begin{gathered}
            h_{0,i}^+ = e_{0,i}^- = 1,\quad
            h_{m,i}^+ h_{n,j}^+ = h_{n,j}^+ h_{m,i}^+,\quad
            e_{m,i}^- e_{n,j}^- = e_{n,j}^- e_{m,i}^-,
            \\
            e_{m,i}^- h_{n,j}^+ = \sum_{r=0}^{\min(m,n)} \extdim{r}{-\qint{k}_d \langle j,i \rangle_A} h_{n-r,j}^+ e_{m-r,i}^-.
        \end{gathered}
    \end{equation}
\end{lem}

\begin{proof}
    It suffices to prove that the last relation in \cref{synth1} is equivalent to the last relation in \cref{synth2}.  We prove that \cref{synth1} implies \cref{synth2}; the reverse implication is similar.  We proceed by induction on $n$, the case $n=0$ being immediate.  For $n \ge 1$, we have
    \begin{align*}
        e_{m,i}^- h_{n,j}^+
        \ &\overset{\mathclap{\cref{synth1}}}{=}\
        h_{n,j}^+ e_{m,i}^- - \sum_{r=1}^{\min(m,n)} \extdim{r}{[k]_d\langle j,i \rangle_A} e_{m-r,i}^- h_{n-r,j}^+ \\
        &= h_{n,j}^+ e_{m,i}^- - \sum_{r=1}^{\min(m,n)} \sum_{t=0}^{\min(m-r,n-r)} \extdim{r}{[k]_d\langle j,i \rangle_A} \extdim{t}{-[k]_d \langle j, i \rangle_A} h_{n-r-t,j}^+ e_{m-r-t,i}^- \\
        &= h_{n,j}^+ e_{m,i}^- - \sum_{l=1}^{\min(m,n)} \sum_{\substack{r \ge 1,\, t \ge 0 \\ r+t=l}}  \extdim{r}{[k]_d\langle j,i \rangle_A} \extdim{t}{-[k]_d \langle j, i \rangle_A} h_{n-l,j}^+ e_{m-l,i}^- \\
        &\overset{\mathclap{\cref{poke}}}{=}\ h_{n,j}^+ e_{m,i}^- + \sum_{l=1}^{\min(m,n)} \extdim{l}{-[k]_d \langle j, i \rangle_A} h_{n-l,j}^+ e_{m-l,i}^-.  \qedhere
    \end{align*}
\end{proof}

\begin{cor} \label{ramen}
    We have a conjugate-linear ring isomorphism
    \[
        \omega_{k} \colon \rHeis_k(A) \xrightarrow{\cong} \rHeis_{-k}(A),\quad
        s_{\lambda,i}^\pm \mapsto s_{\lambda^T,i}^\mp.
    \]
\end{cor}

\begin{proof}
    This follows from \cref{marahau}, together with the fact that
    \begin{equation} \label{walk}
        \overline{\extdim{r}{-[k]_d\langle j,i \rangle_A}}
        = \extdim{r}{- \overline{[k]_d} \overline{\langle j,i \rangle_A}}
        = \extdim{r}{[-k]_d \langle i,j \rangle_A}.  \qedhere
    \end{equation}
\end{proof}

\begin{lem}
    For $k = l + m$, there is an algebra homomorphism $\delta_{l|m} \colon \rHeis_k(A) \to \rHeis_l(A) \otimes_{\Zq} \rHeis_m(A)$ given by
    \begin{equation} \label{puppet}
        \begin{aligned}
            \delta_{l|m}(p_{n,i}^+) &= p_{n,i}^+ \otimes 1 + 1\otimes q^{-l\dA n} p_{n,i}^+,
            &
            \delta_{l|m}(p_{n,i}^-) &= q^{m \dA n} p_{n,i}^- \otimes 1 + 1 \otimes p_{n,i}^-,
            \\
            \delta_{l|m}(h_{n,i}^+) &= \sum_{r=0}^n h_{n-r,i}^+ \otimes q^{-l \dA r} h_{r,i}^+,
            &
            \delta_{l|m}(h_{n,i}^-) &= \sum_{r=0}^n q^{m \dA r} h_{r,i}^- \otimes h_{n-r,i}^-,
            \\
            \delta_{l|m}(e_{n,i}^+) &= \sum_{r=0}^n e_{n-r,i}^+ \otimes q^{-l \dA r} e_{r,i}^+,
            &
            \delta_{l|m}(e_{n,i}^-) &= \sum_{r=0}^n q^{m \dA r} e_{r,i}^- \otimes e_{n-r,i}^-.
        \end{aligned}
    \end{equation}
\end{lem}

\begin{proof}
    Extend scalars to $\Qq$ and define $\delta_{l|m}$ via the first two expressions in \cref{puppet}.  To see that $\delta_{l|m}$ is well defined, we must check the last relation in \cref{fitting}, since the others are clear.  We have
    \begin{align*}
        \delta_{l|m} ( p_{n,i}^+ p_{r,j}^- )
        &= q^{m \dA r} p_{n,i}^+ p_{r,j}^- \otimes 1 + p_{n,i}^+ \otimes p_{r,j}^- + q^{(mr - ln) \dA} p_{r,j}^- \otimes p_{n,i}^+ + q^{-l \dA n} \otimes p_{n,i}^+ p_{r,j}^-, \\
        \delta_{l|m} ( p_{r,j}^- p_{n,i}^+ )
        &= q^{m \dA r} p_{r,j}^- p_{n,i}^+\otimes 1 + p_{n,i}^+ \otimes p_{r,j}^- + q^{(mr - ln) \dA} p_{r,j}^- \otimes p_{n,i}^+ + q^{-l \dA n} \otimes p_{r,j}^- p_{n,i}^+.
    \end{align*}
    Thus
    \begin{align*}
        \delta_{l|m} ( p_{n,i}^+ p_{r,j}^- - p_{r,j}^- p_{n,i}^+ )
        &= \delta_{n,r} n \left( q^{m\dA n} \theta_n([l]_d \langle i,j \rangle_A) + q^{-l\dA n} \theta_n([m]_d \langle i,j \rangle_A) \right)
        \\
        &= \delta_{n,r} n \theta_n \left( \left( q^{m\dA} [l]_d + q^{-l\dA} [m]_d \right) \langle i,j \rangle_A \right)
        \\
        &= \delta_{n,r} n \theta_n([k]_d \langle i,j \rangle_A).
    \end{align*}
    The expressions for $\delta_{l|m}(h_{n,i}^\pm)$ and $\delta_{l|m}(e_{n,i}^\pm)$ now follow just as for the usual comultiplication on $\Sym_\Qq$.  Restricting scalars to $\Zq$ then gives the desired result.
\end{proof}

\section{Grothendieck ring\label{sec:k0}}

In this final section, our goal is compute the Grothendieck ring of
the additive Karoubi envelope of $\Heis{A}{k}$ and show that it is
isomorphic to the lattice Heisenberg algebra $\rHeis_k(A)$ from
\cref{sec:lattic}, under the assumption that {\em $A$ satisfies condition
\cref{virginia} from the introduction}.  (However, see \cref{Dagger}.)
When $A$ is purely even and semisimple, this follows easily from
\cref{moritathm} and results of \cite{BSW-K0}.  Therefore, until the
statement of \cref{K0isom}, {\em we assume also that the grading on $A$ is nontrivial}, i.e. $d > 0$.
Recall the objects $S_{\blambda}^{\pm}, H_{n,i}^{\pm}, E_{n,i}^{\pm}\in\Kar(\Heis{A}{k}_{q,\pi})$
from \cref{dungeons,dragons,dd}.

\begin{lem} \label{coffee}
    The category $\Kar(\Heis{A}{k}_{q,\pi})$ is Krull--Schmidt.  If $k \ge 0$ (resp.\ $k \le 0$) then the objects $S_\bmu^- \otimes S_\blambda^+$ (resp.\ $S_\bmu^+ \otimes S_\blambda^-$), $\bmu,\blambda \in \cP^N$, are a complete set of pairwise non-isomorphic indecomposable projective objects in $\Kar(\Heis{A}{k}_{q,\pi})$, up to grading and parity shifts.
\end{lem}

\begin{proof}
    It suffices to prove the result for $k \ge 0$, since the case $k
    \le 0$ then follows by applying $\Omega_k$. By
    \cref{invrel}, every object of $\Kar(\Heis{A}{k}_{q,\pi})$ is
    isomorphic, up to grading and parity shifts, to a summand of $\downarrow^{\otimes m} \otimes \uparrow^{\otimes n}$ for some $m,n \in \N$.  We claim that
    \[
        \End_{\Heis{A}{k}} (\downarrow^{\otimes m} \otimes \uparrow^{\otimes n})_i \cong
        \begin{cases}
            0 & \text{if } i < 0, \\
            \WA[A_0^\op]{m} \otimes \WA[A_0]{n} & \text{if } i=0.
        \end{cases}
    \]
    Using \cite[Ch.~XII, Prop.~3.3]{Bass}, the statement that $\Kar(\Heis{A}{k}_{q,\pi})$ is Krull--Schmidt then follows from this claim and the fact that $\WA[A_0^\op]{m}\otimes\WA[A_0]{n}$ is finite dimensional for all $m,n \in \N$, while the statement about indecomposable projective objects follows from \cref{trips}.

    The proof of the claim is essentially the same as that of \cite[Lem.~10.2]{RS17}.  However, the basis theorem (\cref{basis}) allows us to be more concise.  First note that
    \[
        \deg \ccbubble{a}{n-k-1} = 2(n-1) \dA + \deg a.
    \]
    Together with \cref{weed}, it follows that $\beta(\Sym(A))$ is
    positively graded with degree zero piece equal to $\kk 1_\one$.
    Then, by \cref{basis,trappist}, we see that $\End_{\Heis{A}{k}}
    (\downarrow^{\otimes m} \otimes \uparrow^{\otimes n})_i = 0$ if $i < 0$ and that any basis elements of degree zero must have no cups, caps, or dots.  Then the isomorphism for $i=0$ is induced by \cref{imath,jmath}.
\end{proof}

Consider the $\Zpi[q,q^{-1}]$-module homomorphism
\begin{equation} \label{gamma}
  \gamma_{k} \colon \rHeis_k(A) \to K_0\left(\Kar\big(\Heis{A}{k}_{q,\pi}\big)\right),\quad
  \begin{cases}
    s_{\blambda}^- s_{\bmu}^+ \mapsto [ S_{\blambda}^- \otimes S_{\bmu}^+ ] & \text{if } k \ge 0,
    \\
    s_{\blambda}^+ s_{\bmu}^- \mapsto [ S_{\blambda}^+ \otimes S_{\bmu}^- ] & \text{if } k \le 0.
  \end{cases}
\end{equation}
In particular, $\gamma_{k} ( h_{n,i}^\pm ) = [H_{n,i}^\pm]$ and
$\gamma_{k} ( e_{n,i}^\pm ) = [E_{n,i}^\pm]$.  By \cref{coffee},
$\gamma_{k}$ is an isomorphism of $\Zq$-modules.  Our goal is to show that
it is an isomorphism of $\Zq$-algebras.

We begin by forgetting the grading, working in the
$\Pi$-envelope $\Heis{A}{k}_\pi$ rather than in $\Heis{A}{k}_{q,\pi}$.
This is merely a monoidal $\Pi$-supercategory rather than a graded
monoidal $(Q,\Pi)$-supercategory.
We use the notation $\ug{S}_\blambda^{\pm}$
to denote the objects of $\Kar(\Heis{A}{k}_\pi)$ defined in the same
way as \cref{dungeons,dd}.
Also let $\rtHeis_k(A) := \Zpi \otimes_{\Zq} \rHeis_k(A)$, where we
view $\Zpi$ as a $\Zq$-module via the map $q \mapsto 1$.
Note that $\rtHeis_k(A)$ can be defined directly by mimicking \cref{hd}
over the ring $\Zpi$ instead of $\Zq$, replacing $q$ by $1$ in
formulae such as \cref{Hopfpair}.
The ungraded analog of the map \cref{gamma} is
the $\Zpi$-module homomorphism
\begin{equation}\label{tgamma}
  \tgamma_{k} \colon \rtHeis_k(A) \to K_0\left(\Kar\big(\Heis{A}{k}_{\pi}\big)\right),\quad
        \begin{cases}
            \underline{s}_{\blambda}^- \underline{s}_{\bmu}^+ \mapsto [ \ug{S}_{\blambda}^- \otimes \ug{S}_{\bmu}^+ ] & \text{if } k \ge 0,
            \\
            \underline{s}_{\blambda}^+ \underline{s}_{\bmu}^- \mapsto [ \ug{S}_{\blambda}^+ \otimes \ug{S}_{\bmu}^- ] & \text{if } k \le 0,
        \end{cases}
\end{equation}
where $\underline{s}_{\blambda}^{\pm} := 1\otimes s_{\blambda}^{\pm}$.
We do not have available the ungraded analog of Lemma~\ref{coffee}, so
do not know that this map is an isomorphism. However, we have the
following, which is proved in a completely
different way, exploiting the categorical comultiplication.

\begin{lem} \label{shadow}
    Assuming \cref{virginia} and $d > 0$, the map $\tgamma_k$ is an injective $\Zpi$-algebra homomorphism.
\end{lem}

\begin{proof}
    We essentially follow the argument in the proof of
    \cite[Th.~7.2]{BSW-K0}.  Namely, we prove the result by
    categorifying a natural faithful representation of $\rtHeis_k(A)$.
    The \emph{Fock space representation} of $\rtHeis_{-1}(A)$ is the $\Zpi$-module
$\Sym^{\otimes N}_\Zpi$ defined as in \cref{dogs} but replacing $\Zq$
by $\Zpi$. It is a $\rtHeis_{-1}(A)$-module so that, for $f \in \Sym^{\otimes
  N}_\Zpi$, the element $f^+$ acts by left multiplication by $f$, and
$f^-$ acts by the adjoint operator with respect to the
form that is \cref{Hopfpair} specialized at $q=1$.
Let $(\Sym^{\otimes N}_\Zpi)^\vee$ be the $\rtHeis_1(A)$-module
obtained from pulling back the $\rtHeis_{-1}(A)$-action on
$\Sym^{\otimes N}_\Zpi$ along the isomorphism $\underline{\omega}_{1}
\colon \rtHeis_1(A) \xrightarrow{\cong} \rtHeis_{-1}(A)$ arising from
\cref{ramen} specialized at $q=1$.  More generally, for any $l,m \ge 0$ and $k:= m-l$, the tensor product $V(l|m) := (\Sym^{\otimes N}_\Zpi)^{\otimes l} \otimes_\Zpi \left( (\Sym^{\otimes N}_\Zpi)^\vee \right)^{\otimes m}$ is naturally a $\rtHeis_k(A)$-module.  It has a natural basis
    \[
      \left\{ \underline{s}_{\blambda^{(1)}} \otimes \dotsb \otimes \underline{s}_{\blambda^{(l)}} \otimes \underline{s}_{\bmu^{(1)}}^\vee \otimes \dotsb \otimes \underline{s}_{\bmu^{(m)}}^\vee :
      \blambda^{(1)}, \dotsc, \blambda^{(l)}, \bmu^{(1)}, \dotsc, \bmu^{(m)} \in \cP^N \right\},
    \]
    where we use the notation $\underline{s}_\blambda^\vee$ to denote
    $\underline{s}_\blambda$ viewed as an element of $(\Sym^{\otimes
      N}_\Zpi)^\vee$ rather than $\Sym^{\otimes N}_\Zpi$.  The associated representation
    \[
      \psi_{l|m} \colon \rtHeis_k(A) \to \End_\Zpi \left( V(l|m) \right)
    \]
    is faithful providing $l + m > 0$; this is easy to see when both
    $l> 0$ and $m > 0$ which may be assumed for the present purposes.

    For $f(u) \in \kk[u]$ of degree one, the natural map $\WA{n}
    \hookrightarrow \CWA{n}{f}$ is an isomorphism by
    \cite[Cor.~6.13]{Sav20}.  Thus, the $\Heis{A}{-1}$-module
    supercategory $\cV(f)$ from \cref{upaction} is the category
    $\psmod \WA{n}$, and there is an isomorphism of $\Zpi$-modules
    \[
      \Sym^{\otimes N}_\Zpi \xrightarrow{\cong} K_0(\cV(f)),\quad
      \underline{s}_\blambda \mapsto \left[e_\blambda \WA{|\blambda|} \right],
    \]
by the last part of \cref{trips}.
Analogous statements hold for the $\Heis{A}{1}$-module category
$\cV(g)^\vee$ from \cref{downaction} when $g(u) \in \kk[u]$ is of
degree one.  More generally, for $u_1,\dotsc,u_l,v_1,\dotsc,v_m \in
\kk$, we have the Karoubian $\Pi$-supercategory
    \[
      \cV(u_1,\dotsc,u_l | v_1,\dotsc,v_m)
      := \cV(u-u_1) \boxtimes \dotsb \boxtimes \cV(u-u_l) \boxtimes \cV(u-v_1)^\vee \boxtimes \dotsb \boxtimes \cV(u-v_m)^\vee,
    \]
and there is a $\Zpi$-module isomorphism
    \begin{align} \label{maltese}
        V(l|m) &\xrightarrow{\cong} K_0 \left( \cV(u_1,\dotsc,u_l | v_1,\dotsc,v_m) \right).
    \end{align}
This is a
module supercategory over $\Heis{A}{-1} \odot \dotsb \odot
    \Heis{A}{-1} \odot \Heis{A}{1} \odot \dotsb \odot \Heis{A}{1}$.

Now assume in addition that $u_1,\dotsc,u_l,v_1,\dotsc,v_m$ are
all different elements of $\kk$.
Since $d > 0$, the set $\Gamma_R$ as defined prior to
\cref{genericity} is $\{0\}$,
taking $R = \KK$ to be the algebraic closure of $\kk$.
Then, interpreting ``eigenvalue'' as eigenvalue over $\KK$,
\cref{hot} shows that $x_i$ has only one eigenvalue $u_i$ on any
$\operatorname{Wr}_n^{u-u_i}(A)$-supermodule (resp.\
$x_j$ has only one eigenvalue $v_j$ on any $\operatorname{Wr}_m^{u-v_j}(A^\op)$-supermodule).
In particular, all of these eigenvalues actually belong to the ground field $\kk$.
Then we can argue as in \cref{bike}
to see that the action of
$\Heis{A}{-1} \odot \dotsb \odot
    \Heis{A}{-1} \odot \Heis{A}{1} \odot \dotsb \odot \Heis{A}{1}$
on
$\cV(u_1,\dotsc,u_l | v_1,\dotsc,v_m)$
extends to an action of the localization
      $\Heis{A}{-1} \barodotnonumber \dotsb \barodotnonumber \Heis{A}{-1} \barodotnonumber \Heis{A}{1} \barodotnonumber \dotsb \barodotnonumber \Heis{A}{1}$.
    Using the iterated comultiplication from \cref{bunny,coass3},
$\cV(u_1,\dotsc,u_l | v_1,\dotsc,v_m)$
    then becomes a module supercategory over $\Heis{A}{k}$.  Therefore, we have a strict monoidal superfunctor
    \[
      \Psi_{l|m} \colon \Heis{A}{k} \to \cEnd_\kk \left(\cV(u_1,\dotsc,u_l | v_1,\dotsc,v_m) \right),
    \]
which extends canonically to
a monoidal superfunctor $\tilde \Psi_{l|m}$ from $\Kar(\Heis{A}{k}_\pi)$.

Now we consider the diagram
    \begin{equation*}
        \begin{gathered}
            \xymatrixcolsep{3pc}
            \xymatrix{
\rHeis_k(A)\ar[r]^{\operatorname{ev}}\ar[d]_{\gamma_k}&                \rtHeis_k(A) \ar[r]^-{\psi_{l|m}} \ar[d]_{\tgamma_{k}} &
                \End_\Zpi \left( V(l|m) \right) \ar[d]^{c_{k}}
                \\
                K_0\left(\Kar\left(\Heis{A}{k}_{q,\pi}\right)\right) \ar[r]^-{\nu}                 &K_0\left(\Kar\left(\Heis{A}{k}_\pi\right)\right) \ar[r]^-{[\tilde\Psi_{l|m}(-)]} &
                \End_\Zpi \left( K_0 \left( \cV(u_1,\dotsc,u_l | v_1, \dotsc, v_m) \right) \right).
            }
        \end{gathered}
    \end{equation*}
In this diagram, the right hand map $c_k$
is the $\Zpi$-algebra isomorphism defined by
    conjugating with \cref{maltese}.
The top left map $\operatorname{ev}$ is the obvious surjective
$\Zq$-algebra homomorphism
defined by evaluation at $q=1$; it sends $s_\blambda^{\pm} \mapsto
\underline{s}_\blambda^{\pm}$.
The bottom left map $v$ is another $\Zq$-algebra homomorphism which is
induced by the obvious monoidal functor
$\Heis{A}{k}_{q,\pi} \rightarrow \Heis{A}{k}_\pi, Q^m \Pi^r
X\mapsto \Pi^r X, f_{m,r}^{n,s} \mapsto f_r^s$; it sends
$[S_\blambda^{\pm}] \mapsto [\underline{S}_\blambda^{\pm}]$.
The bottom right map is the
    $\Zpi$-algebra
    homomorphism $[X] \mapsto [\tilde\Psi_{l|m}(X)]$.  We also know
    already that
    $\psi_{l|m}$ is an injective $\Zpi$-algebra homomorphism.
The left hand square of the diagram obviously commutes.
The argument in the proof of \cite[Th.~7.2]{BSW-K0} shows that the
right hand square commutes too. Since $c_k \circ
\psi_{l|m}$ is injective, we deduce that the $\Zpi$-module
homomorphism $\tgamma_k$ is injective.
Finally we must show that $\tgamma_k$ is multiplicative. Take $a,b \in
\rtHeis_k(A)$.
Since both $\gamma_k$ and
$\operatorname{ev}$ are onto, it follows that the image of $\tgamma_k$
is equal to the image of $\nu$. Since $\nu$ is an algebra
homomorphism, it follows that the image of $\tgamma_k$ is a $\Zpi$-subalgebra
of $K_0(\Kar(\Heis{A}{k}_\pi))$.
Hence, we have that
$\tgamma_k(a) \tgamma_k(b) = \tgamma_k(c)$ for some $c \in
\rtHeis_k(A)$, and it remains to show that $c=ab$.
As all of the maps in the right hand square apart from $\tgamma_k$ are
already known to be algebra homomorphisms, we have that
\begin{align*}
c_k(\psi_{l|m}(ab)) &=
c_k(\psi_{l|m}(a)) c_k(\psi_{l|m}(b))=
[\tilde\Psi_{l|m}(\tgamma_k(a))] [\tilde\Psi_{l|m}(\tgamma_k(b))]\\
&=[\tilde\Psi_{l|m}(\tgamma_k(a) \tgamma_k(b))]
= [\tilde\Psi_{l|m}(\tgamma_k(c))] = c_k(\psi_{l|m}(c)).
\end{align*}
Since $c_k \circ \psi_{l|m}$ is injective, this implies that $ab = c$
as required.
\end{proof}

Now we return to the graded setting.
Recall that for objects $X,Y \in \Heis{A}{k}$ and homogeneous idempotents $e_X \colon X \to X$, $e_Y \colon Y \to Y$ we have
\[
  \Hom_{\Kar(\Heis{A}{k})} \big( (X,e_X), (Y,e_Y) \big) = e_Y \Hom_{\Heis{A}{k}}(X,Y) e_X
\]
by the definition of Karoubi envelope.  We will use a thick upward (resp.\ downward) string labeled $n$ to denote the identity morphism of $H_n^+$ (resp.\ $E_n^-$).  As in \cite[\S8]{BSW-K0}, we also introduce the diagrammatic shorthands
\begin{align*}
  \begin{tikzpicture}[anchorbase,>=to]
    \draw[->,line width=2pt] (0,-0.01) to (0,0.4) node[anchor=south] {\dotlabel{n,i}};
    \draw[line width=1pt] (-0.3,-0.4) node[anchor=north] {\dotlabel{(n-r),i}\ } to (-0.02,0);
    \draw[line width=1pt] (0.3,-0.4) node[anchor=north] {\ \dotlabel{r,i}} to (0.02,0);
  \end{tikzpicture}
  &:= \imath_n(e_{(n),i}) \colon H_{n-r,i}^+ \otimes H_{r,i}^+ \to H_{n,i}^+,
  &
  \begin{tikzpicture}[anchorbase,>=to]
    \draw[line width=2pt] (0,-0.4) node[anchor=north] {\dotlabel{n,i}} to (0,0.01);
    \draw[->,line width=1pt] (-0.02,0) to (-0.3,0.4) node[anchor=south] {\dotlabel{(n-r),i}\ };
    \draw[->,line width=1pt] (0.02,0) to (0.3,0.4) node[anchor=south] {\ \dotlabel{r,i}};
  \end{tikzpicture}
  &:= \binom{n}{r} \imath_n(e_{(n),i}) \colon H_{n,i}^+ \to H_{n-r,i}^+ \otimes H_{r,i}^+,
  \\
  \begin{tikzpicture}[anchorbase,>=to]
    \draw[->,line width=2pt] (0,0.01) to (0,-0.4) node[anchor=north] {\dotlabel{n,i}};
    \draw[line width=1pt] (-0.3,0.4) node[anchor=south] {\dotlabel{(n-r),i}\ } to (-0.02,0);
    \draw[line width=1pt] (0.3,0.4) node[anchor=south] {\ \dotlabel{r,i}} to (0.02,0);
  \end{tikzpicture}
  &:= \jmath_n(e_{(n),i}) \colon E_{n,i}^- \to E_{n-r,i}^- \otimes E_{r,i}^-,
  &
  \begin{tikzpicture}[anchorbase,>=to]
    \draw[line width=2pt] (0,0.4) node[anchor=south] {\dotlabel{n,i}} to (0,-0.01);
    \draw[->,line width=1pt] (-0.02,0) to (-0.3,-0.4) node[anchor=north] {\dotlabel{(n-r),i}\ };
    \draw[->,line width=1pt] (0.02,0) to (0.3,-0.4) node[anchor=north] {\ \dotlabel{r,i}};
  \end{tikzpicture}
  &:= \binom{n}{r} \jmath_n(e_{(n),i}) \colon E_{n-r,i}^- \otimes E_{r,i}^- \to E_{n,i}^-,
\end{align*}
for $0 \le r \le n$ and $1 \le i \le N$.  These merge and split morphisms are associative in an obvious sense allowing their definition to be extended to more strings, e.g.\ for three strings:
\[
  \begin{tikzpicture}[anchorbase,>=to]
  	\draw[-,thick] (0.36,-.3) to (0.09,0.14);
  	\draw[-,thick] (0.08,-.3) to (0.08,0.14);
  	\draw[-,thick] (-0.2,-.3) to (0.07,0.14);
  	\draw[<-,line width=2pt] (0.08,.45) to (0.08,.1);
  \end{tikzpicture}
  :=
  \begin{tikzpicture}[anchorbase,>=to]
  	\draw[-,thick] (0.35,-.3) to (0.08,0.14);
  	\draw[-,thick] (0.1,-.3) to (-0.04,-0.06);
  	\draw[-,line width=1pt] (0.085,.14) to (-0.035,-0.06);
  	\draw[-,thick] (-0.2,-.3) to (0.07,0.14);
  	\draw[<-,line width=2pt] (0.08,.45) to (0.08,.1);
  \end{tikzpicture}
  =
  \begin{tikzpicture}[anchorbase,>=to]
  	\draw[-,thick] (0.36,-.3) to (0.09,0.14);
  	\draw[-,thick] (0.06,-.3) to (0.2,-.05);
  	\draw[-,line width=1pt] (0.07,.14) to (0.19,-.06);
  	\draw[-,thick] (-0.19,-.3) to (0.08,0.14);
  	\draw[<-,line width=2pt] (0.08,.45) to (0.08,.1);
  \end{tikzpicture}
  ,\quad
  \begin{tikzpicture}[anchorbase,>=to]
  	\draw[<-,thick] (0.36,.3) to (0.09,-0.14);
  	\draw[<-,thick] (0.08,.3) to (0.08,-0.14);
  	\draw[<-,thick] (-0.2,.3) to (0.07,-0.14);
  	\draw[-,line width=2pt] (0.08,-.45) to (0.08,-.1);
  \end{tikzpicture}
  :=
  \begin{tikzpicture}[anchorbase,>=to]
  	\draw[<-,thick] (0.35,.3) to (0.08,-0.14);
  	\draw[<-,thick] (0.1,.3) to (-0.04,0.06);
  	\draw[-,line width=1pt] (0.085,-.14) to (-0.035,0.06);
  	\draw[<-,thick] (-0.2,.3) to (0.07,-0.14);
  	\draw[-,line width=2pt] (0.08,-.45) to (0.08,-.1);
  \end{tikzpicture}
  =
  \begin{tikzpicture}[anchorbase,>=to]
  	\draw[<-,thick] (0.36,.3) to (0.09,-0.14);
  	\draw[<-,thick] (0.06,.3) to (0.2,.05);
  	\draw[-,line width=1pt] (0.07,-.14) to (0.19,.06);
  	\draw[<-,thick] (-0.19,.3) to (0.08,-0.14);
  	\draw[-,line width=2pt] (0.08,-.45) to (0.08,-.1);
  \end{tikzpicture}
  ,\quad
  \begin{tikzpicture}[anchorbase,>=to]
  	\draw[-,thick] (0.36,.3) to (0.09,-0.14);
  	\draw[-,thick] (0.08,.3) to (0.08,-0.14);
  	\draw[-,thick] (-0.2,.3) to (0.07,-0.14);
  	\draw[<-,line width=2pt] (0.08,-.45) to (0.08,-.1);
  \end{tikzpicture}
  :=
  \begin{tikzpicture}[anchorbase,>=to]
  	\draw[-,thick] (0.35,.3) to (0.08,-0.14);
  	\draw[-,thick] (0.1,.3) to (-0.04,0.06);
  	\draw[-,line width=1pt] (0.085,-.14) to (-0.035,0.06);
  	\draw[-,thick] (-0.2,.3) to (0.07,-0.14);
  	\draw[<-,line width=2pt] (0.08,-.45) to (0.08,-.1);
  \end{tikzpicture}
  =
  \begin{tikzpicture}[anchorbase,>=to]
  	\draw[-,thick] (0.36,.3) to (0.09,-0.14);
  	\draw[-,thick] (0.06,.3) to (0.2,.05);
  	\draw[-,line width=1pt] (0.07,-.14) to (0.19,.06);
  	\draw[-,thick] (-0.19,.3) to (0.08,-0.14);
  	\draw[<-,line width=2pt] (0.08,-.45) to (0.08,-.1);
  \end{tikzpicture}
  ,\quad
  \begin{tikzpicture}[anchorbase,>=to]
  	\draw[<-,thick] (0.36,-.3) to (0.09,0.14);
  	\draw[<-,thick] (0.08,-.3) to (0.08,0.14);
  	\draw[<-,thick] (-0.2,-.3) to (0.07,0.14);
  	\draw[-,line width=2pt] (0.08,.45) to (0.08,.1);
  \end{tikzpicture}
  :=
  \begin{tikzpicture}[anchorbase,>=to]
  	\draw[<-,thick] (0.35,-.3) to (0.08,0.14);
  	\draw[<-,thick] (0.1,-.3) to (-0.04,-0.06);
  	\draw[-,line width=1pt] (0.085,.14) to (-0.035,-0.06);
  	\draw[<-,thick] (-0.2,-.3) to (0.07,0.14);
  	\draw[-,line width=2pt] (0.08,.45) to (0.08,.1);
  \end{tikzpicture}
  =
  \begin{tikzpicture}[anchorbase,>=to]
  	\draw[<-,thick] (0.36,-.3) to (0.09,0.14);
  	\draw[<-,thick] (0.06,-.3) to (0.2,-.05);
  	\draw[-,line width=1pt] (0.07,.14) to (0.19,-.06);
  	\draw[<-,thick] (-0.19,-.3) to (0.08,0.14);
  	\draw[-,line width=2pt] (0.08,.45) to (0.08,.1);
  \end{tikzpicture}
  .
\]
We define thick crossings recursively by
\[
  \begin{tikzpicture}[baseline={(0,-0.05)},>=to]
    \draw[->,line width=2pt] (-0.5,-0.5) node[anchor=north] {\dotlabel{n,i}} to (0.5,0.5);
    \draw[->,line width=2pt] (0.5,-0.5) node[anchor=north] {\dotlabel{m,j}} to (-0.5,0.5);
  \end{tikzpicture}
  := \binom{m}{r}^{-1}
  \begin{tikzpicture}[baseline={(0,-0.05)},>=to]
    \draw[->,line width=2pt] (-0.5,-0.5) node[anchor=north] {\dotlabel{n,i}} to (0.5,0.5);
    \draw[line width=2pt] (0.5,-0.5) node[anchor=north] {\dotlabel{m,j}} to (0.29,-0.29);
    \draw[->,line width=2pt] (-0.29,0.29) to (-0.5,0.5);
    \draw[line width=1pt] (0.3,-0.3) to[out=100,in=-10] (-0.3,0.3);
    \draw[line width=1pt] (0.3,-0.3) to[out=170,in=-80] (-0.3,0.3);
    \node at (-0.08,0.4) {\dotlabel{r,j}};
    \node at (-0.7,0.1) {\dotlabel{m-r,j}};
  \end{tikzpicture}
  = \binom{n}{s}^{-1}
  \begin{tikzpicture}[baseline={(0,-0.05)},>=to]
    \draw[->,line width=2pt] (0.5,-0.5) node[anchor=north] {\dotlabel{m,j}} to (-0.5,0.5);
    \draw[line width=2pt] (-0.5,-0.5) node[anchor=north] {\dotlabel{n,i}} to (-0.29,-0.29);
    \draw[->,line width=2pt] (0.29,0.29) to (0.5,0.5);
    \draw[line width=1pt] (-0.3,-0.3) to[out=80,in=190] (0.3,0.3);
    \draw[line width=1pt] (-0.3,-0.3) to[out=10,in=-100] (0.3,0.3);
    \node at (0.08,0.4) {\dotlabel{s,i}};
    \node at (0.66,0.08) {\dotlabel{n-s,i}};
  \end{tikzpicture}.
\]
Similarly, there are thick downward, rightward, and leftward crossings.  Recalling the definition \cref{Panti}, we also defined decorated thick left caps by
\[
  \begin{tikzpicture}[baseline={(0,-0.05)},>=to]
    \draw[line width=2pt,->] (-0.4,-0.4) node[anchor=north] {\dotlabel{n,i}} to (-0.4,0) arc(180:0:0.4) to (0.4,-0.4) node[anchor=north] {\dotlabel{n,j}};
    \multblackdot[north]{0,0.4}{f};
  \end{tikzpicture}
  \ :=\
  \begin{tikzpicture}[anchorbase,>=to]
    \draw[line width=2pt] (-0.8,-1) node[anchor=north] {\dotlabel{n,i}} to (-0.8,-0.6);
    \draw (-0.83,-0.61) \braidto (-1.2,0) arc (180:0:1.2) to[out=down,in=up] (0.83,-0.61);
    \draw (-0.815,-0.61) \braidto (-1,0) arc (180:0:1) to[out=down,in=up] (0.815,-0.61);
    \draw (-0.8,-0.61) \braidto (-0.8,0) arc (180:0:0.8) to[out=down,in=up] (0.8,-0.61);
    \draw (-0.785,-0.61) \braidto (-0.6,0) arc (180:0:0.6) to[out=down,in=up] (0.785,-0.61);
    \draw (-0.77,-0.61) \braidto (-0.4,0) arc (180:0:0.4) to[out=down,in=up] (0.77,-0.61);
    \draw[<-,line width=2pt] (0.81,-1) node[anchor=north] {\dotlabel{n,j}} to (0.8,-0.6);
    \draw[rounded corners,fill=white] (-1.3,-0.2) rectangle (-0.3,0.2);
    \node at (-0.8,0) {\dotlabel{\imath_n(f)}};
  \end{tikzpicture}
\]
for $f \in P_n(e_j A e_i)^{\fS_n\anti}$. (We are essentially working
here in the colored Heisenberg category discussed in \cref{color}.)

\begin{lem} \label{earl}
  For $k \ge 0$, $1 \le i,j \le N$, and $m,n > 0$, we have
  \[
    \begin{tikzpicture}[anchorbase,>=to]
      \draw[->,line width=2pt] (-0.3,-0.7) node[anchor=north] {\dotlabel{m,i}} to[out=60,in=down] (0.3,0) to[out=up,in=-60] (-0.3,0.7);
      \draw[<-,line width=2pt] (0.3,-0.7) to[out=120,in=down] (-0.3,0) to[out=up,in=-120] (0.3,0.7) node[anchor=south] {\dotlabel{n,j}};
    \end{tikzpicture}
    =
    \begin{tikzpicture}[anchorbase,>=to]
      \draw[->,line width=2pt] (-0.3,-0.7) node[anchor=north] {\dotlabel{m,i}} to (-0.3,0.7);
      \draw[<-,line width=2pt] (0.3,-0.7) to (0.3,0.7) node[anchor=south] {\dotlabel{n,j}};
    \end{tikzpicture}
    + \sum_{r=0}^{k-1} \sum_{b \in \B_A}
    \begin{tikzpicture}[anchorbase,>=to]
      \draw[line width=2pt] (-0.35,-0.7) node[anchor=north] {\dotlabel{m,i}} to (-0.35,-0.5);
      \draw[line width=1.6pt] (-0.356,-0.52) to (-0.356,0.2);
      \draw[<-,line width=2pt] (0.35,-0.7) node[anchor=north] {\dotlabel{n,j}} to (0.35,-0.5);
      \draw[line width=1.6pt] (0.356,-0.52) to (0.356,0.2);
      \draw (-0.32,-0.52) arc(180:0:0.32);
      \draw[->,line width=2pt] (-0.35,0.2) to (-0.35,0.7) node[anchor=south] {\dotlabel{m,i}};
      \draw[line width=2pt] (0.35,0.2) to (0.35,0.7) node[anchor=south] {\dotlabel{n,j}};
      \draw[fill=lightgray] (-0.5,0.05) rectangle (0.5,0.35);
      \multblackdot[north]{-0.135,-0.23}{r};
      \blacktoken[north]{0.135,-0.23}{b};
    \end{tikzpicture}
  \]
  where the shaded rectangle indicates a morphism that will not be determined precisely.
\end{lem}

\begin{proof}
  The proof is almost identical to that of \cite[Lem.~8.1]{BSW-K0} and so will be omitted here.
\end{proof}

\begin{cor} \label{grey}
  For $k \ge 0$ and $m,n > 0$, we have
  \[
    \begin{tikzpicture}[anchorbase,>=to]
      \draw[->,line width=2pt] (-0.3,-0.8) node[anchor=north] {\dotlabel{m,i}} to (-0.3,0.8);
      \draw[<-,line width=2pt] (0.3,-0.8) to (0.3,0.8) node[anchor=south] {\dotlabel{n,j}};
    \end{tikzpicture}
    =
    \sum_{r=0}^{\min(m,n)}
    \sum_{f \in \B^{(r)}_{j,i}}
    \begin{tikzpicture}[anchorbase,>=to]
      \draw[line width=2pt] (-0.35,-0.8) node[anchor=north] {\dotlabel{m,i}} to (-0.35,-0.6);
      \draw[line width=1pt] (-0.367,-0.62) to (-0.356,-0.4) node[anchor=south east] {\dotlabel{m-r}} \braidto (0.367,0.3);
      \draw[<-,line width=2pt] (0.35,-0.8) node[anchor=north] {\dotlabel{n,j}} to (0.35,-0.6);
      \draw[line width=1pt] (0.367,-0.62) to (0.356,-0.4) \braidto (-0.367,0.3);
      \draw[line width=1pt] (-0.333,-0.62) arc(180:0:0.333);
      \draw[->,line width=2pt] (-0.35,0.3) to (-0.35,0.8) node[anchor=south] {\dotlabel{m,i}};
      \draw[line width=2pt] (0.35,0.3) to (0.35,0.8) node[anchor=south] {\dotlabel{n,j}};
      \draw[fill=lightgray] (-0.5,0.15) rectangle (0.5,0.45);
      \multblackdot[north]{0,-0.287}{f};
    \end{tikzpicture}
  \]
  where $\B^{(r)}_{j,i}$ is a basis of the subspace of $P_r(e_j A e_i)^{\fS_r\anti}$ spanned by elements of polynomial degree $\le k-1$.
\end{cor}

\begin{proof}
    Rearranging the identity from \cref{earl} gives the $r=0$ term in the sum exactly.  Then we use induction on $\min(m,n)$, together with \cref{dog} to get the other terms.
\end{proof}

\begin{theo} \label{K0isom}
    For any graded symmetric Frobenius superalgebra $A$ satisfying
    \cref{virginia}, the map \eqref{gamma} is an isomorphism of
    $\Zq$-algebras. Moreover,
in $\Kar(\Heis{A}{k}_{q,\pi})$, there are distinguished isomorphisms
\begin{align}
    H_{m,i}^+ \otimes H_{n,j}^+ &\cong H_{n,j}^+ \otimes H_{m,i}^+,\qquad\qquad
    E_{m,i}^- \otimes E_{n,j}^- \cong E_{n,j}^- \otimes E_{m,i}^-,\label{last1}
    \\
    H_{m,i}^+ \otimes E_{n,j}^- &\cong \bigoplus_{r=0}^{\min(m,n)} \extdim{r}{[k]_d \langle i,j \rangle_A} E_{n-r,j}^- \otimes H_{m-r,i}^+ \quad \text{if } k \ge 0,\label{last2}
    \\
    E_{m,i}^- \otimes H_{n,j}^+ &\cong \bigoplus_{r=0}^{\min(m,n)} \extdim{r}{-[k]_d \langle j,i \rangle_A} H_{n-r,j}^+ \otimes E_{m-r,i}^- \quad \text{if } k \le 0,\label{last3}
  \end{align}
  for all $m,n >0$ and $i,j \in \{1,2,\dotsc,N\}$.
(Here, for $z = \sum_{n \in\Z, p \in \Z/2} z_{n,p} q^n \pi^p \in
\Zq$ with all $z_{n,p} \geq 0$, we are using the notation $z V$ to
denote
$\bigoplus_{n \in \Z, p \in \Z/2} Q^n \Pi^p V^{\oplus z_{n,p}}$.)
\end{theo}

\begin{proof}
We first prove the result in the case that $d >0$.
In this case, \cref{coffee} shows that $\gamma_k$ is a $\Zq$-module
isomorphism, and it just remains to show that it is an algebra
homomorphism. To see this, it suffices to verify \cref{last1,last2,last3},
since these isomorphisms imply that the images of the
relations from \cref{marahau} are satisfied in
$K_0(\Kar(\Heis{A}{k}_{q,\pi}))$.
To establish \cref{last1,last2,last3}, we just treat the case $k \ge 0$,
since the result for $k \le 0$ then follows by applying $\Omega_k$ and
using \cref{walk}.
The thick upward (resp.\ downward) crossing gives a canonical
isomorphism $H_{m,i}^+ \otimes H_{n,j}^+ \cong H_{n,j}^+ \otimes
H_{m,i}^+$ (resp.\ $E_{m,i}^- \otimes E_{n,j}^- \cong E_{n,j}^-
\otimes E_{m,i}^-$),
verifying \cref{last1}.  It remains to show that the objects
    \[
        X := H_{m,i}^+ E_{n,j}^-
        \quad \text{and} \quad
        Y := \bigoplus_{r=0}^{\min(m,n)} \extdim{r}{[k]_d \langle i,j \rangle_A} E_{n-r,j}^- \otimes H_{m-r,i}^+
    \]
    are isomorphic.  \Cref{grey} implies that the morphism $\theta_{m,n} \colon X \to Y$ defined by the column vector
    \[
        \left[
            \begin{tikzpicture}[anchorbase,>=to]
                \draw[line width=2pt] (-0.35,-0.8) node[anchor=north] {\dotlabel{m,i}} to (-0.35,-0.6);
                \draw[line width=1pt,->] (-0.367,-0.62) to (-0.356,-0.4) \braidto (0.367,0.3) node[anchor=south] {\dotlabel{m-r,i}};
                \draw[<-,line width=2pt] (0.35,-0.8) node[anchor=north] {\dotlabel{n,j}} to (0.35,-0.6);
                \draw[line width=1pt] (0.367,-0.62) to (0.356,-0.4) \braidto (-0.367,0.3) node[anchor=south] {\dotlabel{n-r,j}};
                \draw[line width=1pt] (-0.333,-0.62) arc(180:0:0.333);
                \multblackdot[north]{0,-0.287}{f};
            \end{tikzpicture}
        \right]_{0 \le r \le \min(m,n),\, f \in \B^{(r)}_{j,i}}
    \]
    has a left inverse $\phi_{m,n}$.  Using \cref{yoneda} to translate
    into the language of projective graded
supermodules over the locally unital graded superalgebra
corresponding to $\Heis{A}{k}$, we have finitely generated projective
graded supermodules $X$, $Y$ and homomorphisms $\theta_{m,n} \colon Y
\to X$ and $\phi_{m,n} \colon X \to Y$ such that $\theta_{m,n} \circ
\phi_{m,n} = \id_X$, and we wish to show that $\theta_{m,n}$ is an
isomorphism. Since $\theta_{m,n}$ has
a right inverse, it is surjective, and it remains to show that
$Z := \ker \theta_{m,n}$ is zero.
By \cref{coffee}, we have that
\[
Z \cong \bigoplus_{\blambda,\bmu \in \cP^N}
z_{\blambda,\bmu}(q) S_{\bmu}^- \otimes S_{\blambda}^+
\]
for $z_{\blambda,\bmu}(q) \in \Zq$ with nonnegative coefficients.
Since $Y$ is projective, we have
$Y \cong X \oplus Z$,
hence $\ug{Y} \cong \ug{X} \oplus \ug{Z}$
where $\ug{X}, \ug{Y}$ and $\ug{Z}$ denote
the projective supermodules obtained from $X, Y$ and $Z$ by
forgetting the grading.
By \cref{shadow} specialized at $q=1$ and \cref{marahau}, we know that $[\ug{Y}]$ and $[\ug{X}]$ are
equal in $K_0(\Kar(\Heis{A}{k}_\pi))$, hence
\[
[\ug{Z}] = \sum_{\blambda,\bmu \in \cP^N} z_{\blambda,\bmu}(1)
[\underline{S}_\bmu^- \otimes \underline{S}_\blambda^+]
= 0.
\]
By \cref{shadow}, the classes of the objects
$\underline{S}_\bmu^- \otimes \underline{S}_, \:\blambda^+\, \blambda,\bmu
\in \cP^N$
are linearly independent in the $\Zpi$-module $K_0(\Heis{A}{k}_\pi)$.
We deduce that $z_{\blambda,\bmu}(1) = 0$, hence since the
coefficients of $z_{\blambda,\bmu}(q)$ are all nonnegative we
actually have that $z_{\blambda,\bmu}(q) = 0$. This implies that $Z =
0$ as required.

Now suppose $d=0$, i.e.\ the grading on $A$ is trivial.  Then, by our
assumption \cref{virginia}, $A$ is purely even and semisimple.  Using
\cref{moritathm}, we may therefore assume that $A = \kk^{\oplus n}$
for some $n \ge 1$.  This case was considered in \cite{Gan18}, where
it was shown that $\Kar(\Heis{\kk^{\oplus n}}{k})$ is isomorphic to
the Karoubi envelope of the symmetric product of $n$ copies of
$\Heis{\kk}{k}$. It was shown in \cite[Th.~1.1]{BSW-K0} that
$K_0(\Kar(\Heis{\kk}{k}))$ is isomorphic to
$\rtHeis_k(\kk)$ specialized at $\pi = 1$ (there being no reason to
consider any $\Z$- or $\Z/2$-gradings).
Reintroducing the trivial gradings, it follows that
    \[
        \Kar\left(\Heis{\kk^{\oplus n}}{k}_{q,\pi}\right) \cong \rHeis_k(\kk)^{\otimes n} \cong \rHeis_k(\kk^{\oplus n}),
    \]
    where the second isomorphism follows from the
    definition of $\rHeis_k(\kk^{\oplus n})$, or from
    \cref{marahau}. The first part of the theorem follows. For the second part, the same morphisms as defined above give the isomorphisms \cref{last1,last2,last3} when $d=0$ too, as follows by similar reductions from
        \cite[Th.~1.2]{BSW-K0} when $i=j$, the result being obvious when $i
    \neq j$.
\end{proof}

For $A = \kk$, \cref{K0isom} was conjectured in \cite[Conj.~1]{Kho14} and proved in \cite[Th.~1.1]{BSW-K0}.  When $\dA > 0$, \Cref{K0isom} was proved in \cite[Th.~1.5]{Sav18} under an additional assumption that one has imposed certain additional relations on the bubbles.  As noted after \cite[Th.~1.5]{Sav18}, the basis theorem (\cref{basis}) allows us to prove the result without these additional relations.

Let us finally prove that the functor $\tilde\Delta_{l|m}$ from
\cref{bunnycor} categorifies the comultiplication $\delta_{l|m}$.

\begin{theo} \label{camping}
  Suppose $A$ is a graded symmetric Frobenius superalgebra $A$ satisfying \cref{virginia}.  For $l,m \in \Z$ with $k=l+m$, we have a commutative diagram
  \[
    \begin{tikzcd}
      \rHeis_k(A) \arrow[r, "\delta_{l|m}"] \arrow[dd, "\gamma_{k}" right, "\cong" left]
      &
      \rHeis_l(A) \otimes_{\Zq} \rHeis_m(A) \arrow[d, "\gamma_{l} \otimes \gamma_{m}" right, "\cong" left]
      \\
      & K_0\left(\Kar\left(\Heis{A}{l}_{q,\pi}\right)\right) \otimes_{\Zq} K_0\left(\Kar\left(\Heis{A}{m}_{q,\pi}\right)\right) \arrow[d, "\epsilon_{l|m}" right]
      \\
      K_0\left(\Kar\left(\Heis{A}{k}_{q,\pi}\right)\right) \arrow[r, "{[\tilde\Delta_{l|m}]}"]
      &
      K_0\left(\Kar\left(\big(\Heis{A}{l} \barodot \Heis{A}{m}\big)_{q,\pi}\right)\right)
    \end{tikzcd}
  \]
  where $\epsilon_{l|m}$ is the ring homomorphism induced by the canonical functors from $\Heis{A}{l}$ and $\Heis{A}{m}$ to $\Heis{A}{l} \barodot \Heis{A}{m}$.
\end{theo}

\begin{proof}
  Since all the maps in the diagram are ring homomorphisms, it suffices to show that the diagram commutes on the generators $h_{n,i}^+$ and $e_{n,i}^-$ of $\rHeis_k(A)$.  This follows from \cref{thedrop,puppet}.
\end{proof}

\begin{rem}\label{Dagger}
    The hypothesis \eqref{virginia} assumed throughout this section and \cref{sec:lattic} can actually be weakened slightly: our arguments only require that all irreducible $A$-modules are of type $\mathtt{M}$, and that \emph{either} $A$ is semisimple with trivial grading \emph{or} $A$ is nontrivially positively graded. The assumption \eqref{virginia} as  formulated in the introduction is slightly stronger than this but actually it holds in all of the examples of interest to us.
\end{rem}

\begin{rem}\label{park}
    We have assumed throughout the article that the ground ring $\kk$ is a field of characteristic zero.  The category $\Heis{A}{k}$ can actually be defined for any symmetric Frobenius superalgebra $A$ over any commutative ground ring $\kk$; such an $A$ comes equipped with a dual pair of homogeneous bases as a free $\kk$-supermodule. Providing $\kk$ is an integral domain, \Cref{basis} continues to hold in this more general setting.  However, the computation of the Grothendieck ring in the current section uses the characteristic zero assumption in a more fundamental way, since the idempotents $e_{(n),i}$ involve division by $n!$.  To work in positive characteristic, one would need to develop a ``thick calculus'' for the Frobenius Heisenberg category.
\end{rem}


\bibliographystyle{alphaurl}
\bibliography{Foundations}

\end{document}